\documentclass{amsart}
\usepackage{amsmath, amsthm, amsfonts, amssymb, color}
 \usepackage{mathrsfs}
\usepackage{amsfonts}
 \usepackage{amstext,amsthm,amsxtra}
 \usepackage{txfonts} 
 \usepackage[colorlinks, citecolor=blue,pagebackref,hypertexnames=false]{hyperref}
 \allowdisplaybreaks

\usepackage{mathrsfs}
\usepackage{enumerate}

\numberwithin{equation}{section}

\newtheorem{thm}{Theorem}[section]
\newtheorem{prop}[thm]{Proposition}
\newtheorem{cor}[thm]{Corollary}
\newtheorem{defi}[thm]{Definition}
\newtheorem{lem}[thm]{Lemma}

\theoremstyle{remark}

\newtheorem{remark}[thm]{Remark}

\newcommand{\Heis}{{\mathbb{H}}}

\title{}

\begin{document}

 \baselineskip 16.6pt
\hfuzz=6pt

\newtheorem{cl}{Claim}

\newtheorem{assum}{Assumption}[section]
\newtheorem{ex}[thm]{Example}
\newtheorem{rem}[thm]{Remark}
\renewcommand{\theequation}
{\thesection.\arabic{equation}}

\def\SL{\sqrt H}
\newcommand{\cent}{\operatorname{cent}}

\newcommand{\wid}{\operatorname{width}}
\newcommand{\heit}{\operatorname{height}}

\newcommand{\cdim}{n}

\newcommand{\set}[1]{\mathfrak{#1}}
\newcommand{\vrect}{\set{V}}
\newcommand{\Stack}{\set{S}}
\newcommand{\tile}{\set{T}}

\newcommand{\mar}[1]{{\marginpar{\sffamily{\scriptsize
        #1}}}}

\newcommand{\as}[1]{{\mar{AS:#1}}}
\newcommand\C{\mathbb{C}}
\newcommand\Z{\mathbb{Z}}
\newcommand\R{\mathbb{R}}
\newcommand\RR{\mathbb{R}}
\newcommand\CC{\mathbb{C}}
\newcommand\NN{\mathbb{N}}
\newcommand\ZZ{\mathbb{Z}}
\newcommand\HH{\mathbb{H}}
\def\RN {\mathbb{R}^n}
\renewcommand\Re{\operatorname{Re}}
\renewcommand\Im{\operatorname{Im}}

\newcommand{\mc}{\mathcal}
\newcommand\D{\mathcal{D}}
\def\hs{\hspace{0.33cm}}
\newcommand{\la}{\alpha}
\def \l {\alpha}
\newcommand{\eps}{\tau}
\newcommand{\pl}{\partial}
\newcommand{\supp}{{\rm supp}{\hspace{.05cm}}}
\newcommand{\x}{\times}
\newcommand{\lag}{\langle}
\newcommand{\rag}{\rangle}

\newcommand{\Bc}{\mathcal{B}}
\newcommand{\sgn}{\operatorname{sgn}}
\newcommand{\red}{{\mathrm{red}}}

\newcommand{\lset}{\left\lbrace}
\newcommand{\rset}{\right\rbrace}

\newcommand\wrt{\,{\rm d}}

\title[]{Endpoint weak Schatten class estimates and trace formula for commutators\\
 of Riesz transforms with multipliers on Heisenberg groups}

\author{Zhijie Fan}
\address{Zhijie Fan, School of Mathematics and Statistics, Wuhan University, Wuhan 430072, China}
\email{ZhijieFan@whu.edu.cn}

\author{Ji Li}
\address{Ji Li, Department of Mathematics, Macquarie University, Sydney}
\email{ji.li@mq.edu.au}

\author{Edward McDonald}
\address{Edward McDonald, School of Mathematics and Statistics, UNSW, Kensington, NSW 2052, Australia}
\email{edward.mcdonald@unsw.edu.au}

\author{Fedor Sukochev}
\address{Fedor Sukochev, School of Mathematics and Statistics, UNSW, Kensington, NSW 2052, Australia}
\email{f.sukochev@unsw.edu.au}

\author{Dmitriy Zanin}
\address{Dmitriy Zanin, School of Mathematics and Statistics, UNSW, Kensington, NSW 2052, Australia}
\email{d.zanin@unsw.edu.au}

  \date{\today}


\begin{abstract}
Along the line of singular value estimates for commutators by Rochberg--Semmes, Lord--McDonald--Sukochev--Zanin and Fan--Lacey--Li,  we establish the endpoint weak Schatten class estimate for commutators
 of Riesz transforms with multiplication operator $M_f$ on Heisenberg groups via homogeneous Sobolev norm of the symbol $f$. The new tool we exploit is the construction of a singular trace formula on Heisenberg groups, which, together with the use of double operator integrals, allows us to bypass the use of Fourier analysis and provides a solid foundation to investigate the singular values estimates for similar commutators in general stratified Lie groups.
\end{abstract}

\maketitle

\section{Introduction and statement of main results}

This paper concerns the classical problems in Harmonic Analysis on non-abelian Heisenberg groups where the use of Fourier analysis is limited. The methods which we employ here are inspired by techniques imported from noncommutative geometry.

\subsection{Background}\label{S1.1}
Originated from Nehari \cite{Ne} and Calder\'on \cite{Cal}, the commutator $[A, M_f] $,  defined by $[A, M_f] :=AM_f- M_f A$ for a singular integral operator $A$ and a multiplication operator $M_f$  with symbol $f$, plays an important role in harmonic analysis. Moreover, it has
tight connections and applications to complex analysis, non-commutative analysis and operator theory, see for example \cite{CLMS,CRW,HLW,Hy,TaylorPAMS2015}.

As illustration, let  $H$ be the Hilbert transform on $\mathbb R$ and $\mathcal{L}_{p}$ be the Schatten class (precise definition will be provided in Section \ref{Spdef}). Peller \cite{P} showed that the commutator $[H,M_f]$ is in the Schatten class $\mathcal{L}_p $, $0<p<\infty$, if and only if the symbol $f$ is in the Besov space ${\rm B}_{p,p}^{1\over p}(\mathbb R)$. A higher dimensional analogy of Peller's result was obtained by Janson--Wolff \cite{JW} 
concerning the $\ell$-th Riesz transform $R_\ell$ on $\mathbb R^n$, $n\geq2$.
They proved that: for $n<p<\infty$, $[R_\ell,M_f]\in \mathcal{L}_p $ {\rm\ if\ and\ only\ if\ } $f\in {\rm B}_{p,p}^{n\over p}(\mathbb R^n)$; for $0<p\leq n$, $[R_\ell,M_f]\in \mathcal{L}_p $ {\rm\ if\ and\ only\ if\ } $f$ is a constant. Unlike in the one-dimensional case, at $p=n$ there is a ``cut-off" in the sense that the function space (for which the equivalence holds) collapses to constant.  Rochberg--Semmes \cite{RS,RS1988} further studied  the Schatten--Lorentz class $\mathcal{L}_{p,q}$ estimates for $[R_\ell,M_f]$ and investigated the endpoint estimate for  $[R_\ell,M_f]\in\mathcal{L}_{n,\infty}$. In a paper which links these commutators to noncommutative geometry,
Connes--Sullivan--Teleman \cite{CST} further announced the full characterisation of the endpoint for $[R_\ell,M_f]\in\mathcal{L}_{n,\infty}$ via a homogeneous Sobolev space $\dot W^{1,n}(\mathbb R^n)$.

Recently, Lord--McDonald--Sukochev--Zanin \cite{LMSZ} provided a new proof for Connes--Sullivan--Teleman \cite{CST}
by proving that
quantised derivatives ${d\hspace*{-0.08em}\bar{}\hspace*{0.1em}}f$ of Alain Connes (introduced in \cite[IV]{Connes} on $\mathbb{R}^n,$ $n>1$) 
is in $\mathcal{L}_{n,\infty}$ if and only if $f$ is in the homogeneous Sobolev space $\dot {W}^{1,n}(\mathbb R^n)$, since ${d\hspace*{-0.08em}\bar{}\hspace*{0.1em}}f$ is linked to $[R_\ell,M_f]$.
The main method in \cite{LMSZ} is via creating a formula for Dixmier trace of $|{d\hspace*{-0.08em}\bar{}\hspace*{0.1em}}f|^n$ and using double operator integrals, which made extensive use of pseudodifferential calculus and the Fourier transform -- in this way the proof was dependent on Fourier theory. Very recently, a different approach to a similar problem was announced 
by Frank \cite{Frank2022}.



Another important development of commutator $[A,M_f]$ is along the direction of Heisenberg groups $\mathbb H^n$ 
 where the Fourier analysis is not as effective as that in $\mathbb R^n$. For the background of Heisenberg groups we refer to Folland--Stein \cite{FoSt}.
To be more explicit, let $\mathbb{H}^{n}$ be a Heisenberg group, which is a nilpotent Lie group with underlying manifold $\mathbb{C}^{n}\times \mathbb{R}=\{[z,t]:z=(z_{1},\cdots,z_{n})\in\mathbb{C}^{n}, t\in \mathbb{R}\}$, anisotropic dilation $\delta_r ([ z, t] ) = [rz_1, ..., rz_n, r^2 t]$, the Kor\'anyi metric $ d_K([z,t]) = (|z|^4+|t|^2)^{1\over 4}$, and the multiplication law
\begin{align*}
[z,t][z^{\prime},t^{\prime}]=[z_{1},\cdots,z_{n},t][z_{1}^{\prime},\cdots,z_{n}^{\prime},t^{\prime}]:=\Big[z_{1}+z_{1}^{\prime},\cdots,z_{n}+z_{n}^{\prime},t+t^{\prime}+{\rm Im}\Big(\sum_{j=1}^{n}z_{j}\overline{z'}_{j}\Big)\Big].
\end{align*}
The $2n+1$ vector fields
$$X_{\ell}:=\frac{\partial}{\partial x_{\ell}}-y_{\ell}\frac{\partial}{\partial t},\ \ Y_{\ell}:=\frac{\partial}{\partial y_{\ell}}+x_{\ell}\frac{\partial}{\partial t},\ \  T:=\frac{\partial}{\partial t},\ \ \ell=1,2,\cdots,n$$
form a natural basic for the Lie algebra of left-invariant vector field on $\mathbb{H}^{n}$. For convenience, we set $X_{n+\ell}:=Y_\ell, \ell=1,2,\cdots,n$ and set $X_{2n+1}:=T$. The standard sub-Laplacian $\Delta_{\mathbb{H}}$ on the Heisenberg group is defined by $\Delta_{\mathbb{H}}:=\sum_{\ell=1}^{2n}X_{\ell}^{2}.$
%
%
For any $\ell=1,2,\ldots,2n$,   the Riesz transform $R_{\ell}$ on Heisenberg groups $\mathbb{H}^{n}$ is defined by
$ R_{\ell} = X_\ell (-\Delta_{\mathbb H})^{-{1\over2}} $,
 which is
%
%
a well-known  Calder\'on--Zygmund operator on $\mathbb H^n$. It lies at the centre of harmonic analysis on $\mathbb H^n$, which connects to harmonic function theory, and the characterisation of one-parameter and multi-parameter Hardy and BMO spaces   \cite{CG,FoSt,  CCLLO}.
The characterisations of the boundedness and compactness of $[R_\ell,M_f]$  were obtained  recently in \cite{DLLW} and \cite{CDLW}, respectively.

Just recently, by launching a new method to  bypass the Fourier analysis, Fan--Lacey--Li \cite{FLLarxiv} obtained the Schatten class characterisation of $[R_\ell, M_f]$ on $\mathbb H^n$, which is analogous to Janson--Wolff \cite{JW} in $\mathbb R^n$. That is: for $\ell=1,2,\ldots,2n$,

(1) when $2n+2<p<\infty$, $[R_\ell,M_f]\in \mathcal{L}_p $ {\rm\ if\ and\ only\ if\ } $f$ is in the Besov space ${\rm B}_{p,p}^{2n+2\over p}(\mathbb H^n)$;

(2) when $0<p\leq 2n+2$, $[R_\ell,M_f]\in \mathcal{L}_p $ {\rm\ if\ and\ only\ if\ } $f$ is a constant.

We point out that $p=2n+2$ (the homogeneous dimension of $\HH^n$) is the endpoint in the setting of $\mathbb H^n$.
  See also the remarks in \cite[p. 266--276]{RS} for the endpoint.
  Moreover, recall from  (\cite[p. 254]{P2}) that for any $0<p<2n+2$, one has
$\mathcal{L}_{p,\infty}\subset \mathcal{L}_{2n+2}$. This, together with (2) above, implies that for any $0<p< 2n+2$,
$[R_{\ell},M_f]\in \mathcal{L}_{p,\infty}$ if and only if $f$ is a constant.

Thus, along the line of \cite{JW,RS}, \cite{LMSZ} and \cite{FLLarxiv}, a natural question arises:

 {\bf Q:} ``How to characterise the endpoint weak Schatten class $\mathcal{L}_{2n+2,\infty}$ estimate for $[R_{\ell},M_f]$ on $\HH^n$ where one cannot apply Fourier analysis as effectively as in the Euclidean setting?''

\subsection{Statement of main results}
In this paper, we give a confirmative answer to the question in Section \ref{S1.1} in terms of the homogeneous Sobolev space on $\HH^n$ by developing a new approach based on Connes' trace theorem.


\begin{thm}\label{main}
Suppose that $f\in  L_{\infty}(\Heis^n)$. Then for any $\ell\in \{1,2,\cdots,2n\}$, the commutator $[R_{\ell},M_f]\in \mathcal{L}_{2n+2,\infty}$ if and only if $f\in \dot{W}^{1,2n+2}(\HH^n)$. Moreover, there exist positive constants $C$ and $c$ such that
\begin{align}\label{main eq}
c\|f\|_{\dot{W}^{1,2n+2}(\HH^n)}\leq \|[R_{\ell},M_f]\|_{\mathcal{L}_{2n+2,\infty}}\leq C\|f\|_{\dot{W}^{1,2n+2}(\HH^n)}.
\end{align}
\end{thm}
 {
Here, $\dot{W}^{1,2n+2}(\HH^n)$ is the homogeneous Sobolev space on $\HH^n$ defined by $\dot{W}^{1,2n+2}(\HH^n):=\big\{f\in(\mathcal S(\HH^n))': X_1f, X_2f,\cdots, \\ X_{2n}f \in L_{2n+2}(\HH^n)\big\}$ with the seminorm
$\| f\|_{\dot{W}^{1,2n+2}(\HH^n)} := \sum_{j=1}^{2n} \|X_j f\|_{L_{2n+2}(\HH^n)}$.
 $\mathcal S(\HH^n)$ denotes the test function space on $\HH^n$ and $(\mathcal S(\HH^n))'$ is the space of distributions (see for example \cite{FoSt}).

\smallskip
The proof of Theorem \ref{main} is a combination of Proposition \ref{sufficiency theorem} and Proposition \ref{necessary theorem}, which shows the second inequality in \eqref{main eq} and the first inequality in \eqref{main eq}, respectively. We explain the idea and techniques below.


{$\underline{ Part\ 1}$: the second inequality in \eqref{main eq}}.

We use the double operator integral techniques so that the endpoint weak-Schatten bound of Riesz transform commutator can be reduced to Cwikel-type estimate on Heisenberg group, which has been deeper discussed in \cite{MSZ_cwikel}.

{$\underline{ Part\ 2}$: the first inequality in \eqref{main eq}}.

The proof of this direction is far more difficult since one cannot apply Fourier transform as effectively as in the Euclidean space.

 As a main tool,  we prove a formula for the Dixmier trace of $|[R_{\ell},M_f]|^{2n+2}$ by
representing the  elements of ${\rm VN}(\HH^n)$ (the group von Neumann algebra
of $\HH^n$) invariant under the group dilation $\delta_r$  via a construction of the homomorphism
\[
    \pi_{\red}:\mathcal{B}(L_2(\mathbb{R}^n))\overline{\otimes} \mathbb{C}^2\to {\rm VN}(\Heis^n),
\]
where the image of $\pi_{\red}$ is precisely dilation-invariant subalgebra operators of ${\rm VN}(\HH^n).$  This $\pi_{\red}$ is a noncommutative substitute
for the representation of dilation invariant operators as Fourier multipliers on $\mathbb R^n$.


\begin{thm}\label{keytraceformula1}
For every $ \ell=1,\ldots, 2n$,  there exists a linearly independent subset $\{x_{k}^\ell\}_{k=1}^{2n}$ of $L_{2n+2}(\mathcal B(L_2(\mathbb{R}^n))\overline{\otimes} \mathbb C^2)$ such that for any $f\in L_{\infty}(\Heis^n)\cap \dot{W}^{1,2n+2}(\HH^n)$, and for every normalised continuous trace $\varphi$ on $\mathcal{L}_{1,\infty}$,
\begin{align}\label{traceformula000000}
\varphi(|[R_\ell,M_f]|^{2n+2}) = \bigg\| \sum_{k=1}^{2n} X_k f \otimes x_{k}^{\ell} \bigg\|_{L_{2n+2}(L_{\infty}(\HH^n)\overline{\otimes} \mathcal{B}(L_2(\mathbb{R}^n))\overline{\otimes} \mathbb{C}^2)}^{2n+2}.
\end{align}
\end{thm}

Due to the inequality $|\varphi(X)| \lesssim \|X\|_{\mathcal{L}_{1,\infty}}$, Theorem
\ref{keytraceformula1} supplies the lower bound in Theorem \ref{main} for sufficiently regular symbol $f$ of the multiplication operator $M_f$. The explicit forms of $x_{k}^\ell$ will be given in Proposition \ref{keytraceformula}, and the full proof of Theorem
\ref{keytraceformula1} is given in sections \ref{auxiliarysection}
--\ref{Proof of necessity direction}.

We remark that Theorem \ref{keytraceformula1} is a non-trivial variant of Dixmier trace formula established in \cite[Theorem 2]{LMSZ} and \cite[Theorem 3(iii)]{Connes1998}. In \cite{Connes1998}, the commutator of
$M_f$ with
the sign of the Dirac operator on a compact Riemannian spin $d$-manifold is termed a ``quantised differential", which links to $[R_{\ell},M_f]$. One can express the Dixmier trace of the $d$th power of a quantised derivative
as the integral of the $d$th power of the gradient of $f$ \cite[Theorem 3(iii)]{Connes1998}. In the words of Connes, by such a formula one can ``pass from quantized $1$-forms to ordinary forms, not by a classical limit, but by a direct application of the Dixmier trace".
Extensions and analogies of this formula \cite[Theorem 3(iii)]{Connes1998} in various non-compact and noncommutative settings have been obtained more recently, see for example the results \cite{LMSZ}, \cite{MSX2019} and \cite{MSX2018} on the non-compact manifold $\mathbb{R}^d$, the quantum Euclidean space and the quantum tori, respectively.


Unlike in \cite{Connes1998} where the quantised derivative is linked to $[R_\ell,M_f]$,  in the setting of Heisenberg groups, the relationship between $[R_\ell,M_f]$ and a ``quantised derivative" remains speculative. We are uncertain
at this stage whether our trace formula \eqref{traceformula000000} 
is relevant to noncommutative geometry, however
there are direct applications in harmonic analysis. For example, following from  \cite[Corollary 1.5]{MSX2019}, one can apply the trace formula \eqref{traceformula000000} to give an alternative proof of \cite[(2) in Theorem 1.1]{FLLarxiv}.

We highlight that the new approach that we develop opens the door for further study of the endpoint Schatten--Lorentz class estimates of commutators $[A,M_f]$ in various settings where Fourier analysis is not available (or not as effective as that in $\mathbb R^n$). Potential developments includes the endpoint Schatten--Lorentz class estimates for commutators $[R_\ell, M_f]$, where $R_\ell$ is the $\ell-$th Riesz transforms on H-type groups or stratified Lie groups, which will be addressed in our subsequent papers.

This paper is organized as follows. In Section \ref{Preliminaries} we provide some definitions and properties of several basic concepts, including Schatten class, traces, non-commutative $L_{p}$ spaces and $L_{p,\infty}$ spaces, group von Neumann algebra of $\mathbb{H}^n$ and Cwikel-type estimate on $\HH^n$. In Section \ref{approsection}, we provide an approximation lemma in the homogeneous Sobolev spaces, which allows us to concentrate on $f\in C_c^\infty(\HH^n)$ in the subsequent calculations for $[R_\ell,M_f]$.
In Section \ref{sufficiencysection}, we give the proof of the second inequality in \eqref{main eq}. In Sections \ref{auxiliarysection}
--\ref{Proof of necessity direction}, we present a complete proof of the first inequality in \eqref{main eq}.

Throughout the paper, we use $A\lesssim B$ to denote the statement that $A\leq CB$ for some constant $C>0$, and $A\approx B$ to denote the statement that $A\lesssim B$ and $B\lesssim A$.
}

\section{Preliminaries}\label{Preliminaries}
\setcounter{equation}{0}

\subsection{Schatten class, ideals and traces}\label{Spdef}
We recall the definitions of the Schatten classes $\mathcal{L}_{p}(H)$ and $\mathcal{L}_{p,\infty}(H)$. Let $H$ be any Hilbert space and $\mathcal{B}(H)$ be a set consisting of all bounded operators on $H$.  Note that if $A$ is any compact operator on $H$, then $A^{*}A$ is compact, positive and therefore diagonalizable. We define the singular values $\{\mu(k;A)\}_{k=0}^{\infty}$ be the sequence of square roots of eigenvalues of $A^{*}A$ (counted according to multiplicity). Equivalently, $\mu(k;A)$ can be characterized by
\begin{align*}
\mu(k;A)=\inf\{\|A-F\|:{\rm rank}(F)\leq k\}.
\end{align*}

For $0<p<\infty$, a compact operator $A$ on $H$ is said to belong to the Schatten class $\mathcal{L}_{p}(H)$ if $\{\mu(k;A)\}_{k=0}^{\infty}$ is $p$-summable, i.e. in the sequence space $\ell_{p}$. If $p\geq 1$, then the $\mathcal{L}_{p}(H)$ norm is defined as
\begin{align*}
\|A\|_{\mathcal{L}_{p}(H)}:=\left(\sum_{k=0}^{\infty}\mu(k;A)^{p}\right)^{1/p}.
\end{align*}
With this norm $\mathcal{L}_p(H)$ is a Banach space, and an ideal of $\mathcal{B}(H)$.

For $0<p<\infty$, the weak Schatten class $\mathcal{L}_{p,\infty}(H)$ is the set consisting of operators $A$ on $H$ such that $\{\mu(k;A)\}_{k=0}^{\infty}$ is in $\ell_{p,\infty}$, with quasi-norm:
\begin{align*}
\|A\|_{\mathcal{L}_{p,\infty}(H)}=\sup\limits_{k\geq0}(k+1)^{1/p}\mu(k;A)<\infty.
\end{align*}
For convenience, we shall use the abbreviations $\mathcal{L}_{p}$ and $\mathcal{L}_{p,\infty}$ to denote $\mathcal{L}_{p}(H)$ and $\mathcal{L}_{p,\infty}(H)$ respectively, whenever the Hilbert space $H$ is clear from context. As with $\mathcal{L}_{p}$ spaces, $\mathcal{L}_{p,\infty}$ is an ideal of $\mathcal{B}(H)$. More details about Schatten class can be found in e.g. \cite{LSZ1,LSZ2}.

In this article, the ideal $\mathcal{L}_{1,\infty}$ and the traces on this ideal play crucial roles in the proof. Recall that a trace on $\mathcal{L}_{1,\infty}$ is a linear functional $\varphi:\mathcal{L}_{1,\infty}\rightarrow \mathbb{C}$ such that for any bounded operators $A$ and for any $B\in\mathcal{L}_{1,\infty}$, we have $\varphi(AB)=\varphi(BA)$. The trace $\varphi$ is called continuous when it is continuous with respect to the $\mathcal{L}_{1,\infty}$ quasinorm. Given an orthonormal basis $\{e_n\}_{n=0}^{\infty}$ on $H$, we define the operator $A:={\rm diag}\left\{\frac{1}{n+1}\right\}_{n=0}^{\infty}$ by $\langle e_n,Ae_m\rangle=\delta_{n,m}\frac{1}{n+1}$.

We say that $\varphi$ is normalized if
\begin{align*}
\varphi\left({\rm diag}\left\{\frac{1}{n+1}\right\}_{n=0}^{\infty}\right)=1.
\end{align*}
Note that for any unitary operators $U$ and all bounded operators $B$ we have $\varphi(UBU^*)=\varphi(B)$, which means that the property that $\varphi$ is normalized is independent of the choice of orthonormal basis.

To continue, for $1\leq p <\infty$, we define an ideal $(\mathcal{L}_{p,\infty})_{0}$ by setting
\begin{align*}
(\mathcal{L}_{p,\infty})_{0}=\{A \in \mathcal{L}_{p,\infty}   :  \lim_{k\to +\infty} k \mu(k;A)^{p} = 0 \}.
\end{align*}
This is a closed subspace of $\mathcal{L}_{p,\infty}$ and coincides with the closure of ideal of all finite rank operators in $\mathcal{L}_{p,\infty}$ and is commonly called the separable part of $\mathcal{L}_{p,\infty}$ (more details about separable part can be found in \cite{LSZ1,LSZ2}).
 An important fact for us is that all continuous traces on $\mathcal{L}_{1,\infty}$ vanish on $(\mathcal{L}_{1,\infty})_0;$ this is a consequence of the singularity of all traces
 on $\mathcal{L}_{1,\infty},$ see for example \cite[Corollary 5.7.7]{LSZ1}. In this paper, we shall consider only continuous normalised traces.

 The following ideal property of $(\mathcal{L}_{p,\infty})_0$ is an easy consequence of H\"{o}lder's inequality, but would be frequently used in the context.
\begin{lem}
Suppose that $1\leq p,q,r<\infty$ satisfying $1/r=1/p+1/q$. If $A\in \mathcal{L}_{p,\infty}$, $B\in (\mathcal{L}_{q,\infty})_0$, then $AB\in (\mathcal{L}_{r,\infty})_{0}$.
\end{lem}
\begin{proof}
It suffices to note that by \cite[Corollary 2.3.16]{LSZ1},
\begin{align*}
(k+1)^{1/r}\mu(k;AB)\leq (k+1)^{1/p}\mu([k/2];A)(k+1)^{1/q}\mu([k/2];B)\leq \|A\|_{\mathcal{L}_{p,\infty}}(k+1)^{1/q}\mu([k/2];B)\rightarrow 0,
\end{align*}
as $k\rightarrow +\infty$. Therefore $AB\in (\mathcal{L}_{r,\infty})_{0}$.
\end{proof}


\subsection{Non-commutative $L_{p}$ spaces and $L_{p,\infty}$ spaces}
For the convenience of those readers who are not familiar with non-commutative analysis, we recall definitions of the classical non-commutative $L_p$ spaces and the weak $L_p$ spaces, which generalize the notion of Schatten classes $\mathcal{L}_p$ and $\mathcal{L}_{p,\infty}$ defined in the preceding subsection. More details can be found in \cite{LSZ1,PX}.

Let $\mathcal{M}\subset \mathcal{B}(H)$ be a von Neumann algebra equipped with a faithful normal semifinite trace $\tau$. A closed and densely defined operator $A:D(A)\rightarrow H$ is said to be affiliated with $\mathcal{M}$ if it commutes with every unitary operator from the commutant of $\mathcal{M}$; Besides, we say that $A$ is $\tau$-measurable if it is affiliated with $\mathcal{M}$ and for every $\varepsilon>0$, there exists a projection $p\in\mathcal{M}$ such that $(1-p)H$ is contained in the domain $D(A)$ of $A$ and $\tau(p)<\varepsilon$ (\cite[Definition 1.2]{FK}). For simplicity, we denote the set of all $\tau$-measurable opeartors by $\mathcal{S}(\mathcal{M},\tau)$.

Let $\mu(t;X)$ denote the $t$-th generalized $s$-number (\cite[Definition 2.1]{FK}) of $X\in \mathcal{S}(\mathcal{M},\tau)$. By \cite[Proposition 2.2]{FK},
\begin{align*}
\mu(t;X)=\inf\left\{s\geq 0:\tau(\chi_{(s,\infty)}(|X|))\leq t\right\}.
\end{align*}
For $1\leq p<\infty$,  the non-commutative $L_p$ space associated with $(\mathcal{M},\tau)$ is defined by
\begin{align*}
L_p(\mathcal{M},\tau):=\{X \ {\rm affiliated}\  {\rm with}\  \mathcal{M}:\tau(|X|^{p})^{1/p}<\infty\}.
\end{align*}
For the endpoint $p=\infty$, we denote $L_\infty(\mathcal{M},\tau):=\mathcal{M}$ and $\|\cdot\|_{L_\infty(\mathcal{M},\tau)}=\|\cdot\|_{H\rightarrow H}$ stand for the operator norm.
Besides, for $0 < p<\infty$, the weak non-commutative $L_p$ spaces associated with $(\mathcal{M},\tau)$ is defined by
\begin{align*}
L_{p,\infty}(\mathcal{M},\tau):=\{X \ {\rm affiliated}\  {\rm with}\  \mathcal{M}:\sup\limits_{t\geq 0}t^{1/p}\mu(t;X)<+\infty\}.
\end{align*}
If $0 < p<\infty$, then the space $L_{p,\infty}(\mathcal{M},\tau)$ equipped with the quasi-norm $\|\cdot\|_{L_{p,\infty}(\mathcal{M},\tau)}:=\sup\limits_{t\geq 0}t^{1/p}\mu(t;X)$ becomes a quasi-Banach space.


\subsection{Group von Neumann algebra of $\mathbb{H}^n$}\label{VNGdef}
At the beginning of this subsection, we recall the definition of the group von Neumann algebra on Heisenberg group from \cite[Section 7.2.1]{Pedersen2018}, \cite[Chapter VII, Section 3]{Takesaki_2}. Let $\lambda:\mathbb{H}^{n}\rightarrow \mathcal{B}(L_2(\mathbb{H}^{n}))$ be the left regular representation. That is, for any $\gamma\in\mathbb{H}^{n}$, $\lambda(\gamma)$ is the unitary operator on $L_2(\mathbb{H}^{n})$ defined by
\begin{align*}
(\lambda(\gamma) \xi)(h)=\xi(\gamma^{-1}h),\quad  h,\gamma\in\mathbb{H}^{n},\  \xi\in L_2(\mathbb{H}^n).
\end{align*}
Given $f \in C_c(\HH^n)$, write $\lambda(f)$ for the operator on $L_2(\HH^n)$ defined by
\[
(\lambda(f)\xi)(h) = \int_{\HH^n} f(h \gamma^{-1})\xi(\gamma)\,d\gamma,\quad \xi \in L_2(\HH^n),\; h \in \HH^n.
\]
Observe that $\lambda(f)\lambda(g)=\lambda(f\ast g),$ where $f\ast g \in C_c(\HH^n)$ is the convolution
\[
(f\ast g)(h) := \int_{\HH^n} f(h \gamma^{-1})g(\gamma)\,d\gamma,\quad f,g \in C_c(\HH^n),\;h \in \HH^n.
\]
Formally, $\lambda(f)\varphi = f\ast \varphi$ so that $\lambda(f)$ is rightfully called a ``convolution operator".

\begin{defi}\label{VNG_defi}
	The group von Neumann algebra ${\rm VN}(\HH^n)$ is defined to be closure of $\lambda(C_c(\HH^n))$
	in the weak operator topology of $\mathcal{B}(L_2(\HH^n)).$ Equivalently, ${\rm VN}(\HH^n)$
	is the weak operator topology closure of the linear span of the family $\{\lambda(\gamma)\}_{\gamma \in \HH^n}$. 
	That is,
	\[
	{\rm VN}(\HH^n) = \lambda(\HH^n)'' \subseteq \mathcal{B}(L_2(\HH^n)),
	\]
	where $''$ denotes the double commutant.
\end{defi}

\begin{defi}
	On the algebra $\lambda(C_c(\HH^n)),$ we define a functional
	\[
            \tau(\lambda(f)) = f(0),\quad f \in C_c(\HH^n).
	\]
\end{defi}
It is not obvious that $\tau$ is well-defined, indeed in principle it is possible that $\lambda$ is not injective. The fact that $\lambda$ is injective follows from \cite[Chapter VII, Theorem 3.4]{Takesaki_2}, and hence $\tau$ is well-defined on the algebra $\lambda(C_c(\HH^n)).$

Note that the unimodularity of $\HH^n$ implies that $\tau$ is a trace such that
\[
\tau(\lambda(f)\lambda(g)) = \tau(\lambda(g)\lambda(f)),\quad f,g \in C_c(\HH^n).
\]

%

\begin{lem}\cite[Proposition 7.2.8]{Pedersen2018}
	The functional $\tau$ uniquely extends to a semifinite normal trace on ${\rm VN}(\HH^n).$
\end{lem}

Now we recall from \cite[Section 6.2.3]{FischerRuzhansky2016} that if $G$ is a unimodular group, then
$$L_2(G) = \int_{\widehat{G}}^{\oplus} H_{s} \,d\mu(s),$$
where $\widehat{G}$ is the space of irreducible unitary representations $(\rho_{s},H_{s})$ of $G$, equipped with a Plancherel measure $\mu.$ Corresponding to this is the direct integral decomposition
$${\rm VN}(G) = \int_{\widehat{G}}^{\oplus} \rho_{s}({\rm VN}(G))\,d\mu(s).$$
That is, given $x \in {\rm VN}(G)$ and $\xi\in L_2(G)$ we have
$$x\xi = \int_{\widehat{G}}^{\oplus} \rho_{s}(x)\xi_{s}\,d\mu(s),$$
where $\xi_{s}$ is the component of $\xi$ in the representation $\xi = \int_{\widehat{G}}^{\oplus} \xi_{s} \,d\mu(s).$ For more information about direct integral decomposition, we refer the readers to \cite[Appendix B]{FischerRuzhansky2016} and the reference
therein.

To adapt the above direct integral decomposition to Heisenberg group and its group von Neumann algebra, we first recall Schr\"{o}dinger representations of Heisenberg group (see for example \cite{FischerRuzhansky2016}), which are the infinite dimensional unitary representations of $\HH^n$. For $s\in\mathbb{R}^{*}=\mathbb{R}\backslash\{0\}$, consider the Schr\"{o}dinger representations of $\mathbb{H}^{n}$ acting on $L_2(\mathbb{R}^{n})$ defined by
\begin{align*}
\rho_{s}(g)\varphi(u)=e^{is(t+\frac{1}{2}x\cdot y)}e^{i{\rm sgn}(s)\sqrt{|s|} y \cdot u}\varphi\big(u+\sqrt{|s|}x\big),
\end{align*}
where $\cdot$ is the usual inner product in $\mathbb{R}^{n}$, $g=(z,t)$, $z=x+iy$, $x,y,u\in\mathbb{R}^{n}$ and $\varphi\in L_2(\mathbb{R}^{n})$ (See \cite[Formula (6.8)]{FischerRuzhansky2016}).

Now let $G=\Heis^n$ be Heisenberg group with Lie algebra $\mathfrak{g}$, then the unitary dual $(\widehat{\Heis^n},\mu)$ is summarised in \cite[Equation (6.29)]{FischerRuzhansky2016},
$$\widehat{\HH^n} = \mathbb{R}\setminus\{0\},\quad H_{s} := L_2(\mathbb{R}^n)$$
up to a set of zero Plancherel measure.
The Plancherel measure $d\mu(s)$ is
$$d\mu(s) = c_n|s|^n\,ds,$$
where $c_n>0$ is some constant.
The representation at $s \in \widehat{\HH^n}$, denoted $\rho_{s}$, satisfies the relations \cite[Section 6.3.3, p.453]{FischerRuzhansky2016}
$$\rho_{s}(X_j)=i|s|^{\frac12}p_j,\quad \rho_{s}(Y_j)= i\sgn(s)|s|^{\frac12}q_j,\quad \rho_{s}(T) = is,\quad 1\leq j\leq n.$$
Here $p_j := -i\partial_{x_j}$ and $q_j := M_{x_j}$ denote the momentum and position operators on $L_2(\mathbb{R}^n)$, respectively, and we identify $\rho_s$ with the corresponding representation of the Lie algebra of $\HH^n$ on $L_2(\mathbb{R}^n).$
The representation $\rho_s$ can be extended to a representation of ${\rm VN}(\HH^n)$ on $L_2(\mathbb{R}^n),$ and by the Stone-von Neumann theorem
\[
    \rho_s({\rm VN}(\Heis^n))  = \mathcal{B}(H_{s}) = \mathcal{B}(L_2(\mathbb{R}^n)),\quad s \in \mathbb{R}\setminus \{0\}.
\]

Since every $\mathcal{B}(H_{s})$ is the same,  we have
$${\rm VN}(\Heis^n) \cong \mathcal{B}(L_2(\mathbb{R}^n))\overline{\otimes} L_{\infty}(\mathbb{R}).$$
An explicit $*$-isomorphism is provided by the following proposition.
\begin{prop}\label{pi rigorous prop}
There exists a unique isomorphism of von Neumann algebras
$$\pi:\mathcal{B}(L_2(\mathbb{R}^n))\overline{\otimes} L_{\infty}(\mathbb{R})\to {\rm VN}(\Heis^n)$$
such that
$$\exp(itp_j\otimes |s|^{\frac12})\mapsto\lambda(\exp(tX_j)),\quad \exp(itq_j\otimes \sgn(s)|s|^{\frac12})\mapsto\lambda(\exp(tY_j)).$$
Here, we use the notation $s$ for the coordinate function $s\mapsto s.$
The isomorphism $\pi$ is trace-preserving, in the sense that if $x \in L_1(\mathcal{B}(L_2(\mathbb{R}^n))\overline{\otimes} L_{\infty}(\mathbb{R})),$ then we have $\pi(x) \in L_1(\mathrm{VN}(\HH^n))$
and
\[
    \tau(\pi(x)) = c_n\int_{-\infty}^\infty {\rm Tr}(x(s))\,s^nds.
\]
Here, $x$ is regarded as an integrable function from $(\mathbb{R},s^n\,ds)$ to $\mathcal{L}_1(L_2(\mathbb{R}^n)).$
\end{prop}
\begin{proof}
We decompose $L_2(\Heis^n)$ as
$$L_2(\Heis^n) \equiv \int_{\mathbb{R}}^{\oplus} H_{s}s^n\,ds \equiv L_2(\mathbb{R}^n){\otimes} L_2(\mathbb{R},s^n\,ds),$$
where ${\otimes}$ is the Hilbert space tensor product \cite[Chapter IV, Section 1, Definition 1.2]{Takesaki_1}.
The action $\rho$ of ${\rm VN}(\Heis^n)$ on $L_2(\HH^n)$ in terms of this decomposition is given by
$$\rho = \int_{\widehat{\HH^n}}^{\oplus} \rho_{s}\,d\mu(s)=\int_{\mathbb{R}}^{\oplus}\rho_{s}\cdot s^nds.$$
	
Using the notation $s$ for the coordinate function $s\mapsto s,$ the explicit form of the representation $\rho_s$ yields
$$\rho(\lambda(\exp(tX_j))) = \exp(itp_j\otimes \sgn(s)|s|^{1/2}),\quad \rho(\lambda(\exp(tY_j))) = \exp(itq_j\otimes |s|^{1/2}).$$
In fact $\rho$ is uniquely specified by these identities, since the ${\rm VN}(\HH^n)$ is generated by $\lambda(\HH^n),$ and $\HH^n$ is generated by the semigroups $\exp(tX_j),$ $\exp(tY_j),$
for $j=1,\ldots,n.$

The map $\rho$ is an isomorphism, because it coincides with the direct integral decomposition of ${\rm VN}(\HH^n)$. Hence the inverse of $\rho,$ $\pi:= \rho^{-1}$ is also an isomorphism.

The fact that $\pi$ is trace-preserving follows from the corresponding assertion about $\rho,$ which is \cite[p.368]{DixmierCStar1977}.
\end{proof}

\begin{cor}\label{pi_computations} We have
$$\pi(H\otimes|s|)=-\Delta_{\mathbb{H}},\quad\pi(ip_jH^{-\frac12}\otimes 1) = R_j,\quad \pi(iq_jH^{-\frac12}\otimes {\rm sgn}(s)) = R_{j+n},\quad 1\leq j\leq n.$$
Here, $H$ is the $n$-dimensional harmonic oscillator defined as the operator on $L_2(\mathbb{R}^n)$ by the usual formula
$$H=\sum_{j=1}^np_j^2+q_j^2.$$ Since $-H\otimes |s|$ is an unbounded operator affiliated to the von Neumann algebra $\mathcal{B}(L_2(\mathbb{R}^n))\overline{\otimes} L_{\infty}(\mathbb{R})$ and not contained in the domain of $\pi,$ the first identity above should be intepreted as
\[
    \pi(e^{itH\otimes |s|}) = e^{-it\Delta_{\mathbb{H}}},\quad t \in \mathbb{R}.
\]
\end{cor}

%

Besides the algebra ${\rm VN}(\HH^n),$ and its representation as $\mathcal{B}(L_2(\mathbb{R}^n))\overline{\otimes} L_{\infty}(\mathbb{R})$ via $\pi,$ we are also concerned with the subalgebra of all
\emph{dilation invariant} elements of ${\rm VN}(\HH^n),$ where in this context an operator $A \in {\rm VN}(\HH^n)$ is dilation invariant if for all $r>0$ we have
\begin{equation}\label{unitary_dilation_definition}
    \sigma_{r^{-1}}\circ A\circ \sigma_r = A,
\end{equation}
where $r\mapsto \sigma_r$ is the unitary action of $\mathbb{R}_+$ induced by the anisotropic dilation $r\mapsto \delta_r,$
\[
    \sigma_ru(\gamma) = r^{2n+2}u(\delta_r \gamma),\quad u \in L_2(\HH^n),\; \gamma \in \HH^n.
\]

Under the isomorphism $\pi,$ the dilation $\sigma_r$ acts as dilation on the second tensor factor. That is,
\[
    \sigma_r\pi(x)\sigma_{r^{-1}} = \pi((1\otimes \alpha_r)(x)),\quad x \in \mathcal{B}(L_2(\mathbb{R}^n))\overline{\otimes} L_{\infty}(\mathbb{R}),
\]
where $\alpha_r:L_{\infty}(\mathbb{R})\to L_{\infty}(\mathbb{R})$ is the automorphism
\[
    \alpha_r f(t) = f(r^{2}t),\quad t \in \mathbb{R},\; f \in L_{\infty}(\mathbb{R}).
\]

Let $\iota$ denote the embedding
$$\iota:\mathcal{B}(L_2(\mathbb{R}^n)){\overline{\otimes}} \mathbb{C}^2\hookrightarrow \mathcal{B}(L_2(\mathbb{R}^n))\overline{\otimes} L_{\infty}(\mathbb{R}),$$
where $\mathbb{C}^2$ is considered as a subalgebra of $L_{\infty}(\mathbb{R})$, according to the embedding
$$(z_1,z_2) \mapsto z_1\chi_{(-\infty,0)}+z_2\chi_{(0,\infty)} \in L_{\infty}(\mathbb{R}),\quad (z_1,z_2)\in\mathbb{C}^2.$$
From hereon, we understand $\mathbb{C}^2$ as being identified with the two-dimensional von Neumann algebra $\mathbb{C}\oplus \mathbb{C}.$ That is,
\[
    (z_1,z_2)(w_1,w_2) = (z_1w_1,z_2w_2),\quad (z_1,z_2),(w_1,w_2)\in \mathbb{C}^2
\]
and
\[
    (z_1,z_2)^* = \overline{(z_1,z_2)} = (\overline{z_1},\overline{z_2}),\quad (z_1,z_2)\in \mathbb{C}^2.
\]

We define a $*$-homomorphism
$$\pi_{\red}:\mathcal{B}(L_2(\mathbb{R}^n))\overline{\otimes} \mathbb{C}^2\to {\rm VN}(\mathbb{H}^n)$$
by setting
$$\pi_{\red} := \pi\circ \iota.$$
Observing that $\iota\begin{pmatrix} -1 \\ 1 \end{pmatrix} = \sgn(s),$ it follows from Corollary \ref{pi_computations} that
\begin{equation}\label{Riesz_in_terms_of_pired}
\pi_{\red}\left(ip_jH^{-\frac12}\otimes \begin{pmatrix} 1\\ 1\end{pmatrix}\right) = R_j,\quad \pi_{\red}\left(iq_jH^{-\frac12}\otimes \begin{pmatrix} -1\\ 1\end{pmatrix}\right) = R_{j+n},\quad 1\leq j\leq n.
\end{equation}
Similarly, \eqref{Riesz_in_terms_of_pired} implies that
\begin{equation}\label{T_in_terms_of_pired}
    \pi_{\red}\left(H^{-1}\otimes \begin{pmatrix} 1 \\ 1 \end{pmatrix}\right) = |T|(-\Delta_{\HH})^{-1}.
\end{equation}

Observe that $\iota(\mathcal{B}(L_2(\mathbb{R}^n))\otimes \mathbb{C}^2)$ is precisely the $r\mapsto 1\otimes \alpha_r$-invariant subalgebra
of $\mathcal{B}(L_2(\mathbb{R}^n))\otimes L_{\infty}(\mathbb{R}),$ and hence the image of $\pi_{\red}$ consists of the dilation-invariant subalgebra of ${\rm VN}(\HH^n).$
\subsection{Cwikel-type estimate on Heisenberg group}

Cwikel-type estimates on $L_2(\HH^n)$ concern operators of the form
\[
(M_{f}\lambda(g))u(\gamma) = \int_{\HH^n} f(\gamma)g(\gamma \eta^{-1})u(\eta)\,d\nu(\eta),\quad u\in L_2(\HH^n),\; \gamma \in \HH^n.
\]
More generally, we may consider operators of the form $M_fx$ where $x$ is affiliated to the von Neumann algebra ${\rm VN}(\mathbb{H}^n).$ The basic Cwikel-type estimate we refer to is the following:

\begin{thm}\cite{MSZ_cwikel}\label{L_2_L_infty_cwikel_estimate}
	If $2 < p < \infty$, then there are constants $C_p$ and $C_p'$ such that
	\[
	\|M_f x\|_{\mathcal{L}_p} \leq C_p\|f\|_{L_p(\HH^n)}\|x\|_{L_p(VN(\HH^n),\tau)},\quad \|M_f x\|_{\mathcal{L}_{p,\infty}} \leq C_{p}'\|f\|_{L_p(\HH^n)}\|x\|_{L_{p,\infty}(VN(\HH^n),\tau)}.
	\]
\end{thm}

Theorem \ref{L_2_L_infty_cwikel_estimate} fundamentally uses the following formula due to M.~Christ:
\begin{lem}\cite[Proposition 3]{Christ1991}\label{Christlm}
Let $m\in L_{\infty}([0,\infty))$, then $m(-\Delta_{\mathbb{H}})$ is a convolution operator. Denote its convolution kernel by $\check{m}$. Then $\check{m}\in L_{2}(\mathbb{H})$ if and only if
\begin{align*}
\int_{0}^{\infty}|m(s)|^{2}s^{n}ds<\infty.
\end{align*}
Moreover, there exists a constant $c$ such that for all such $m$,
\begin{align*}
\|\check{m}\|_{2}^{2}=c\int_{0}^{\infty}|m(s)|^{2}s^{n}ds.
\end{align*}
\end{lem}

\begin{cor}\label{tau of gDelta} For any $g\in L_1(\mathbb{R}_{+},s^n ds)$, we have
	$$\tau(g(-\Delta_{\HH}))=\int_0^{\infty}g(s)s^nds.$$
\end{cor}
\begin{proof} If $g\geq0,$ then by Lemma \ref{Christlm},
	$$\tau(g(-\Delta_{\HH}))=\|g^{\frac12}(-\Delta_{\HH})\|_{L_2({\rm VN}(\mathbb{H}^n),\tau)}^2=\|g^{\frac12}\|_{L_2(\mathbb{R}_+,t^ndt)}^2=\int_0^{\infty}g(s)s^nds.$$
	The assertion follows now by linearity.
\end{proof}
%
%
%
%

We record some useful consequences of Theorem \ref{L_2_L_infty_cwikel_estimate}. The following two results are proved in \cite{MSZ_cwikel}.
\begin{lem}\label{L_p_cwikel} \cite[Theorem 1.1]{MSZ_cwikel}
    Let $p > 2,$ $f \in L_p(\Heis^n)$ and $g\in L_p(\mathbb{R}_+,t^n\,dt)$. We have
    \[
        \|M_fg(-\Delta_{\HH})\|_{\mathcal{L}_p} \leq C_p\|f\|_{L_p(\Heis^n)}\|g\|_{L_p(\mathbb{R}_+,t^n\,dt)}.
    \]
    If $g \in L_{p,\infty}(\mathbb{R}_+,t^n\,dt),$ then
    \[
        \|M_fg(-\Delta_{\HH})\|_{\mathcal{L}_{p,\infty}} \leq C_p\|f\|_{L_p(\Heis^n)}\|g\|_{L_{p,\infty}(\mathbb{R}_+,t^n\,dt)}.
    \]
\end{lem}

\begin{cor}\label{specific_cwikel} \cite[Section 3]{MSZ_cwikel}
    For any $0 < \beta < n+1$, if $f \in L_{\frac{2n+2}{\beta}}(\Heis^n)$, then the operator
    \[
        M_f(-\Delta_{\HH})^{-\frac{\beta}{2}}
    \]
    is compact, and belongs to the ideal $\mathcal{L}_{\frac{2n+2}{\beta},\infty}.$ Moreover, there exists a constant $C_{n,\beta}>0$ such that
    \[
        \|M_f(-\Delta_{\HH})^{-\frac{\beta}{2}}\|_{\mathcal{L}_{\frac{2n+2}{\beta},\infty}} \leq C_{n,\beta}\|f\|_{L_{\frac{2n+2}{\beta}}(\Heis^n)}.
    \]
\end{cor}
\begin{defi}
   Let ${\rm Tr}\otimes \Sigma$ denote the trace on the von Neumann algebra $\mathcal{B}(L_2(\mathbb{R}^n))\overline{\otimes} \mathbb{C}^2$ given by the tensor product
    of the standard operator trace on $\mathcal{B}(L_2(\mathbb{R}^n))$ and the ``sum" functional
    \[
        \Sigma:\mathbb{C}^2\to \mathbb{C},\quad \Sigma(\xi) = \xi_1+\xi_2,\quad \xi\in \mathbb{C}^2.
    \]
\end{defi}

Note that the operator $T = \frac{\partial}{\partial t}$ on $L_2(\HH^n)$ is not invertible, its spectrum consists of the entire real line $\mathbb{R}.$ However, as the next lemma shows, we can define the operator
\[
    \pi_{\red}(x)|T|^{-\frac12}
\]
for all $x \in L_{2n+2,\infty}(\mathrm{VN}(\HH^n),\tau).$ This should be interpreted as the operator norm limit of
\[
    \pi_{\red}(x)\chi_{(\varepsilon,\infty)}(|T|)|T|^{-\frac12}
\]
as $\varepsilon\downarrow 0.$ In fact, the operator $\pi_{\red}(x)|T|^{-\frac12}$ is compact as the next lemma shows.

\begin{lem} If $x\in L_p(\mathcal{B}(L_2(\mathbb{R}^n))\overline{\otimes}\mathbb{C}^2,{\rm Tr}\otimes\Sigma),$ then
	$$\|\pi_{{\rm red}}(x)|T|^{-\frac12}\|_{L_{2n+2,\infty}(\mathrm{VN}(\HH^n),\tau)}=c_n\|x\|_{L_{2n+2}(\mathcal{B}(L_2(\mathbb{R}^n))\overline{\otimes}\mathbb{C}^2,{\rm Tr}\otimes\Sigma)}.$$
	Here, $c_n>0$ is a constant.
\end{lem}
\begin{proof}
    We identify $L_{\infty}(\mathbb{R})$ with $\mathbb{C}^2\overline{\otimes} L_{\infty}(0,\infty).$ Hence, $\mathcal{B}(L_2(\mathbb{R}^n))\overline{\otimes}L_{\infty}(\mathbb{R})$ is identified with $\mathcal{B}(L_2(\mathbb{R}^n))\overline{\otimes}\mathbb{C}^2\overline{\otimes}L_{\infty}(0,\infty).$ The measure  $|t|^{2n+1}dt$ on $\mathbb{R}$ is identified with $\Sigma\otimes t^{2n+1}dt.$

    Observe that $x\otimes s^{-1} \in L_{p,\infty}(\mathcal{B}(L_2(\mathbb{R}^n))\overline{\otimes} L_{\infty}(\mathbb{R})),$ where the second factor is equipped with the Plancherel measure $d\mu(s) = c_ns^{n}\,ds.$
    Hence,
    $$\pi_{{\rm red}}(x)|T|^{-\frac12}=\pi(x\otimes s^{-1}).$$
    Since $\pi$ preserves $L_{p,\infty}$-norm, it follows that
    $$\|\pi_{{\rm red}}(x)|T|^{-\frac12}\|_{\mathcal{L}_{2n+2,\infty}}=\|x\otimes s^{-1}\|_{L_{p,\infty}(\mathcal{B}(L_2(\mathbb{R}^n))\overline{\otimes} L_{\infty}(\mathbb{R}))}=c_n\|x\|_{L_{2n+2}(\mathcal{B}(L_2(\mathbb{R}^n))\bar{\otimes}\mathbb{C}^2,{\rm Tr}\otimes\Sigma)}.$$
    Here, the last equality follows from the claim in the proof of \cite[Lemma 3.7]{CSZ_AIF}.
\end{proof}

\begin{lem}\label{homogeneous cwikel lemma} For all $x \in L_{2n+2}(\mathcal{B}(L_2(\mathbb{R}^n))\overline{\otimes} \mathbb{C}^2)$ and $f \in L_{2n+2}(\HH^n),$ we have
	$$\|M_f\pi_{{\rm red}}(x)|T|^{-\frac12}\|_{\mathcal{L}_{2n+2,\infty}}\leq c_n\|f\|_{L_{2n+2}(\HH^n)}\|x\|_{L_{2n+2}(\mathcal{B}(L_2(\mathbb{R}^n))\bar{\otimes}\mathbb{C}^2,{\rm Tr}\otimes\Sigma)}.$$
\end{lem}
\begin{proof} It follows from the preceding lemma and Lemma \ref{L_p_cwikel} with $p=2n+2.$
\end{proof}

\begin{lem}\label{cwikel-like commutator lemma}
    For all $f \in C^\infty_c(\HH^n),$ we have
	$$[M_f,(-\Delta_{\HH})^{-\frac12}]\in(\mathcal{L}_{2n+2,\infty})_0.$$
\end{lem}
\begin{proof}
    It is proved as \cite[Theorem 4.1(iv)]{MSZ_cwikel} that
    \[
        [M_f,(1-\Delta_{\HH})^{-\frac12}] \in \mathcal{L}_{n+1,\infty}\subset (\mathcal{L}_{2n+2,\infty})_0.
    \]
    It follows from Lemma \ref{L_p_cwikel} that
    \[
        ((-\Delta_{\HH})^{-\frac12}-(1-\Delta_{\HH})^{-\frac12})M_f,\quad M_f((-\Delta_{\HH})^{-\frac12}-(1-\Delta_{\HH})^{-\frac12}) \in \mathcal{L}_{p,\infty}\subset (\mathcal{L}_{2n+2,\infty})_0
    \]
    for every $n+1 < p < 2n+2.$
    Combining these assertions yields the desired inclusion.
\end{proof}

%

%

\bigskip

\section{Approximation lemmas in the homogeneous Sobolev spaces}\label{approsection}

We start with Poincar\'{e}'s inequality on Heisenberg groups.
\begin{lem}\label{Poin} Let $1<p<\infty.$ Then there exists a constant $C_{n,p}>0$ such that for any $f\in \dot{W}^{1,p}(\HH^n)$,
\begin{align}\label{Poin ine for ball}
\bigg( \int_{B(0,{r\over2})}|f(g)-\langle f\rangle_{B(0,{r\over2})}|^pdg\bigg)^{1\over p}
&\leq C_{n,p}\, r\, \bigg(\int_{B(0,2r)}|\nabla f(g)|^p \ dg\bigg)^{1\over p},
\end{align}
where for any $0<r<\infty$, $B(0,r)$ is a ball in $\HH^n$ with respect to the Kor\'{a}nyi metric and $\langle f\rangle_{B(0,r)}$ denotes the average of $f$ over $B(0,r).$
\end{lem}
\begin{proof} Note that it was shown in \cite[Page 507]{Je} that there is a constant $\gamma$ such that
\begin{align}\label{Poin ine for ball pre}
\int_{B(0,{1\over2})}|f(g)-\gamma|^2 dg
\leq C  \int_{B(0,2)}|\nabla f(g)|^2 \ dg,
\end{align}
which,  together with dilation (as commented in line 1 on p.507 in \cite{Je}), gives
$$
\bigg(\int_{B(0,{r\over2})}|f(g)-\gamma|^2 dg\bigg)^{1\over 2}
\leq C r \bigg(\int_{B(0,2r)}|\nabla f(g)|^2 \ dg\bigg)^{1\over 2}.
$$
Next, we note that
\begin{align*}
\bigg(\int_{B(0,{r\over2})}|f(g)-\langle f\rangle_{B(0,{r\over2})}|^2 dg\bigg)^{1\over 2}
&\leq  \bigg(\int_{B(0,{r\over2})}|f(g)-\gamma|^2 dg\bigg)^{1\over 2}+\bigg(\int_{B(0,{r\over2})}|\gamma-\langle f\rangle_{B(0,{r\over2})}|^2 dg\bigg)^{1\over 2}\\
&\leq C r\, \bigg(  \int_{B(0,2r)}|\nabla f(g)|^2 \ dg\bigg)^{1\over 2},
\end{align*}
which shows that \eqref{Poin ine for ball} holds for $p=2$.

We point out that the argument for general $1<p<\infty$ follows by the same technique as that for $p=2$ since as stated in \cite[Section 6]{Je}, 
the inequality \eqref{Poin ine for ball pre} also holds for $1<p<\infty$. The proof of Lemma \ref{Poin} is complete.
%
%
%
%
\end{proof}

We point out that based on Lemma \ref{Poin}, by following the argument in \cite[Section 5]{Je} via Whitney decomposition, one can deduce that for every $g \in \HH^n$ and $r>0$,
$$\bigg( \int_{B(g,r)}|f(h)-\langle f\rangle_{B(g,{r})}|^pdh\bigg)^{1\over p}
\leq C_{n,p}\, r\, \bigg(\int_{B( g,r)}|\nabla f(h)|^p \ dh\bigg)^{1\over p},
$$
It further implies that  (\cite[Section 6]{Je}, see also \cite{FGW})
 for any $R>0$,
\begin{align}\label{Poincare on annuli}
\bigg\| f-  {1\over | B(0,2R)\backslash B(0,R) |}\int_{ B(0,2R)\backslash B(0,R)}f(h)dh \bigg\|_{L^p(B(0,2R)\backslash B(0,R))}
\leq CR \|\nabla f \|_{L^p(B(0,2R)\backslash B(0,R))},
\end{align}
where the constant $C$ depends only on $n$ and $p$.

\begin{lem}\label{predensity}
$C^\infty(\HH^n)\cap\dot{W}^{1,p}(\HH^n)$ is dense in $\dot{W}^{1,p}(\HH^n)$.
\end{lem}
\begin{proof}
	To begin with, fix $f\in \dot{W}^{1,p}(\HH^n)$ and choose $\phi\in C_c^\infty(\HH^n)$ with
	$ \int_{\HH^n} \phi(g)\,dg=1. $
	Define
	$ \phi_{m}(g) = m^{2n+2} \phi( \delta_{m}(g)), $ where $\delta_{m}(g)=\delta_{m}([z,t]):=(mz,m^2t)$.
	then we see that $\phi_{m} *f\in C^\infty(\HH^n)$.
	Moreover,
	we have that
	$ \phi_{m} *f \to f \ {\rm as}\ {m}\to +\infty $
	in the sense of $(\mathcal S(\HH^n))'$ (see for example \cite[Proposition 1.49]{FoSt}).
	Furthermore, by noting that $X_j (\phi_{m} *f) = \phi_{m} * (X_jf)$ (see for example \cite[p. 38]{FoSt}),
	we conclude that $ X_j (\phi_{m} *f)  \to X_j f$ as $m\rightarrow +\infty $ in the sense of $L_p(\HH^n)$.
	
	Hence, $f$ can be approximated by functions $\phi_{m} *f$ in $C^\infty(\HH^n)$ and
	$$ \| f\|_{\dot{W}^{1,p}(\HH^n)}  =\lim_{{m}\to +\infty}  \| \phi_{m}*f\|_{\dot{W}^{1,p}(\HH^n)}. $$
This ends the proof of Lemma \ref{predensity}.
\end{proof}

\begin{lem}\label{density2}
Let $1<p<\infty$. For any
	$f \in \dot{W}^{1,p}(\HH^n)$, one can find a sequence $\{f_m\}_{m=1}^\infty \subset C_c^\infty(\HH^n)$ such that $f_m\to f $ in $\dot{W}^{1,p}(\HH^n)$. If $f$ satisfies an extra condition $f\in L_{\infty}(\HH^n)\cap\dot{W}^{1,p}(\HH^n)$, then there exists a constant $c$ such that $\{f_m\}_{m=1}^\infty$ also satisfy

	${\rm (a)}$ $\sup_m\| f_m\|_{L_\infty(\HH^n)} <\infty$;

	${\rm (b)}$ $f_m\to f-c $ uniformly in any compact subset in $\HH^n$;

    ${\rm (c)}$ $M_{f_m}\rightarrow M_{f-c}$ in the strong operator topology.
	
\end{lem}




\begin{proof}

By Lemma \ref{predensity}, it suffices to show the argument holds for $f\in C^\infty(\HH^n)\cap\dot{W}^{1,p}(\HH^n)$. To this end, we consider the smooth cut-off function
	$\eta_m \in C_c^\infty(\HH^n)$ such that
	${\rm supp\ } \eta_m\subset B(0,m),\ \eta_m\equiv1\ {\rm on}\ B(0,m/2), $
	and that
	$ |\nabla\eta_m(g)| \lesssim {1\over m}.$ Let
$$c_m:={1\over | B(0,2m)\backslash B(0,m) |}\int_{ B(0,2m)\backslash B(0,m)}f(g)dg .$$

It is easy to see that
\begin{align}\label{deffm}
f_m:=(f-  c_m)\eta_m\in C_c^\infty(\HH^n).
\end{align}

	Next we claim that
	\begin{align}\label{approclaim}
	f_m \to f \quad {\rm in}\ \ \dot{W}^{1,p}(\HH^n).
	\end{align}
	To see this,
	\begin{align*}
	&\sum_{j=1}^{2n}\Bigg\| X_j\Bigg[  \bigg(f- {1\over | B(0,2m)\backslash B(0,m) |} \int_{ B(0,2m)\backslash B(0,m)}f(g)dg \bigg)\eta_m \Bigg] - X_j f\Bigg\|_{L_p(\mathbb H^n)}\\
	&=\sum_{j=1}^{2n}\Bigg\| X_j  f \cdot \eta_m -X_jf-  \bigg(f- {1\over | B(0,2m)\backslash B(0,m) |} \int_{ B(0,2m)\backslash B(0,m)}f(g)dg \bigg)X_j \eta_m \Bigg\|_{L_p(\mathbb H^n)}\\
	&\leq\sum_{j=1}^{2n}\Big\| X_j  f \cdot \eta_m -X_jf\Big\|_{L_p(\mathbb H^n)}+\sum_{j=1}^{2n}\Bigg\|   \bigg(f-  {1\over | B(0,2m)\backslash B(0,m) |}\int_{ B(0,2m)\backslash B(0,m)}f(g)dg \bigg)X_j \eta_m \Bigg\|_{L_p(\mathbb H^n)}\\
	&=: {\rm Term_1}+{\rm Term_2}.
	\end{align*}
	
	For ${\rm Term_1}$, it is clear that
	\begin{align*}
	{\rm Term_1}
	&\leq\sum_{j=1}^{2n}\Big\| X_j  f \cdot \eta_m -X_jf\Big\|_{L_p(\mathbb H^n)}
	\leq C\sum_{j=1}^{2n}\Big\| X_jf\Big\|_{L_p(B(0,m/2)^c)}\to0
	\end{align*}
	as $m\to+\infty$.
	
	To deal with the ${\rm Term_2}$, we  apply the Poincar\'e inequality \eqref{Poincare on annuli} to conclude that
	%
	%
	%
	\begin{align}\label{Poincare}
	\bigg\| f- {1\over | B(0,2m)\backslash B(0,m) |} \int_{ B(0,2m)\backslash B(0,m)}f(g)dg \bigg\|_{L_p(B(0,2m)\backslash B(0,m))}
	\leq Cm \|\nabla f \|_{L_p(B(0,2m)\backslash B(0,m))}.
	\end{align}

	From the property of $\eta_m$ we see that  $X_j \eta_m $ is supported in $B(0,2m)\backslash B(0,m)$ with the upper bound $m^{-1}$ (up to a multiplicative constant). Hence, by using the Poincar\'{e}'s inequality \eqref{Poincare}, we have
	\begin{align*}
	{\rm Term_2}&\leq C {1\over m}\sum_{j=1}^{2n}  \Bigg\|   f- {1\over | B(0,2m)\backslash B(0,m) |} \int_{ B(0,2m)\backslash B(0,m)}f(g)dg \Bigg\|_{L_p( B(0,2m)\backslash B(0,m) )}
	\leq C\|\nabla f \|_{L_p(B(0,2m)\backslash B(0,m))}\to0
	\end{align*}
	as $m\to+\infty$.
	
	Combining the estimates for ${\rm Term_1}$ and ${\rm Term_2}$ yields the proof of claim \eqref{approclaim}. 
	For the second statement, we note that $(a)$ is a direct consequence of the following inequality:
	$$ \sup\limits_{m\in \mathbb{N}}\| f_m\|_{L_\infty(\HH^n)}\leq \|f\|_{L_\infty(\HH^n)}+\frac{1}{|B(0,2m)\backslash B(0,m)|}\int_{\HH^n}|f(g)|dg\leq 2\|f\|_{L_\infty(\HH^n)}<\infty.$$
	To show $(b)$, we observe from the construction of $f_m$ that $f_m+c_m=f$ on $B(0,m/2)$, which implies that
$f_m+c_m\rightarrow f$ uniformly on any compact set. Note that $\{c_m\}_{m\geq 1}$ is a bounded sequence since $f\in L_\infty(\HH^n)$. Therefore, one can assume $c_m\rightarrow c$ without loss of generality by passing to a subsequence if needed. Thus, $f_m\rightarrow f-c$ uniformly on any compact subset of $\HH^n$. $(c)$ is a direct consequence of the dominated convergence theorem combined with $(a)$ and $(b)$. This ends the proof of the second statement, and therefore Lemma \ref{density2}.
\end{proof}

\section{Proof of upper bound}\label{sufficiencysection}
\setcounter{equation}{0}
At the beginning of this section, we recall the definition and some properties of double operator integrals. General surveys of double operator integration include \cite{Peller-moi-2016,SkripkaTomskova}. In particular our approach is similar to that of \cite{PS-crelle} (see also \cite{PSW}).
\begin{defi}
Let $H$ be a complex separable Hilbert space. Let $D_0$ and $D_1$ be self-adjoint (potentially unbounded) operators on $H$ and let $E^0$ and $E^1$ be the associated spectral measures.
\end{defi}

Note that for any $x,y\in\mathcal{L}_2(H)$, the measure
\begin{align*}
(\lambda,\mu)\rightarrow {\rm Tr}(xdE^{0}(\lambda)ydE^{1}(\mu))
\end{align*}
is a countably additive complex-valued measure on $\mathbb{R}^{2}$. For any $\phi\in L_\infty(\mathbb{R}^2)$, $\phi$ is said to be integrable if there is an operator $T_{\phi}^{D_0,D_1}\in \mathcal{L}_\infty (\mathcal{L}_2 (H))$ such that for all $x,y\in\mathcal{L}_{2}(H)$,
\begin{align*}
{\rm Tr}(xT_{\phi}^{D_0,D_1}y)=\int_{\mathbb{R}^{2}}\phi(\lambda,\mu){\rm Tr}(xdE^{0}(\lambda)ydE^{1}(\mu)).
\end{align*}
Next for any $A\in\mathcal{L}_2(H)$, the double operator integral is defined by
\begin{align}\label{double integral}
T_{\phi}^{D_0,D_1}(A):=\int_{\mathbb{R}^{2}}\phi(\lambda,\mu)dE^0(\lambda)AdE^1(\mu).
\end{align}
The definition of double operator integral can be extended to $\mathcal{L}_{p}$ and $\mathcal{L}_{p,\infty}$ for $1<p<\infty$ in the following way: suppose that
$$\|T_{\phi}^{D_0,D_1}(A)\|_{\mathcal{L}_{p}} \leq C\|A\|_{\mathcal{L}_{p}},\ {\rm for}\ {\rm all}\ A \in \mathcal{L}_{2}\cap \mathcal{L}_p.$$
Then $T_{\phi}^{D_0,D_1}$ admits a unique extension to $\mathcal{L}_{p}$.

Recall that if $1<p_0<p<p_1< \infty$, then $\mathcal{L}_{p_0} \subset \mathcal{L}_{p,\infty} \subset \mathcal{L}_{p_1}$ and $\mathcal{L}_{p,\infty}$ is an interpolation space of the couple
$(\mathcal{L}_{p_0},\mathcal{L}_{p_1}).$ It follows that if $T_{\phi}^{D_0,D_1}$ is bounded on $\mathcal{L}_{p_0}$ and $\mathcal{L}_{p_1},$ then $T_{\phi}^{D_0,D_1}$ admits a unique extension to $\mathcal{L}_{p,\infty}$. The extension of $T_{\phi}^{D_0,D_1}$ to $\mathcal{L}_\infty$ is more difficult, for details we refer to \cite{DDSZ}.

Similarly, if $(\mathcal{M},\tau)$ is a semifinite von Neumann algebra and $D_0$ and $D_1$ are self-adjoint operators affiliated with $\mathcal{M},$
then $T^{D_0,D_1}_{\phi}$ may be defined on the noncommutative $L_p$ spaces $L_{p}(\mathcal{M},\tau)$ and weak $L_p$ space $L_{p,\infty}(\mathcal{M},\tau)$ in a similar way.
For further details, see \cite{PSW} and \cite{DDSZ} which discuss this issue.

We will use the following properties frequently:
\begin{enumerate}
\item for every $\varphi,\phi\in L_\infty(\mathbb{R}^{2}),$ we have
$$T_{\alpha \varphi+\beta\phi}^{D_0,D_1}=\alpha T_{\varphi}^{D_0,D_1}+\beta T_{\phi}^{D_0,D_1},\quad \alpha,\beta\in\mathbb{C}.$$
\item for every $\varphi,\phi\in L_\infty(\mathbb{R}^{2}),$ we have
$$T_{\varphi\phi}^{D_0,D_1}=T_{\varphi}^{D_0,D_1}T_{\phi}^{D_0,D_1}.$$
\end{enumerate}

\begin{lem}\label{fraction schur lemma}
Let $B\geq 0$ be a (possibly unbounded) linear operator on a Hilbert space $H$ with $ ker(B)=0,$ then for any $1<p<\infty,$ there exists a constant $c_p<\infty$ such that
$$\Big\|T^{B,B}_{\frac{\lambda}{\lambda+\mu}}\Big\|_{\mathcal{L}_{p,\infty}\to\mathcal{L}_{p,\infty}}\leq c_p.$$	
\end{lem}
\begin{proof}
Recalling the formula $\min\{\lambda,\mu\} = \frac{1}{2}(\lambda+\mu - |\lambda-\mu|),$ we have
$$\frac{2\lambda}{\lambda+\mu}=1+{\rm sgn}(\lambda-\mu)\cdot\frac{|\lambda-\mu|}{\lambda+\mu}=1+{\rm sgn}(\lambda-\mu)-2{\rm sgn}(\lambda-\mu)\cdot\frac{{\rm min}(\lambda,\mu)}{\lambda+\mu}.$$
Therefore,
$$T^{B,B}_{\frac{\lambda}{\lambda+\mu}}=\frac12{\rm id}+\frac12 T^{B,B}_{{\rm sgn}(\lambda-\mu)}-T^{B,B}_{{\rm sgn}(\lambda-\mu)}\circ T^{B,B}_{\frac{\min(\lambda,\mu)}{\lambda+\mu}}.$$
It suffices to show that the operators
$$T^{B,B}_{{\rm sgn}(\lambda-\mu)}:\mathcal{L}_{p,\infty}\to\mathcal{L}_{p,\infty},\quad T^{B,B}_{\frac{\min(\lambda,\mu)}{\lambda+\mu}}:\mathcal{L}_{p,\infty}\to\mathcal{L}_{p,\infty}$$
are bounded by constants which depend only on $p.$ The first of those operators is a triangular truncation (known to be bounded in $\mathcal{L}_{p,\infty}$ by a constant which depends only on $p$, see \cite{Arazy1978}). The second operator is bounded in $\mathcal{L}_1$ and in $\mathcal{L}_{\infty}$ by \cite[Lemmas 2.1\& 3.1]{CPSZ1}. By interpolation, it is also bounded in $\mathcal{L}_{p,\infty}.$
\end{proof}

\begin{lem}\label{doi bounded above and below}
Let $B$ be a self-adjoint linear operator on a Hilbert space $H$ and there exists a constant $c>0$ such that $c\leq B\leq c^{-1}$. If $A\in\mathcal{L}_{p,\infty}$ for some $1<p<\infty$, then we have
$$[B,A]B^{-1}=T^{B,B}_{\frac{\lambda}{\lambda+\mu}}(B^{-1}[B^2,A]B^{-1}).$$	
\end{lem}
\begin{proof} Note that
$$[B,A]B^{-1}=T^{B,B}_{\frac{\lambda-\mu}{\mu}}(A),\quad B^{-1}[B^2,A]B^{-1}=T^{B,B}_{\frac{\lambda^2-\mu^2}{\lambda\mu}}(A).$$
Clearly,
$$\frac{\lambda-\mu}{\mu}=\frac{\lambda}{\lambda+\mu}\cdot\frac{\lambda^2-\mu^2}{\lambda\mu}.$$
By Lemma \ref{fraction schur lemma}, the double operator integral with symbol $\frac{\lambda}{\lambda+\mu}$ is bounded on $\mathcal{L}_{p,\infty}$. Thus,
$$[B,A]B^{-1}=T^{B,B}_{\frac{\lambda-\mu}{\mu}}(A)=T^{B,B}_{\frac{\lambda}{\lambda+\mu}}\Big(T^{B,B}_{\frac{\lambda^2-\mu^2}{\lambda\mu}}(A)\Big)=T^{B,B}_{\frac{\lambda}{\lambda+\mu}}(B^{-1}[B^2,A]B^{-1}).$$
This ends the proof of Lemma \ref{doi bounded above and below}.
\end{proof}

\begin{lem}\label{surprisingly_technical_lemma}
Let $B\geq 0$ be a (potentially unbounded) linear operator on a Hilbert space $H$ with $ker(B)=0$, and let $A$ be a bounded operator on $H$. Suppose that $1<p<\infty$ and
\begin{enumerate}
\item $B^{-1}[B^2,A]B^{-1}\in\mathcal{L}_{p,\infty};$
\item $[B,A]B^{-1}\in\mathcal{L}_{\infty};$
\item $AB^{-1}\in\mathcal{L}_{p,\infty}$ or $B^{-1}A\in\mathcal{L}_{p,\infty}.$
\end{enumerate}
Then we have $[B,A]B^{-1} \in \mathcal{L}_{p,\infty}$ and there exists a constant $c_p>0$ such that
$$\|[B,A]B^{-1}\|_{\mathcal{L}_{p,\infty}} \leq c_p\|B^{-1}[B^2,A]B^{-1}\|_{\mathcal{L}_{p,\infty}}.$$
\end{lem}
\begin{proof} Let $p_n=\chi_{(\frac1n,n)}(B).$ Consider Hilbert space $p_n(H)$ and the operators $A_n=p_nAp_n$ and $B_n=Bp_n$ on $p_n(H).$ Clearly, $\frac1n \leq B_n\leq n$ on $p_n(H).$ If $AB^{-1}\in\mathcal{L}_{p,\infty},$ then it follows that $Ap_n\in\mathcal{L}_{p,\infty}.$ Similarly, if $B^{-1}A\in\mathcal{L}_{p,\infty},$ then it follows that $p_n A\in\mathcal{L}_{p,\infty}.$ In either case, one has $A_n\in\mathcal{L}_{p,\infty}.$ Note that
$$p_n\cdot [B,A]B^{-1}\cdot p_n=[B_n^{-1},A_n]B_n^{-1},\quad p_n\cdot B^{-1}[B^2,A]B^{-1}\cdot p_n=B_n^{-1}[B_n^2,A_n]B_n^{-1}.$$
By Lemma \ref{doi bounded above and below}, we have
$$[B_n,A_n]B_n^{-1}=T^{B_n,B_n}_{\frac{\lambda}{\lambda+\mu}}(B_n^{-1}[B_n^2,A_n]B_n^{-1}).$$

Thus,
$$p_n\cdot [B,A]B^{-1}\cdot p_n=T^{B_n,B_n}_{\frac{\lambda}{\lambda+\mu}}(p_n\cdot B^{-1}[B^2,A]B^{-1}\cdot p_n)=p_n\cdot T^{B,B}_{\frac{\lambda}{\lambda+\mu}}(B^{-1}[B^2,A]B^{-1})\cdot p_n.$$
By the definition of $p_n$, we have that $p_n\uparrow 1$ and therefore, appealing to the assumption (2), we see that
$$p_n\cdot [B,A]B^{-1}\cdot p_n\to [B,A]B^{-1},\quad n\to\infty,$$
in weak operator topology. Similarly,
$$p_n\cdot T^{B,B}_{\frac{\lambda}{\lambda+\mu}}(B^{-1}[B^2,A]B^{-1})\cdot p_n\to T^{B,B}_{\frac{\lambda}{\lambda+\mu}}(B^{-1}[B^2,A]B^{-1}),\quad n\to\infty,$$
in weak operator topology. Thus,
$$[B,A]B^{-1}=T^{B,B}_{\frac{\lambda}{\lambda+\mu}}(B^{-1}[B^2,A]B^{-1}).$$
The assertion follows from Lemma \ref{fraction schur lemma}.
Observe that the assumption that $ker(B)=0$ is used in order to ensure the well-definedness of the double operator integral with symbol $\frac{\lambda}{\lambda+\mu}.$
\end{proof}

Now we give a proof of the upper bound in Theorem \ref{main}.
\begin{prop}\label{sufficiency theorem}
 Suppose that $f\in  L_{\infty}(\Heis^n)\cap \dot{W}^{1,2n+2}(\HH^n)$ and $\ell\in \{1,2,\cdots,2n\}$, then there exists a positive constant $C_n>0$ such that
\begin{align}\label{rlmf}
\|[R_{\ell},M_f]\|_{\mathcal{L}_{2n+2,\infty}}\leq C_n\|f\|_{\dot{W}^{1,2n+2}(\HH^n)}.
\end{align}
\end{prop}
\begin{proof}
We first show the conclusion holds for $f \in C^\infty_c(\Heis^n)$. To begin with, for any $\ell\in\{1,2,\cdots,2n\}$, we decompose $[R_\ell,M_f]$ by Leibniz's rule into two parts as follows.
\begin{align}\label{Rieszdecompose}
    [R_\ell,M_f]&= [X_\ell(-\Delta_{\HH})^{-\frac12},M_f]\nonumber\\
                         &= [X_\ell,M_f](-\Delta_{\HH})^{-\frac12}+X_\ell[(-\Delta_{\HH})^{-\frac12},M_f]\nonumber\\
                         &= M_{X_\ell f}(-\Delta_{\HH})^{-\frac12}-R_\ell[(-\Delta_{\HH})^{\frac12},M_f](-\Delta_{\HH})^{-\frac12},
\end{align}
where in the last step we applied the following commutator formula:
\begin{align}\label{commutator}
    [A^{-1},B] = -A^{-1}[A,B]A^{-1}.
\end{align}

To estimate the first term, we apply Corollary \ref{specific_cwikel} with $\beta=1$ to conclude that
\begin{equation}\label{specific_cwikel_application}
    \|M_{X_\ell f}(-\Delta_{\HH})^{-\frac12}\|_{\mathcal{L}_{2n+2,\infty}}\lesssim\|X_\ell f\|_{L_{2n+2}(\Heis^n)}.
\end{equation}

To estimate the second term, since Riesz transform $R_{\ell}$ is bounded on $L_{2}(\HH^{n})$, it remains to show that
\begin{align}
\|[(-\Delta_{\HH})^{\frac12},M_f](-\Delta_{\HH})^{-\frac12}\|_{\mathcal{L}_{2n+2,\infty}}\lesssim\sum_{j=1}^{\infty}\|X_{j}f\|_{L_{2n+2}(\HH^{n})}.
\end{align}
To this end, we apply Lemma \ref{surprisingly_technical_lemma} to $B = (-\Delta_{\HH})^{\frac12},$ $A= M_f$ and $p=2n+2.$ This yields
\[
    \|[(-\Delta_{\HH})^{\frac12},M_f](-\Delta_{\HH})^{-\frac12}\|_{\mathcal{L}_{2n+2,\infty}} \lesssim\|(-\Delta_{\HH})^{-\frac12}[\Delta_{\HH},M_f](-\Delta_{\HH})^{-\frac12}\|_{\mathcal{L}_{2n+2,\infty}}.
\]

To continue, we apply commutator formula to see that
\begin{align*}
    (-\Delta_{\HH})^{-\frac12}[\Delta_{\HH},M_f](-\Delta_{\HH})^{-\frac12}&= \sum_{k=1}^{2n}(-\Delta_{\HH})^{-\frac12}[X_k^2,M_f](-\Delta_{\HH})^{-\frac12}\\
                         &=\sum_{k=1}^{2n}(-\Delta_{\HH})^{-\frac12}X_k M_{X_kf}(-\Delta_{\HH})^{-\frac12}
                         +\sum_{k=1}^{2n}(-\Delta_{\HH})^{-\frac12}M_{X_kf}X_k(-\Delta_{\HH})^{-\frac12}.
\end{align*}
By the triangle inequality and the ideal property of $\mathcal{L}_{2n+2,\infty},$ it follows that
\begin{align*}
    \|(-\Delta_{\HH})^{-\frac12}&[\Delta_{\HH},M_f](-\Delta_{\HH})^{-\frac12}\|_{\mathcal{L}_{2n+2,\infty}}\\
                           &\lesssim\sum_{k=1}^{2n}\|(-\Delta_{\HH})^{-\frac12}X_k\|_{\mathcal{L}_{\infty}}\| M_{X_kf}(-\Delta_{\HH})^{-\frac12}\|_{\mathcal{L}_{2n+2,\infty}}
                            +\sum_{k=1}^{2n}\|(-\Delta_{\HH})^{-\frac12}M_{X_kf}\|_{\mathcal{L}_{2n+2,\infty}}\|X_k(-\Delta_{\HH})^{-\frac12}\|_{\mathcal{L}_{\infty}}.
\end{align*}
Applying \eqref{specific_cwikel_application} leads us to
\[
    \|(-\Delta_{\HH})^{-\frac12}[\Delta_{\HH},M_f](-\Delta_{\HH})^{-\frac12}\|_{\mathcal{L}_{2n+2,\infty}} \lesssim\sum_{k=1}^{2n} \|X_k f\|_{L_{2n+2}(\Heis^n)}.
\]

Now we apply an approximation argument to remove assumption that $f \in C^\infty_c(\Heis^n)$. To this end, we suppose $f\in  L_{\infty}(\Heis^n)\cap \dot{W}^{1,2n+2}(\HH^n)$ and let $\{f_m\}_{m\geq 1}$ be the sequence chosen in Lemma \ref{density2}, then $f_m\in C_c^\infty(\HH^n)$ for $m\geq 1$ and $\{f_m\}_{m\geq 1}$ is a Cauchy sequence on $\dot{W}^{1,2n+2}(\HH^n)$. Since we have already shown that \eqref{rlmf} holds for $f\in C_c^\infty(\HH^n)$, one has
\begin{align*}
\|[R_\ell,M_{f_{m_1}}]-[R_\ell,M_{f_{m_2}}]\|_{\mathcal{L}_{2n+2,\infty}}=\|[R_\ell,M_{f_{m_1}-f_{m_2}}]\|_{\mathcal{L}_{2n+2,\infty}}\lesssim \|f_{m_1}-f_{m_2}\|_{\dot{W}^{1,2n+2}(\HH^n)},\ \ m_1,m_2\geq 1,
\end{align*}
where the implicit constant is independent of $m_1$ and $m_2$. This implies that $\{[R_\ell,M_{f_{m}}]\}_{m}$ is a Cauchy sequence on $\mathcal{L}_{2n+2,\infty}$, and therefore it converges to some $A\in \mathcal{L}_{2n+2,\infty}$. In particular, $\{[R_\ell,M_{f_{m}}]\}_{m}\rightarrow A$ in the strong operator topology. On the other hand, by (c) in Lemma \ref{density2}, we see that $M_{f_m}\rightarrow M_{f-c}$ in the strong operator topology. Therefore, $[R_\ell,M_{f_m}]\rightarrow [R_\ell,M_f]$ in the strong operator topology. By uniqueness of the limit, $A=[R_\ell,M_f]$, which implies that  $[R_\ell,M_{f_m}]\rightarrow [R_\ell,M_f]$ in $\mathcal{L}_{2n+2,\infty}$. We now have
\begin{align*}
\|[R_\ell,M_f]\|_{\mathcal{L}_{2n+2,\infty}}=\lim_{m\rightarrow \infty}\|[R_\ell,M_{f_m}]\|_{\mathcal{L}_{2n+2,\infty}}\lesssim \limsup_{m\rightarrow \infty}\|f_m\|_{\dot{W}^{1,2n+2}(\HH^n)}=C_n\|f\|_{\dot{W}^{1,2n+2}(\HH^n)}.
\end{align*}
This completes the proof of Proposition \ref{sufficiency theorem}.
\end{proof}

\section{Auxiliary computations for trace theorem}\label{auxiliarysection}
\setcounter{equation}{0}
The Hermite basis $\{h_{\alpha}\}_{\alpha\in \mathbb{Z}_+^n}$ is a particular orthonormal basis for $L_2(\mathbb{R}^n)$ which will be useful here. Recall
that if $\alpha = (\alpha_1,\ldots,\alpha_n)\in \mathbb{Z}_+^n$ is a multi-index, then we denote
$$|\alpha| = |\alpha_1|+|\alpha_2|+\cdots+|\alpha_n|,\quad \alpha = (\alpha_1,\alpha_2,\cdots,\alpha_n).$$
The Hermite basis element $h_{\alpha}$ is defined as
$$h_\alpha(x) = \left(\frac{2^{|\alpha|}}{\pi^n \alpha!}\right)^{1/2}\prod_{j=1}^nH_{\alpha_j}(x_j)\cdot e^{-\frac12|x|^2},\quad x \in \mathbb{R}^n.$$
Here, $H_k$ is the $k$-th Hermite polynomial given by
\[
    H_k(x) = \sum_{j=0}^{\lfloor \frac{k}{2}\rfloor} (-1)^j \frac{k!}{j!(k-2j)!}(2x)^{k-2j},\quad x \in \mathbb{R}.
\]
Equivalently, the sequence $\{H_k\}_{k=0}^\infty$ can be defined by the generating function
$$\sum_{k=0}^\infty \frac{H_k(x)}{k!}t^k = \exp(2xt-t^2),\quad x \in \mathbb{R}.$$
The family $\{h_{\alpha}\}_{\alpha\in \mathbb{Z}_+^n}$ is an orthonormal basis for the Hilbert space $L_2(\mathbb{R}^n)$ and it is with respect to this specific basis that we understand the isomorphism
\[
    L_2(\mathbb{R}^n) \approx \ell_2(\mathbb{Z}_+^n).
\]
For further details we refer to \cite[Section 1.5]{GlimmJaffe}, \cite[Chapter 22]{AbramowitzStegun}, \cite[Section 6.1]{Zworski2012}.

Recall that $p_j := -i\partial_{x_j}$ and $q_j := M_{x_j}$ for $1 \leq j\leq n$ denote the momentum and position operators respectively on $L_2(\mathbb{R}^n).$
In terms of the Hermite basis, we have
\[
    p_jh_{\alpha} = -i\left(\frac{\alpha_j}{2}\right)^{\frac12}h_{\alpha-e_j}+i\left(\frac{\alpha_j+1}{2}\right)^{\frac12}h_{\alpha+e_j}
\]
and
\[
    q_jh_{\alpha} = \left(\frac{\alpha_j}{2}\right)^{\frac12}h_{\alpha-e_j}+\left(\frac{\alpha_j+1}{2}\right)^{\frac12}h_{\alpha+e_j},
\]
see e.g. \cite[Lemma 1.5.3, Proposition 1.5.4]{GlimmJaffe}, \cite[Table 22.8]{AbramowitzStegun}.
Here, $e_j\in \mathbb{Z}_+^n, 1\leq j\leq n$ is the multi-index equal to $1$ in the $j$th position and zero elsewhere. These relations are stated with the understanding that $h_{\alpha-e_j}=0$ if $\alpha_j = 0.$ It follows from these relations that $\{h_{\alpha}\}_{\alpha\in \mathbb{Z}_+^n}$ consists of eigenfunctions for the harmonic oscillator $H = \sum_{j=1}^{n} q_j^2+p_j^2.$
We have
\begin{equation}\label{landau_formula}
    Hh_{\alpha} = (2|\alpha|+n)h_{\alpha},\quad \alpha \in \mathbb{Z}_+^n.
\end{equation}
See e.g. \cite[Equation (6.1.10)]{Zworski2012}.

\begin{defi} Let $n\geq 1.$ Given $\alpha,\beta\in \mathbb{Z}_+^n,$ $E_{\alpha,\beta}$ denotes the projection
\[
    E_{\alpha,\beta}u := h_\alpha\langle u,h_\beta\rangle_{L_2(\mathbb{R}^n)},\quad u\in L_2(\mathbb{R}^n).
\]
\end{defi}
Observe that for all $\alpha,\beta \in \mathbb{Z}^n$ we have
\[
    E_{\alpha,\beta} = E_{\alpha_1,\beta_1}\otimes E_{\alpha_2,\beta_2}\otimes\cdots \otimes E_{\alpha_n,\beta_n},
\]
where the operators $E_{\alpha_1,\beta_1}, E_{\alpha_2,\beta_2},$ etc. are the projections onto the Hermite basis for $L_2(\mathbb{R}),$ and we understand
$L_2(\mathbb{R}^n) = L_2(\mathbb{R})\overline{\otimes} L_2(\mathbb{R}) \overline{\otimes} \cdots \overline{\otimes} L_2(\mathbb{R}).$ With this notation, $H$ may be written as
\[
    H = \sum_{k=0}^\infty (2k+n)\sum_{|\alpha|=k} E_{\alpha,\alpha}.
\]
Similarly,
\[
    H^{-\frac12} = \sum_{k=0}^\infty (2k+n)^{-\frac12} \sum_{|\alpha|=k} E_{\alpha,\alpha}.
\]
and the position and momentum operators are
\[
    p_j = \sum_{\alpha\in \mathbb{Z}_+^n} -i\left(\frac{\alpha_j}{2}\right)^{\frac12}E_{\alpha-e_j,\alpha}+i\left(\frac{\alpha_j+1}{2}\right)^{\frac12}E_{\alpha+e_j,\alpha},\quad 1\leq j\leq n
\]
and
\[
    q_j = \sum_{\alpha\in \mathbb{Z}_+^n} \left(\frac{\alpha_j}{2}\right)^{\frac12}E_{\alpha-e_j,\alpha}+\left(\frac{\alpha_j+1}{2}\right)^{\frac12}E_{\alpha+e_j,\alpha},\quad 1\leq j\leq n.
\]
It follows from these relations that the operators $p_j+iq_j$ and $q_j+ip_j$ are given by
\begin{equation}\label{ladder_relations}
    p_j+iq_j = i\sum_{\alpha \in \mathbb{Z}_+^n} (2\alpha_j+2)^{\frac12}E_{\alpha+e_j,\alpha},\quad  q_j+ip_j = i\sum_{\alpha \in \mathbb{Z}_+^n} (2\alpha_j)^{\frac12}E_{\alpha-e_j,\alpha}.
\end{equation}
The bounded operators $p_jH^{-\frac12}$ and $q_jH^{-\frac12}$ may be written as
\[
    p_jH^{-\frac12} = \sum_{\alpha\in \mathbb{Z}_+^n} -i\left(\frac{\alpha_j}{4|\alpha|+2n}\right)^{\frac12}E_{\alpha-e_j,\alpha}+i\left(\frac{\alpha_j+1}{4|\alpha|+2n}\right)^{\frac12}E_{\alpha+e_j,\alpha},\quad 1\leq j\leq n
\]
and
\[
    q_jH^{-\frac12} = \sum_{\alpha\in \mathbb{Z}_+^n} \left(\frac{\alpha_j}{4|\alpha|+2n}\right)^{\frac12}E_{\alpha-e_j,\alpha}+\left(\frac{\alpha_j+1}{4|\alpha|+2n}\right)^{\frac12}E_{\alpha+e_j,\alpha},\quad 1\leq j\leq n.
\]
Conversely, it is possible to extract each $E_{\alpha,\beta}$ from $\{p_jH^{-\frac12},q_jH^{-\frac12}\}_{j=1}^n,$ in a certain sense described in the next lemma.

For simplicity, we abbreviate
\[
    \textbf{1}  = \begin{pmatrix} 1 \\ 1\end{pmatrix},\quad z = \begin{pmatrix} -1\\ 1\end{pmatrix}
\]
so that $\{{\bf 1},z\}$ is a basis for $\mathbb{C}^2$ as a vector space.
\begin{lem}\label{homogeneous cast algebra}
For all $\xi \in \mathbb{C}^2$ and $\alpha,\beta \in \mathbb{Z}_+^n,$ the operator $E_{\alpha,\beta}\otimes \xi$ belongs to the $C^*$-subalgebra
of $\mathcal{B}(L_2(\mathbb{R}^n))\overline{\otimes} \mathbb{C}^2$ generated by  $\{p_kH^{-\frac12}\otimes {\bf 1},q_kH^{-\frac12} \otimes z\}_{k=1}^{n}.$
\end{lem}
\begin{proof}
Denote by $A$ the $C^*$-subalgebra of $\mathcal{B}(L_2(\mathbb{R}^n))\overline{\otimes} \mathbb{C}^2$ generated by $\{p_kH^{-\frac12}\otimes {\bf 1},q_kH^{-\frac12}\otimes z\}_{k=1}^n.$ Since $\{{\bf 1},z\}$ is a basis for $\mathbb{C}^2$ as a vector space, it suffices to show that $E_{\alpha,\beta}\otimes {\bf 1} \in A$ and $E_{\alpha,\beta}\otimes z\in A$.

For all $1\leq k\leq n,$ note that the relation $[p_k,q_k] =-i$ implies that
\begin{align*}
    (H^{-\frac12}p_k\otimes {\bf 1})\cdot (q_kH^{-\frac12}\otimes z)-(H^{-\frac12}q_k\otimes {\bf 1})\cdot(p_kH^{-\frac12}\otimes z)&=H^{-\frac12}[p_k,q_k]H^{-\frac12}\otimes z\\&=-iH^{-1}\otimes z.
\end{align*}
It follows that $H^{-1}\otimes z\in A.$ Thus, $H^{-2}\otimes \textbf{1}=(H^{-1}\otimes z)\cdot(H^{-1}\otimes z)\in A.$

Let $m$ be an eigenvalue of $H,$ and denote by $P_m$ the corresponding eigen-projection of $H.$ That is,
\[
    P_m := \sum_{2|\alpha|+n=m} E_{\alpha,\alpha}.
\]
Since the spectrum of $H$ is discrete, we may choose $\phi_m\in C^{\infty}_c(\mathbb{R})$ such that $P_m=\phi_m(H^{-2}).$ Since $A$
is closed under continuous functional calculus, we have
$$P_m\otimes \textbf{1}=\phi_m(H^{-2}\otimes \textbf{1})\in A.$$
Thus, for all $1\leq k\leq n$ we have
\begin{align*}
    P_m\otimes z&=m(H^{-1}\otimes z)\cdot (P_m\otimes \textbf{1})\in A,\\
    q_kP_m\otimes z&=m^{\frac12}(q_kH^{-\frac12}\otimes z)\cdot (P_m\otimes {\bf 1})\in A,&
    p_kP_m\otimes z&=m^{\frac12}(p_kH^{-\frac12}\otimes {\bf 1})\cdot (P_m\otimes z)\in A.
\end{align*}
Therefore,
$$|(p_k+iq_k)P_m|^2\otimes \textbf{1}\in A.$$


From \eqref{ladder_relations}, we have
\[
    (p_k+iq_k)(p_k-iq_k)P_m  = \sum_{2|\alpha|+n=m} 2\alpha_k E_{\alpha,\alpha}.
\]
Therefore,
\[
    A \ni |(p_k+iq_k)P_m|^2\otimes \textbf{1} = P_m(p_k-iq_k)(p_k+iq_k)P_m\otimes \textbf{1} = \sum_{2|\alpha|+n=m} 2\alpha_k E_{\alpha,\alpha}\otimes \textbf{1}.
\]
Since each $|(p_k+iq_k)P_m|^2\otimes \textbf{1}$ belongs to the $C^*$-algebra $A,$ and has discrete spectrum it follows that the spectral projections of
$|(p_k+iq_k)P_m|^2\otimes \textbf{1}$ belong to $A.$ Hence, for all $j\geq 0$ and $1\leq k\leq n$ we have
\[
    P_{m,k,j}\otimes \textbf{1} := \sum_{2|\alpha|+n=m,\; \alpha_k=j} E_{\alpha,\alpha}\otimes \textbf{1} \in A.
\]
Therefore,
\begin{align}\label{belongA1}
    E_{\alpha,\alpha}\otimes \textbf{1} = \prod_{k=1}^n P_{2|\alpha|+n,k,\alpha_k}\otimes \textbf{1} \in A.
\end{align}
Besides, note that for all $1\leq k\leq n$  we have
\begin{align}\label{belongA2}
&q_kP_{2|\alpha|+n}\otimes {\bf 1}={(2|\alpha|+n)}^{\frac{1}{2}}(q_k H^{-\frac{1}{2}}\otimes z)(P_{2|\alpha|+n}\otimes z)\in A,
\end{align}
\begin{align}\label{belongA3}
&p_kP_{2|\alpha|+n}\otimes {\bf 1}={(2|\alpha|+n)}^{\frac{1}{2}}(p_k H^{-\frac{1}{2}}\otimes {\bf 1})(P_{2|\alpha|+n}\otimes {\bf 1})\in A.
\end{align}
Combining the inclusions \eqref{belongA1}, \eqref{belongA2} and \eqref{belongA3}, we see that for all $1\leq k\leq n$,
$$(q_k+ip_k)E_{\alpha,\alpha}\otimes {\bf 1}=((q_k+ip_k)P_{2|\alpha|+n}\otimes \textbf{1})\cdot(E_{\alpha,\alpha}\otimes {\bf 1})\in A.$$

However, from \eqref{ladder_relations} we have
$$(q_k+ip_k)E_{\alpha,\alpha}=\sqrt{2\alpha_k}E_{\alpha+e_k,\alpha}.$$
Thus,
$E_{\alpha+e_k,\alpha}\otimes \textbf{1}\in A.$
Taking the adjoint, we obtain
$E_{\alpha-e_k,\alpha}\otimes \textbf{1}\in A.$
Proceding in this manner, for every $\alpha,\beta \in \mathbb{Z}_+^d$, we obtain
$E_{\beta,\alpha}\otimes \textbf{1}\in A$
and
$$E_{\beta,\alpha}\otimes z=(E_{\beta,\alpha}\otimes \textbf{1})\cdot (P_{2|\alpha|+n}\otimes z)\in A.$$ This completes the proof.
\end{proof}

In the following lemma, we will denote by ${\rm Alg}(S)$ the unital $*$-subalgebra generated by a family
of operators $S\subset \mathcal{B}(L_2(\mathbb{R}^n)\overline{\otimes} \mathbb{C}^2)$. In this notation, Lemma \ref{homogeneous cast algebra} states that for every $\xi\in \mathbb{C}^2$,
every $E_{\alpha,\beta}\otimes \xi$ belongs to the operator-norm closure of
\[
    {\rm Alg}\big(\{p_kH^{-\frac12}\otimes {\bf 1},\,q_kH^{-\frac12}\otimes z\}_{k=1}^n\big)\subset \mathcal{B}(L_2(\mathbb{R}^n)\overline{\otimes} \mathbb{C}^2).
\]
\begin{lem}\label{compact commutator lemma} Let $f\in C_c^\infty(\Heis^n)$.
For every $\alpha,\beta\in\mathbb{Z}^n_+$ and $\xi \in \mathbb{C}^2,$ the commutator
$$[\pi_{{\rm red}}(E_{\alpha,\beta}\otimes \xi),M_f] \in \mathcal{B}(L_2(\Heis^n))$$
is compact.
\end{lem}
\begin{proof}
Fix $\alpha,\beta\in\mathbb{Z}^n_+.$   and $\xi \in \mathbb{C}^2$. By Lemma \ref{homogeneous cast algebra}, there exists a sequence
$$\{x_m\}_{m\geq1}\subset {\rm Alg}\big(\{p_kH^{-\frac12}\otimes {\bf 1},\ q_kH^{-\frac12}\otimes z\}_{k=1}^n\big)$$	
such that $x_m\to E_{\alpha,\beta}\otimes \xi$ in the operator norm. Since $\pi_{{\rm red}}$ is a $*$-homomorphism, $\pi_{\red}(x_m)$ converges to $\pi_{\red}(E_{\alpha,\beta}\otimes \xi)$ in the norm of $\mathcal{B}(L_2(\HH^n)).$ This implies that
$$[\pi_{{\rm red}}(x_m),M_f]\to [\pi_{{\rm red}}(E_{\alpha,\beta}\otimes \xi),M_f]$$
in the uniform norm.

By Proposition \ref{sufficiency theorem}, for any $1\leq j\leq n$, the operators
$$[R_j,M_f],\quad [R_{n+j},M_f],\quad [R_j^{\ast},M_f],\quad [R_{n+j}^{\ast},M_f]$$
are compact on $L_2(\Heis^n).$
Recall that $\pi_{\red}$ is determined by the relations
\[
    \pi_{\red}(p_jH^{-\frac12}\otimes {\bf 1}) = R_j,\quad \pi_{\red}(q_jH^{-\frac12}\otimes z) = R_{j+n},\quad 1\leq j\leq n.
\]
Since $\pi_{\red}$ is injective, we have
$${\rm Alg}(\{R_j\}_{j=1}^{2n})=\pi_{{\rm red}}\Big({\rm Alg}\big(\{p_kH^{-\frac12}\otimes {\bf 1},\,q_kH^{-\frac12}\otimes z\}_{k=1}^n\big)\Big).$$


By the linearity and the Leibniz's rule, for all $m\geq 0$ we see that
$[\pi_{{\rm red}}(x_m),M_f]$ is compact in $L_2(\HH^n)$.
Hence, $[\pi_{{\rm red}}(E_{\alpha,\beta}\otimes \xi),M_f]$ is a limit (in the operator norm) of a sequence of compact operators. Thus, for all $\xi \in  \mathbb{C}^2,$
\[
    [\pi_{\red}(E_{\alpha,\beta}\otimes \xi),M_f]
\]
is also compact in $L_2(\HH^n).$
\end{proof}

In the following lemma we frequently use the elementary fact that if $A$ is a compact linear operator and $B \in \mathcal{L}_{p,\infty},$ then $AB \in (\mathcal{L}_{p,\infty})_0.$ To see this, select a sequence $\{A_n\}_{n=0}^\infty$ of finite rank operators which approximate $A$ in the uniform norm. It follows that $\{A_nB\}_{n=0}^\infty$ is a sequence of finite
rank operators approximating $AB$ in the $\mathcal{L}_{p,\infty}$ quasi-norm, and hence $AB\in (\mathcal{L}_{p,\infty})_0.$

\begin{lem}\label{t and units commutator lemma}
    For every $f \in C^\infty_c(\Heis^n)$, $\xi \in \mathbb{C}^2$ and $\alpha,\beta \in \mathbb{Z}_+^n,$ we have
    \begin{equation*}
        M_f|T|^{-\frac12}\pi_{{\rm red}}(E_{\alpha,\beta}\otimes \xi)-\pi_{{\rm red}}(E_{\alpha,0}\otimes {\bf 1})\cdot M_f|T|^{-\frac12}\pi_{{\rm red}}(E_{0,0}\otimes \xi)\cdot \pi_{{\rm red}}(E_{0,\beta}\otimes {\bf 1})\in(\mathcal{L}_{2n+2,\infty})_0.
    \end{equation*}
\end{lem}
\begin{proof}
    By multiplying on the right by $\pi_{{\rm red}}(1\otimes \xi),$ it suffices to consider the case $\xi=\textbf{1}.$ Note that $E_{0,0}E_{0,\beta} = E_{0,\beta}$. The statement of the lemma is equivalent to
	$$M_f|T|^{-\frac12}\pi_{{\rm red}}(E_{\alpha,\beta}\otimes \textbf{1})-\pi_{{\rm red}}(E_{\alpha,0}\otimes \textbf{1})M_f|T|^{-\frac12}\pi_{{\rm red}}(E_{0,\beta}\otimes \textbf{1})\in(\mathcal{L}_{2n+2,\infty})_0.$$
	From \eqref{T_in_terms_of_pired} and \eqref{landau_formula} we have
	\begin{align*}
        (-\Delta_{\HH})^{-\frac12}\pi_{{\rm red}}(E_{\alpha,\beta}\otimes \textbf{1}) &= |T|^{-\frac12}\pi_{\red}(H^{-\frac12}\otimes \textbf{1})\pi_{{\rm red}}(E_{\alpha,\beta}\otimes \textbf{1})\\
                                                                          &= |T|^{-\frac12} \pi_{\red}(H^{-\frac12}E_{\alpha,\beta}\otimes \textbf{1})\\
                                                                          &= (n+2|\alpha|)^{-\frac12}|T|^{-\frac12}\pi_{{\rm red}}(E_{\alpha,\beta}\otimes \textbf{1}).
	\end{align*}
	In particular,
	$$|T|^{-\frac12}\pi_{{\rm red}}(E_{0,\beta}\otimes \textbf{1})=n^{\frac12}(-\Delta_{\HH})^{-\frac12}\pi_{{\rm red}}(E_{0,\beta}\otimes \textbf{1}).$$
	Hence, it suffices to verify that
	\[
        (n+2|\alpha|)^{\frac12}\cdot M_f(-\Delta_{\HH})^{-\frac12}\pi_{{\rm red}}(E_{\alpha,\beta}\otimes \textbf{1})
                               -n^{\frac12}\cdot \pi_{{\rm red}}(E_{\alpha,0}\otimes \textbf{1})M_f(-\Delta_{\HH})^{-\frac12}\pi_{{\rm red}}(E_{0,\beta}\otimes \textbf{1})\in(\mathcal{L}_{2n+2,\infty})_0.
	\]
	Select $f_1,f_2 \in C^\infty_c(\Heis^n)$ such that $f=f_1f_2.$  Clearly,
	\begin{align*}
        \pi_{{\rm red}}(E_{\alpha,0}\otimes \textbf{1})M_f(-\Delta_{\HH})^{-\frac12}&=\pi_{{\rm red}}(E_{\alpha,0}\otimes \textbf{1})M_{f_1}M_{f_2}(-\Delta_{\HH})^{-\frac12}\\
                                     &=[\pi_{{\rm red}}(E_{\alpha,0}\otimes \textbf{1}),M_{f_1}]\cdot M_{f_2}(-\Delta_{\HH})^{-\frac12}+M_{f_1}\pi_{{\rm red}}(E_{\alpha,0}\otimes \textbf{1})\cdot [M_{f_2},(-\Delta_{\HH})^{-\frac12}]\\
                                     &\quad+M_{f_1}\pi_{{\rm red}}(E_{\alpha,0}\otimes \textbf{1})\cdot (-\Delta_{\HH})^{-\frac12}M_{f_2}.
	\end{align*}
    By Lemma \ref{compact commutator lemma}, the commutator
	$$[\pi_{{\rm red}}(E_{\alpha,0}\otimes \textbf{1}),M_{f_1}]$$
	is compact, and since $f_2 \in C^\infty_c(\HH^n)$ it follows from Corollary \ref{specific_cwikel} that $M_{f_2}(-\Delta_{\HH})^{-\frac12} \in \mathcal{L}_{2n+2,\infty}$ and hence,
	$$[\pi_{{\rm red}}(E_{\alpha,0}\otimes \textbf{1}),M_{f_1}]\cdot M_{f_2}(-\Delta_{\HH})^{-\frac12}\in(\mathcal{L}_{2n+2,\infty})_0.$$
	From Lemma \ref{cwikel-like commutator lemma}, we also have
	$$[M_{f_2},(-\Delta_{\HH})^{-\frac12}]\in(\mathcal{L}_{2n+2,\infty})_0.$$
	Therefore,
	$$M_{f_1}\pi_{{\rm red}}(E_{\alpha,0}\otimes \textbf{1})\cdot[M_{f_2},(-\Delta_{\HH})^{-\frac12}]\in(\mathcal{L}_{2n+2,\infty})_0.$$
	Thus,
	$$\pi_{{\rm red}}(E_{\alpha,0}\otimes \textbf{1})M_f(-\Delta_{\HH})^{-\frac12}\in M_{f_1}\pi_{{\rm red}}(E_{\alpha,0}\otimes \textbf{1})\cdot (-\Delta_{\HH})^{-\frac12}M_{f_2}+(\mathcal{L}_{2n+2,\infty})_0.$$
	Note that
	$$\pi_{{\rm red}}(E_{\alpha,0}\otimes \textbf{1})\cdot (-\Delta_{\HH})^{-\frac12}=\left(\frac{n+2|\alpha|}{n}\right)^{\frac12}\cdot (-\Delta_{\HH})^{-\frac12}\cdot\pi_{{\rm red}}(E_{\alpha,0}\otimes \textbf{1}).$$
	Hence,
	\begin{align*}
        \pi_{{\rm red}}(E_{\alpha,0}\otimes \textbf{1})M_f(-\Delta_{\HH})^{-\frac12}\pi_{{\rm red}}(E_{0,\beta}\otimes \textbf{1})\in \left(\frac{n+2|\alpha|}{n}\right)^{\frac12}\cdot M_{f_1}(-\Delta_{\HH})^{-\frac12}\pi_{{\rm red}}(E_{\alpha,0}\otimes \textbf{1})\cdot M_{f_2}\pi_{{\rm red}}(E_{0,\beta}\otimes \textbf{1})+(\mathcal{L}_{2n+2,\infty})_0.
	\end{align*}
	
	We have
	\begin{align*}
        M_{f_1}&(-\Delta_{\HH})^{-\frac12}\pi_{{\rm red}}(E_{\alpha,0}\otimes \textbf{1})\cdot M_{f_2}\pi_{{\rm red}}(E_{0,\beta}\otimes \textbf{1})\\
               &\in M_{f_1}(-\Delta_{\HH})^{-\frac12}\cdot [\pi_{{\rm red}}(E_{\alpha,0}\otimes \textbf{1}),M_{f_2}]\cdot\pi_{{\rm red}}(E_{0,\beta}\otimes \textbf{1})+M_f(-\Delta_{\HH})^{-\frac12}\pi_{{\rm red}}(E_{\alpha,\beta}\otimes \textbf{1})+(\mathcal{L}_{2n+2,\infty})_0.
    \end{align*}
    By Lemma \ref{compact commutator lemma}, the commutator is compact, and since $M_{f_1}(-\Delta_{\HH})^{-\frac12}$ belongs to $\mathcal{L}_{2n+2,\infty},$ the result follows.
\end{proof}
	
\section{A version of trace theorem suitable for the proof of necessity direction}\label{tracelemmasection}
A key part of the proof of the trace formula in \cite{LMSZ} and the related papers \cite{MSX2018,MSX2019} was a general statement about the traces of homogeneous Fourier multipliers.
For example, a consequence of Connes' trace formula \cite{Connes1998} is that if $f \in C^\infty_c(\mathbb{R}^d),$ $d\geq 1$ and $g \in C^\infty(S^{d-1}),$ then for every continuous normalised trace
$\varphi$ on $\mathcal{L}_{1,\infty}$ we have
\begin{equation}\label{connes_trace_formula_commutative_example}
    \varphi(M_fg(-i\nabla(-\Delta)^{-\frac12})(1-\Delta)^{-\frac{d}{2}}) = \frac{1}{d(2\pi)^d}\int_{\mathbb{R}^d\times S^{d-1}} f(s)g(\xi)\,dsd\xi.
\end{equation}
In \cite{Dao1} it was explained how the above formula could be proved very easily as a consequence of its symmetry and continuity properties. Specifically,
\eqref{connes_trace_formula_commutative_example} was proved in \cite[Theorem 5.1]{Dao1} by demonstrating that if $\varphi$ is a continuous trace then the bilinear functional
\[
    (f,g)\mapsto \varphi(M_fg(-i\nabla(-\Delta)^{-\frac12})(1-\Delta)^{-\frac{d}{2}})
\]
is necessarily a constant multiple of the Lebesgue integral of $f\otimes g$ on $\mathbb{R}^d\times S^{d-1}$ as a consequence of its invariance under isometries.

We seek a similar result with an analogous proof, but now for the Heisenberg group $\Heis^n.$ In place of the operator $M_fg(-i\nabla (-\Delta)^{-1/2})(1-\Delta)^{-\frac{d}{2}}$, we consider operators of the form
\[
    A(x,f)=M_f\pi_{{\rm red}}(x)|T|^{-\frac12} \in \mathcal{B}(L_2(\Heis^n)).
\]
Here, $x \in \mathcal{B}(L_2(\mathbb{R}^n))\overline{\otimes} \mathbb{C}^2$ and $f \in L_{2n+2}(\Heis^n).$
With this notation, Lemma \ref{homogeneous cwikel lemma} states that there is a constant $C_n>0$ such that
\[
    \|A(x,f)\|_{\mathcal{L}_{2n+2,\infty}}\leq C_n\|f\|_{L_{2n+2}(\HH^n)}\|x\|_{L_{2n+2}(\mathcal{B}(L_2(\mathbb{R}^n))\bar{\otimes}\mathbb{C}^2,{\rm Tr}\otimes\Sigma)}.
\]
In the setting of the Heisenberg group, the algebra of homogeneous Fourier multipliers
\[
    \{ g(-i\nabla(-\Delta)^{-\frac12})\;:\; g \in L_{\infty}(S^{d-1})\} \subset \mathcal{B}(L_2(\mathbb{R}^d))
\]
is replaced by the \emph{noncommutative} algebra given by the image of $\pi_{\red}$ in $\mathcal{B}(L_2(\Heis^n)).$ Instead of integrating
over the ($d-1$)-sphere $S^{d-1},$ we take a trace on the algebra of ``Fourier symbols" $\mathcal{B}(L_2(\mathbb{R}^n))\overline{\otimes} \mathbb{C}^2.$

As a replacement for the formula \eqref{connes_trace_formula_commutative_example}, we have the following result.
\begin{thm}\label{connes product theorem} Let $\{f_k\}_{k=1}^{2n+2}\subset L_{2n+2}(\mathbb{H}^n)$ and $\{x_k\}_{k=1}^{2n+2}\subset L_{2n+2}(\mathcal{B}(L_2(\mathbb{R}^n))\overline{\otimes}\mathbb{C}^2,{\rm Tr}\otimes\Sigma).$ Then there exists a constant $c_n>0$ depending only on $n$ such that for every normalised continuous trace $\varphi$ on $\mathcal{L}_{1,\infty}$ we have
\begin{align*}
\varphi(A(x_1,f_1)^{\ast}&A(x_2,f_2)A(x_3,f_3)^{\ast}A(x_4,f_4)\cdots A(x_{2n+1},f_{2n+1})^{\ast}A(x_{2n+2},f_{2n+2}))\\
                         &=c_n\int_{\mathbb{H}^n}\bar{f}_1f_2\bar{f}_3f_4\cdots\bar{f}_{2n+1}f_{2n+2}\cdot \big({\rm Tr}\otimes\Sigma\big)(x_1^{\ast}x_2x_3^{\ast}x_4\cdots x_{2n+1}^{\ast}x_{2n+2}).
\end{align*}
\end{thm}

As a Corollary of Theorem \ref{connes product theorem}, we derive the following trace formula.
\begin{cor}\label{connes product corollary} Let $\{f_k\}_{k=0}^m\subset L_{2n+2}(\mathbb{H}^n)$ and $\{x_k\}_{k=0}^m\subset L_{2n+2}(\mathcal{B}(L_2(\mathbb{R}^{n}))\overline{\otimes}\mathbb{C}^2,{\rm Tr}\otimes\Sigma).$ Then there exists a constant $c_n>0$ depending only on $n$ such that for every normalised continuous trace $\varphi$ on $\mathcal{L}_{1,\infty}$ we have
$$\varphi\bigg(\bigg|\sum_{k=0}^mA(x_k,f_k)\bigg|^{2n+2}\bigg)=c_n\bigg\|\sum_{k=0}^mf_k\otimes x_k\bigg\|_{L_{2n+2}(L_{\infty}(\mathbb{H}^n)\bar{\otimes}\mathcal{B}(L_2(\mathbb{R}^n))\overline{\otimes}\mathbb{C}^2)}^{2n+2}.$$
\end{cor}
We will complete the proof of Corollary \ref{connes product corollary} after establishing a few important lemmas. Note that in order for Corollary \ref{connes product corollary} to follow from Theorem \ref{connes product theorem}, we will use in an essential way that $2n+2$, the homogeneous dimension on $\Heis^n,$ is even. This is a feature of the Heisenberg group which is crucial to the approach of this paper.

Precisely as with the proof of \eqref{connes_trace_formula_commutative_example} in \cite{Dao1}, the proof of Theorem \ref{connes product theorem} will be based on a symmetry and continuity
argument. The first step is the following lemma.
\begin{lem}\label{first multilinear lemma}
Let $l:L_{n+1}(\mathbb{H}^n)^{\times (n+1)}\to \mathbb{C}$ be a bounded multilinear functional satisfying
\begin{align}\label{condi}
l(ff_1,f_2,\cdots,f_{n+1})=l(f_1,\cdots,f_{k-1},ff_k,f_{k+1},\cdots,f_{n+1}),\quad 1\leq k\leq n+1,
\end{align}
whenever $f,f_1,\cdots,f_{n+1}\in C^{\infty}_c(\mathbb{H}^n),$ then there exists $h\in L_{\infty}(\mathbb{H}^n)$ such that
$$l(f_1,f_2,\cdots,f_{n+1})=\int_{\mathbb{H}^n}hf_1\cdots f_{n+1},\quad f_1,\cdots,f_{n+1}\in L_{n+1}(\mathbb{H}^n).$$
\end{lem}
\begin{proof} Assume, for simplicity of notations, that $n=1.$ The argument for an arbitrary $n$ follows {\it mutatis mutandi}. For simplicity of notations, $\|l\|_{L_2\times L_2\to\mathbb{C}}=1.$
	
Let $A,B\subset\mathbb{H}^1$ be disjoint Borel sets with finite Lebesgue measure.
Let $\{f^m\}_{m\geq1},\{g^m\}_{m\geq1}\subset C^{\infty}_c(\mathbb{R})$ be sequences bounded in the uniform norm such that $f^m\to\chi_A,$ $g^m\to\chi_B$ in $L_2(\mathbb{H}^1).$ We have that $(f^m)^2\to\chi_A$ in $L_2(\mathbb{H}^1).$ Thus,
$$l(\chi_A,\chi_B)=\lim_{m\to\infty}l((f^m)^2,g^m)=\lim_{m\to\infty}l(f^m,f^mg^m).$$
By continuity, we have
$$|l(f^m,f^mg^m)|\leq\|f^m\|_2\|f^mg^m\|_2\leq\|f^m\|_2\|f^m-\chi_A\|_2\|g^m\|_{\infty}+\|f^m\|_2\|g^m-\chi_B\|_2.$$

Set
$$\nu(A)=l(\chi_A,\chi_A).$$
We claim that $\nu$ is a measure. Let $A,B\subset\mathbb{H}^1$ be disjoint Borel sets. By the preceding paragraph, we have
$$\nu(A\cup B)=\nu(A)+\nu(B)+l(\chi_A,\chi_B)+l(\chi_B,\chi_A)=\nu(A)+\nu(B).$$
Since $l$ is bounded, we have
$$|\nu(A)|\leq\|\chi_A\|_2\|\chi_A\|_2=m(A)$$
for every set of finite measure, where $m$ is the Lebesgue measure. This means $\nu$ is countably additive and absolutely continuous with respect to the Lebesgue measure. By the Radon-Nikodym theorem there exists $h \in L_{1}(\mathbb{H}^1)$ such that $d\nu=hdm.$ Since $\nu(A)\leq m(A),$ we in fact have $h \in L_{\infty}(\HH^1).$

Let $f_1,f_2\in L_2(\mathbb{H}^1)$ be supported on a set of finite measure. Suppose $f_1$ and $f_2$ are taking finitely many values. We write
$$f_1=\sum_ka_k\chi_{A_k},\quad f_2=\sum_kb_k\chi_{A_k},$$
where sets $\{A_k\}$ are pairwise disjoint. We have
$$l(f_1,f_2)=\sum_{k,l}a_kb_l l(\chi_{A_k},\chi_{A_l})=\sum_ka_kb_k l(\chi_{A_k},\chi_{A_k})=\sum_k a_kb_k\nu(A_k)=\int_{\mathbb{H}^1}hf_1f_2.$$
By continuity, the same equality holds for arbitrary $f_1,f_2\in L_2(\mathbb{H}^1).$
\end{proof}

\begin{lem}\label{second multilinear lemma} Let $l:L_{n+1}(\mathbb{H}^n)^{\times (n+1)}\to \mathbb{C}$ be a bounded multilinear functional satisfying the condition \eqref{condi} and
$$l(\tau^g f_1,\cdots,\tau^gf_{n+1})=l(f_1,\cdots,f_{n+1}),\ {\rm for}\ {\rm any}\ g\in\HH^n,$$
whenever $f_1,\cdots,f_{n+1}\in L_{n+1}(\mathbb{H}^n),$ where $\tau^g$ is the left translation operator defined by $\tau^g f(x):=f(gx)$, then there exists $c_l \in \mathbb{C}$ such that
$$l(f_1,\cdots,f_{n+1})=c_l\int_{\mathbb{H}^n}f_1\cdots f_{n+1},\quad f_1,\cdots,f_{n+1}\in L_{n+1}(\mathbb{H}^n).$$
\end{lem}
\begin{proof}
By Lemma \ref{first multilinear lemma}, there exists $h \in L_{\infty}(\HH^n)$ such that for any $f_1,\cdots,f_{n+1}\in L_{n+1}(\mathbb{H}^n)$,
\[
    l(f_1,f_2,\cdots,f_{n+1})=\int_{\mathbb{H}^n}hf_1\cdots f_{n+1}.
\]
By assumption, for every $g \in \HH^n$ we have
$$\int_{\mathbb{H}^n}h\cdot (\tau^g f_1)\cdots (\tau^g f_{n+1})=\int_{\HH^n} hf_1\cdots f_{n+1}$$
for all $f_1,\cdots,f_{n+1}\in L_{n+1}(\mathbb{H}^n).$ If $f_1$ is supported in a compact set $K\subset \mathbb{H}^n$ and $f_2=f_3=\cdots=\chi_K,$ in particular we have
\[
\int_{\mathbb{H}^n}h\cdot (\tau^g f_1)=\int hf_1,\quad g \in \HH^n.
\]
Since the Lebesgue measure is invariant under $H$, it follows that
\[
    \int_{\HH^n} (\tau^{g^{-1}}h-h)f_1 = 0,\quad g\in \HH^n
\]
for all $f_1\in L_{2n+2}(\HH^n)$ supported in an arbitrary compact set $K$. It follows that $h\circ \tau^{g^{-1}}=h$ pointwise almost everywhere. That is, $h$ is almost everywhere
equal to a constant.
\end{proof}

\begin{lem}\label{commutator separable part lemma}    For every $f \in C^\infty_c(\Heis^n)$, $\xi \in \mathbb{C}^2$ and $\alpha \in \mathbb{Z}_+^n,$ we have
$$\pi_{{\rm red}}(E_{0,\alpha}\otimes \xi)\cdot M_f(-\Delta_{\HH})^{-\frac12}  - \Big(\frac{n}{n+2|\alpha|}\Big)^{\frac12} M_f(-\Delta_{\HH})^{-\frac12}\cdot  \pi_{{\rm red}}(E_{0,\alpha}\otimes \xi)\in(\mathcal{L}_{2n+2,\infty})_0$$
and
$$\pi_{{\rm red}}(E_{0,\alpha}\otimes \xi)\cdot (-\Delta_{\HH})^{-\frac12}M_f  - \Big(\frac{n}{n+2|\alpha|}\Big)^{\frac12} (-\Delta_{\HH})^{-\frac12}M_f\cdot  \pi_{{\rm red}}(E_{0,\alpha}\otimes \xi)\in(\mathcal{L}_{2n+2,\infty})_0.$$		
\end{lem}
\begin{proof} Let $f_1,f_2\in C^{\infty}_c(\mathbb{H}^n).$ Commuting $\pi_{\red}(E_{0,\alpha}\otimes \xi)$ with $M_{f_1}$ yields
\begin{equation*}
    \pi_{{\rm red}}(E_{0,\alpha}\otimes \xi)\cdot M_{f_1}(-\Delta_{\HH})^{-\frac12}M_{f_2}=[\pi_{{\rm red}}(E_{0,\alpha}\otimes \xi),M_{f_1}]\cdot (-\Delta_{\HH})^{-\frac12}M_{f_2}+M_{f_1}\cdot \pi_{{\rm red}}(E_{0,\alpha}\otimes \xi)(-\Delta_{\HH})^{-\frac12}\cdot M_{f_2}.
\end{equation*}
Similarly, commuting $\pi_{\red}(E_{0,\alpha}\otimes \xi)$ with $M_{f_2}$ leads to
\begin{equation*}
    M_{f_1}(-\Delta_{\HH})^{-\frac12}M_{f_2}\cdot \pi_{{\rm red}}(E_{0,\alpha}\otimes \xi) =M_{f_1}(-\Delta_{\HH})^{-\frac12}\cdot [M_{f_2},\pi_{{\rm red}}(E_{0,\alpha}\otimes \xi)]+M_{f_1}(-\Delta_\HH)^{-\frac{1}{2}}\pi_{{\rm red}}(E_{0,\alpha}\otimes \xi)M_{f_{2}}.
\end{equation*}
By Lemma \ref{compact commutator lemma}, the commutators $[\pi_{\red}(E_{0,\alpha}\otimes \xi),M_{f_1}]$ and $[\pi_{\red}(E_{0,\alpha}\otimes \xi),M_{f_2}]$ are compact. Since $M_{f_1}(-\Delta_{\HH})^{-\frac12},(1-\Delta_{\HH})^{-\frac12}M_{f_2} \in \mathcal{L}_{2n+2,\infty}$, it follows that
$$\pi_{{\rm red}}(E_{0,\alpha}\otimes \xi)\cdot M_{f_1}(-\Delta_{\HH})^{-\frac12}M_{f_2}\in M_{f_1}\cdot \pi_{{\rm red}}(E_{0,\alpha}\otimes \xi)(-\Delta_{\HH})^{-\frac12}\cdot M_{f_2}+(\mathcal{L}_{2n+2,\infty})_0$$
and
$$M_{f_1}(-\Delta_{\HH})^{-\frac12}M_{f_2}\cdot \pi_{{\rm red}}(E_{0,\alpha}\otimes \xi)\in M_{f_1}(-\Delta_\HH)^{-\frac{1}{2}}\pi_{{\rm red}}(E_{0,\alpha}\otimes \xi)M_{f_{2}}+(\mathcal{L}_{2n+2,\infty})_0.$$
Note that
$$\pi_{{\rm red}}(E_{0,\alpha}\otimes \xi)\cdot(-\Delta_{\HH})^{-\frac12}=\Big(\frac{n}{n+2|\alpha|}\Big)^{\frac12}(-\Delta_{\HH})^{-\frac12}\cdot\pi_{{\rm red}}(E_{0,\alpha}\otimes \xi).$$
Combining the proceeding identities, we conclude that
\[
    \pi_{{\rm red}}(E_{0,\alpha}\otimes \xi)\cdot M_{f_1}(-\Delta_{\HH})^{-\frac12}M_{f_2}\in \Big(\frac{n}{n+2|\alpha|}\Big)^{\frac12} M_{f_1}(-\Delta_{\HH})^{-\frac12}M_{f_2}\cdot  \pi_{{\rm red}}(E_{0,\alpha}\otimes \xi)+(\mathcal{L}_{2n+2,\infty})_0.
\]
	
We write $f=f_1f_2,$ where $f_1,f_2\in C^{\infty}_c(\mathbb{H}^n).$ By Lemma \ref{cwikel-like commutator lemma}, we have
$$M_f(-\Delta_{\HH})^{-\frac12}-M_{f_1}(-\Delta_{\HH})^{-\frac12}M_{f_2}\in (\mathcal{L}_{2n+2,\infty})_0,$$
$$(-\Delta_{\HH})^{-\frac12}M_f-M_{f_1}(-\Delta_{\HH})^{-\frac12}M_{f_2}\in (\mathcal{L}_{2n+2,\infty})_0.$$
The assertion follows now from the preceding paragraph.
\end{proof}

\begin{lem}\label{killing e00 lemma} There exists $c_n>0,$ depending only on $n,$ such that for all $\xi \in \mathbb{C}^2$ and $\{f_k\}_{k=1}^{2n+2} \subset C^\infty_c(\HH^n),$
\[
\varphi\bigg(\pi_{{\rm red}}(E_{0,0}\otimes \xi)\cdot \prod_{k=1}^{n+1}(-\Delta_{\HH})^{-\frac12}M_{\bar{f}_{2k-1}f_{2k}}(-\Delta_{\HH})^{-\frac12}\bigg)=c_n\varphi\bigg(\pi_{{\rm red}}(1\otimes \xi)\cdot\prod_{k=1}^{n+1}(-\Delta_{\HH})^{-\frac12}M_{\bar{f}_{2k-1}f_{2k}}(-\Delta_{\HH})^{-\frac12}\bigg).
\]
\end{lem}
\begin{proof}
Clearly,
$$\pi_{{\rm red}}(E_{\alpha,\alpha}\otimes \xi)=\pi_{{\rm red}}(E_{\alpha,0}\otimes \xi)\cdot \pi_{{\rm red}}(E_{0,\alpha}\otimes {\bf 1}).$$
By Lemma \ref{commutator separable part lemma}, we can replace
$$\pi_{{\rm red}}(E_{\alpha,\alpha}\otimes \xi)\cdot \prod_{k=1}^{n+1}(-\Delta_{\HH})^{-\frac12}M_{\bar{f}_{2k-1}f_{2k}}(-\Delta_{\HH})^{-\frac12}$$
by
$$\Big(\frac{n}{n+2|\alpha|}\Big)^{n+1} \pi_{{\rm red}}(E_{\alpha,0}\otimes \xi)\cdot \prod_{k=1}^{n+1}(-\Delta_{\HH})^{-\frac12}M_{\bar{f}_{2k-1}f_{2k}}(-\Delta_{\HH})^{-\frac12}\cdot\pi_{{\rm red}}(E_{0,\alpha}\otimes {\bf 1})$$
modulo $(\mathcal{L}_{1,\infty})_0.$ Since $\varphi$ vanishes on $(\mathcal{L}_{1,\infty})_0,$ it follows from the tracial property that
\[
    \varphi\bigg(\pi_{{\rm red}}(E_{\alpha,\alpha}\otimes \xi)\cdot \prod_{k=1}^{n+1}(-\Delta_{\HH})^{-\frac12}M_{\bar{f}_{2k-1}f_{2k}}(-\Delta_{\HH})^{-\frac12}\bigg)
           =\Big(\frac{n}{n+2|\alpha|}\Big)^{n+1}\varphi\bigg(\pi_{{\rm red}}(E_{0,0}\otimes \xi)\cdot \prod_{k=1}^{n+1}(-\Delta_{\HH})^{-\frac12}M_{\bar{f}_{2k-1}f_{2k}}(-\Delta_{\HH})^{-\frac12}\bigg).
\]
By continuity, we have
\begin{align*}
    \varphi\bigg(\pi_{{\rm red}}(1\otimes \xi)\cdot\prod_{k=1}^{n+1}(-\Delta_{\HH})^{-\frac12}M_{\bar{f}_{2k-1}f_{2k}}(-\Delta_{\HH})^{-\frac12}\bigg)&=\sum_{\alpha\in\mathbb{Z}^n_+}\varphi\bigg(\pi_{{\rm red}}(E_{\alpha,\alpha}\otimes \xi)\cdot \prod_{k=1}^{n+1}(-\Delta_{\HH})^{-\frac12}M_{\bar{f}_{2k-1}f_{2k}}(-\Delta_{\HH})^{-\frac12}\bigg)\\
           &=\sum_{\alpha\in\mathbb{Z}^n_+}\Big(\frac{n}{n+2|\alpha|}\Big)^{n+1}\varphi\bigg(\pi_{{\rm red}}(E_{0,0}\otimes \xi)\cdot \prod_{k=1}^{n+1}(-\Delta_{\HH})^{-\frac12}M_{\bar{f}_{2k-1}f_{2k}}(-\Delta_{\HH})^{-\frac12}\bigg).
\end{align*}
This ends the proof of Lemma \ref{killing e00 lemma}.
\end{proof}

\begin{lem}\label{connes product lemma} There exists $c_n>0$ (depending only on $n$) such that for all $\{f_k\}_{k=1}^{n+1}\subset L_{n+1}(\HH^n)$ we have
\[
    \varphi\bigg(\prod_{k=1}^{n+1}(-\Delta_{\HH})^{-\frac12}M_{f_k}(-\Delta_{\HH})^{-\frac12}\bigg)=c_n\int f_1\cdots f_{n+1},
\]
and
\[
    \varphi\bigg(\pi_{{\rm red}}(1\otimes z)\cdot \prod_{k=1}^{n+1}(-\Delta_{\HH})^{-\frac12}M_{f_k}(-\Delta_{\HH})^{-\frac12}\bigg)=0.
\]
\end{lem}
\begin{proof}
For any $f_1,\cdots,f_{n+1}\in C^{\infty}_c(\mathbb{H}^n),$ consider multilinear functionals $l_1,l_2:L_{n+1}(\mathbb{H}^n)^{\times (n+1)}\to \mathbb{C}$ defined by
\begin{align*}
l_1(f_1,\cdots,f_{n+1}):=\varphi\bigg(\prod_{k=1}^{n+1}(-\Delta_{\HH})^{-\frac12}M_{f_k}(-\Delta_{\HH})^{-\frac12}\bigg),
\end{align*}
and
\begin{align*}
l_2(f_1,\cdots,f_{n+1}):=\varphi\bigg(\pi_{{\rm red}}(1\otimes z)\cdot \prod_{k=1}^{n+1}(-\Delta_{\HH})^{-\frac12}M_{f_k}(-\Delta_{\HH})^{-\frac12}\bigg).
\end{align*}
Now we verify that $l_1$ and $l_2$ satisfy the assumption of Lemma \ref{first multilinear lemma}. Indeed, note that for all $f \in L_{n+1}(\HH^n),$ we have $|f|^{\frac12} \in L_{2n+2}(\HH^n),$
and therefore by Corollary \ref{specific_cwikel},
\[
    M_{|f|^{\frac12}}(-\Delta_{\HH})^{-\frac12} \in \mathcal{L}_{2n+2,\infty}.
\]
From the H\"older inequality, we have
\[
    \|(-\Delta_{\HH})^{-\frac12}M_f(-\Delta_{\HH})^{-\frac12}\|_{\mathcal{L}_{n+1,\infty}} \lesssim \|M_{|f|^{\frac12}}(-\Delta_{\HH})^{-\frac12}\|_{\mathcal{L}_{2n+2,\infty}}^2 \lesssim \||f|^{\frac12}\|_{L_{2n+2}(\HH^n)}^2 = \|f\|_{L_{n+1}(\HH^n)}.
\]
It follows that if $f_1,\ldots,f_{n+1} \in L_{n+1}(\HH^n)$, then again by the H\"{o}lder inequality,
\[
    \left\|\prod_{k=1}^{n+1}(-\Delta_{\HH})^{-\frac12}M_{f_k}(-\Delta_{\HH})^{-\frac12}\right\|_{\mathcal{L}_{1,\infty}}\lesssim \prod_{k=1}^{n+1}\|(-\Delta_{\HH})^{-\frac12}M_{f_k}(-\Delta_{\HH})^{-\frac12}\|_{\mathcal{L}_{n+1,\infty}} \lesssim \prod_{k=1}^{n+1}\|f_k\|_{L_{n+1}(\HH^n)}.
\]
Since the trace $\varphi$ is continuous in the $\mathcal{L}_{1,\infty}$ quasinorm, we see that
$l_1$ and $l_2$ are bounded multilinear functionals on $L_{n+1}(\HH^n)^{\times (n+1)}.$
Next, note that
\begin{align*}
(-\Delta_{\HH})^{-\frac12}M_{ff_1}(-\Delta_{\HH})^{-\frac12}
&=(-\Delta_{\HH})^{-\frac12}M_{f_1}(-\Delta_{\HH})^{-\frac12}\cdot M_f+(-\Delta_{\HH})^{-\frac12}M_{f_1}\cdot [M_f,(-\Delta_{\HH})^{-\frac12}],
\end{align*}
and that
\begin{align*}
(-\Delta_{\HH})^{-\frac12}M_{ff_2}(-\Delta_{\HH})^{-\frac12}
&=M_f\cdot (-\Delta_{\HH})^{-\frac12}M_{f_2}\cdot (-\Delta_{\HH})^{-\frac12}+[(-\Delta_{\HH})^{-\frac12},M_f]\cdot M_{f_2}(-\Delta_{\HH})^{-\frac12}.
\end{align*}
By Lemma \ref{cwikel-like commutator lemma} and Corollary \ref{specific_cwikel}, the second term in the right hand side fall into $(\mathcal{L}_{n+1,\infty})_0.$ So, we can replace
$$(-\Delta_{\HH})^{-\frac12}M_{ff_1}(-\Delta_{\HH})^{-\frac12}\mbox{ with }(-\Delta_{\HH})^{-\frac12}M_{f_1}(-\Delta_{\HH})^{-\frac12}\cdot M_f$$
and after that replace
$$M_f\cdot (-\Delta_{\HH})^{-\frac12}M_{f_2}(-\Delta_{\HH})^{-\frac12}\mbox{ with }(-\Delta_{\HH})^{-\frac12}M_{ff_2}(-\Delta_{\HH})^{-\frac12}.$$
Hence, the functionals $l_1$ and $l_2$ satisfy the assumption of Lemma \ref{first multilinear lemma}.

We now verify that $l_1$ and $l_2$ satisfy the assumption in Lemma \ref{second multilinear lemma}. Indeed, it suffices to note that $\Delta_{\HH}$ and $\pi_{{\rm red}}(1\otimes z)$ commute with $\lambda(g)$ and
$$\lambda(g)M_f\lambda(g)^{-1}=M_{\tau^{g^{-1}} f},$$
where we recall that $\lambda(g)$ is the left regular representation defined in Section \ref{VNGdef}. Combining these with the tracial property of $\varphi$, we see that $l_1$ and $l_2$ satisfy the assumption in Lemma \ref{second multilinear lemma}.

It follows now from Lemma \ref{second multilinear lemma} that
$$\varphi\Big(\prod_{k=1}^{n+1}(-\Delta_{\HH})^{-\frac12}M_{f_k}(-\Delta_{\HH})^{-\frac12}\Big)=c_{n,\varphi}\int f_1\cdots f_{n+1},$$	
$$\varphi\Big(\pi_{{\rm red}}(1\otimes z)\cdot \prod_{k=1}^{n+1}(-\Delta_{\HH})^{-\frac12}M_{f_k}(-\Delta_{\HH})^{-\frac12}\Big)=c_{n,\varphi}'\int f_1\cdots f_{n+1}$$
whenever $\{f_k\}_{k=1}^{n+1}\subset L_{n+1}(\mathbb{H}^n).$

Let $V$ be the unitary linear map on $L_2(\HH^n)$ given by $V\xi(g) = \xi(-g).$
Observing that $V\pi(1\otimes s) = -iVT= iTV = \pi(1\otimes -s)V,$ and $V|T| = |T|V,$
it follows that
\[
    V\pi(1\otimes \sgn(s)) = -\pi(1\otimes \sgn(s))V
\]
and therefore
$$V\cdot \pi_{{\rm red}}(1\otimes z)=-\pi_{{\rm red}}(1\otimes z)\cdot V.$$

Let $f_1=f_1=\cdots = f_{n+1} = f,$ where $f\geq 0$ is a non-negative radial function. Then $VM_fV^* = M_f,$ and noting that $V(-\Delta_{\HH})^{-1/2} = (-\Delta_{\HH})^{-1/2}V,$
the unitary invariance of the trace $\varphi$ implies that
\begin{align*}
    c_{n,\varphi}'\int_{\HH^n} f^{n+1} &= \varphi\Big(\pi_{{\rm red}}(1\otimes z)\cdot \prod_{k=1}^{n+1}(-\Delta_{\HH})^{-\frac12}M_{f_k}(-\Delta_{\HH})^{-\frac12}\Big)\\
                                       &= \varphi\Big(V\pi_{{\rm red}}(1\otimes z)\cdot \prod_{k=1}^{n+1}(-\Delta_{\HH})^{-\frac12}M_{f_k}(-\Delta_{\HH})^{-\frac12}V^*\Big)\\
                                       &= \varphi\Big(V\pi_{{\rm red}}(1\otimes z)V^*\cdot \prod_{k=1}^{n+1}(-\Delta_{\HH})^{-\frac12}M_{f_k}(-\Delta_{\HH})^{-\frac12}\Big)\\
                                       &= -\varphi\Big(\pi_{{\rm red}}(1\otimes z)\cdot \prod_{k=1}^{n+1}(-\Delta_{\HH})^{-\frac12}M_{f_k}(-\Delta_{\HH})^{-\frac12}\Big)\\
                                       &= -c_{n,\varphi}'\int_{\HH^n} f^{n+1}.
\end{align*}
We conclude that $c_{n,\varphi}'=0.$

Substituting $f_1=\cdots=f_n=f\geq0,$ we obtain
$$\varphi\Big(\Big((-\Delta_{\HH})^{-\frac12}M_f(-\Delta_{\HH})^{-\frac12}\Big)^{n+1}\Big)=c_{n,\varphi}\int f^{n+1},\quad f\in L_{n+1}(\mathbb{H}^n).$$
Comparing this with \cite{MSZ_cwikel}, we conclude that $c_{n,\varphi}$ does not depend on $\varphi$ (and is strictly positive).	This ends the proof.
\end{proof}

Since $\{\mathbf{1},z\}$ is a basis for $\mathbb{C}^2,$ it follows from Lemma \ref{connes product lemma} that for all $\xi \in \mathbb{C}^2$ and $f_1,\ldots,f_{n+1}\in L_{n+1}(\HH^n)$ we have
\[
    \varphi(\pi_{\red}(1\otimes \xi)\prod_{k=1}^{n+1} (-\Delta_{\HH})^{-\frac12}M_{f_k}(-\Delta_{\HH})^{-\frac12}) = c_n\Sigma(\xi)\int_{\HH^n} f_1\cdots f_{n+1},
\]
where $c_n>0$ is independent of $\xi.$

\begin{proof}[Proof of Theorem \ref{connes product theorem}] Recall from the begining of Section 5 that we have the isomorphism
\[
    L_2(\mathbb{R}^n) \approx \ell_2(\mathbb{Z}_+^n).
\]
Then by continuity, it suffices to prove the assertion for the case when each $f_k$ is a $C_c^\infty(\HH^n)$ function and when each $x_k$ has only finitely many non-zero matrix entries. By linearity, we may assume without loss of generality that $x_k=E_{\alpha_k,\beta_k}\otimes \xi_k,$ where $\alpha_k,\beta_k \in \mathbb{Z}_+^n$ and $\xi_k\in \mathbb{C}^2.$

By Lemma \ref{t and units commutator lemma}, up to a negligible $(\mathcal{L}_{2n+2,\infty})_0$ term, we may replace
$A(E_{\alpha_k,\beta_k}\otimes \xi_k,f_k)$
by
$$\pi_{{\rm red}}(E_{\alpha_k,0}\otimes {\bf 1})\cdot A(E_{0,0}\otimes \xi_k,f_k)\cdot \pi_{{\rm red}}(E_{0,\beta_k}\otimes {\bf 1}).$$
So,
$$A(x_1,f_1)^{\ast}A(x_2,f_2)A(x_3,f_3)^{\ast}A(x_4,f_4)\cdots A(x_{2n+1},f_{2n+1})^{\ast}A(x_{2n+2},f_{2n+2})$$
becomes
\begin{align*}
&\pi_{{\rm red}}(E_{\beta_1,0}\otimes {\bf 1})\cdot A(E_{0,0}\otimes \xi_1,f_1)^{\ast}\cdot \pi_{{\rm red}}(E_{0,\alpha_1}\otimes {\bf 1})\pi_{{\rm red}}(E_{\alpha_2,0}\otimes {\bf 1})\cdot A(E_{0,0}\otimes \xi_2,f_2)\cdot \pi_{{\rm red}}(E_{0,\beta_2}\otimes {\bf 1})\\
&\quad \times\pi_{{\rm red}}(E_{\beta_3,0}\otimes {\bf 1})\cdot A(E_{0,0}\otimes \xi_3,f_3)^{\ast}\cdot \pi_{{\rm red}}(E_{0,\alpha_3}\otimes {\bf 1})\\
&\quad \times\pi_{{\rm red}}(E_{\alpha_4,0}\otimes {\bf 1})\cdot A(E_{0,0}\otimes \xi_4,f_4)\cdot \pi_{{\rm red}}(E_{0,\beta_4}\otimes {\bf 1})\\
&\quad \cdots\\
&\quad \times\pi_{{\rm red}}(E_{\beta_{2n+1},0}\otimes {\bf 1})\cdot A(E_{0,0}\otimes \xi_{2n+1},f_{2n+1})^{\ast}\cdot \pi_{{\rm red}}(E_{0,\alpha_{2n+1}}\otimes {\bf 1})\\
&\quad \times \pi_{{\rm red}}(E_{\alpha_{2n+2},0}\otimes {\bf 1})\cdot A(E_{0,0}\otimes \xi_{2n+2},f_{2n+2})\cdot \pi_{{\rm red}}(E_{0,\beta_{2n+2}}\otimes {\bf 1}).
\end{align*}
Note that $E_{0,\alpha}E_{\beta,0}=0$ whenever $\alpha\neq \beta$. Hence, if $\alpha_1\neq\alpha_2$ or $\beta_2\neq \beta_3$ or $\alpha_3\neq \alpha_4$ or ... or $\alpha_{2n+1}\neq \alpha_{2n+2},$ then the above expression vanishes. So, we re-write it as
\begin{align*}
&\delta_{\alpha_1,\alpha_2}\delta_{\beta_2,\beta_3}\delta_{\alpha_3,\alpha_4}\cdots \delta_{\alpha_{2n+1},\alpha_{2n+2}}\pi_{{\rm red}}(E_{\beta_1,0}\otimes {\bf 1})\\
&\quad \times \pi_{{\rm red}}(E_{0,0}\otimes {\bf 1})\cdot A(E_{0,0}\otimes \xi_1,f_1)^{\ast}\cdot \pi_{{\rm red}}(E_{0,0}\otimes {\bf 1})\\
&\quad \times \pi_{{\rm red}}(E_{0,0}\otimes {\bf 1})\cdot A(E_{0,0}\otimes \xi_2,f_2)\cdot \pi_{{\rm red}}(E_{0,0}\otimes {\bf 1})\\
&\quad \times\pi_{{\rm red}}(E_{\beta_3,0}\otimes {\bf 1})\cdot A(E_{0,0}\otimes \xi_3,f_3)^{\ast}\cdot \pi_{{\rm red}}(E_{0,0}\otimes {\bf 1})\\
&\quad \times\pi_{{\rm red}}(E_{\alpha_4,0}\otimes {\bf 1})\cdot A(E_{0,0}\otimes \xi_4,f_4)\cdot \pi_{{\rm red}}(E_{0,0}\otimes {\bf 1})\\
&\quad \cdots\\
&\quad \times\pi_{{\rm red}}(E_{0,0}\otimes {\bf 1})\cdot A(E_{0,0}\otimes \xi_{2n+1},f_{2n+1})^{\ast}\cdot \pi_{{\rm red}}(E_{0,0}\otimes {\bf 1})\\
&\quad \times \pi_{{\rm red}}(E_{0,0}\otimes {\bf 1})\cdot A(E_{0,0}\otimes \xi_{2n+2},f_{2n+2})\cdot \pi_{{\rm red}}(E_{0,0}\otimes {\bf 1})\\
&\quad \times \pi_{{\rm red}}(E_{0,\beta_{2n+2}}\otimes {\bf 1}).
\end{align*}
Again using Lemma \ref{t and units commutator lemma}, we replace the latter expression with
\begin{align*}
&\delta_{\alpha_1,\alpha_2}\delta_{\beta_2,\beta_3}\delta_{\alpha_3,\alpha_4}\cdots \delta_{\alpha_{2n+1},\alpha_{2n+2}}\pi_{{\rm red}}(E_{\beta_1,0}\otimes {\bf 1})\\
&\quad \times A(E_{0,0}\otimes \xi_1,f_1)^{\ast}A(E_{0,0}\otimes \xi_2,f_2)\\
&\quad \times A(E_{0,0}\otimes \xi_3,f_3)^{\ast}A(E_{0,0}\otimes \xi_4,f_4)\\
&\quad \cdots\\
&\quad \times A(E_{0,0}\otimes \xi_{2n+1},f_{2n+1})^{\ast}A(E_{0,0}\otimes \xi_{2n+2},f_{2n+2})\\
&\quad \times \pi_{{\rm red}}(E_{0,\beta_{2n+2}}\otimes {\bf 1}).
\end{align*}

By the tracial property, the trace of the latter expression vanishes when $\beta_{2n+2}\neq\beta_1.$ Thus,
\begin{align*}
    \varphi(A(x_1,f_1)^{\ast}&A(x_2,f_2)A(x_3,f_3)^{\ast}A(x_4,f_4)\cdots A(x_{2n+1},f_{2n+1})^{\ast}A(x_{2n+2},f_{2n+2}))\\
                             &=\delta_{\alpha_1,\alpha_2}\delta_{\beta_2,\beta_3}\delta_{\alpha_3,\alpha_4}\cdots \delta_{\alpha_{2n+1},\alpha_{2n+2}}\delta_{\beta_{2n+2},\beta_1} \times\varphi\Big(\pi_{{\rm red}}(E_{0,0}\otimes {\bf 1})\\
&\quad \times A(E_{0,0}\otimes \xi_1,f_1)^{\ast}A(E_{0,0}\otimes \xi_2,f_2)A(E_{0,0}\otimes \xi_3,f_3)^{\ast}A(E_{0,0}\otimes \xi_4,f_4)\\
&\quad \cdots\\
&\quad \times A(E_{0,0}\otimes \xi_{2n+1},f_{2n+1})^{\ast}A(E_{0,0}\otimes \xi_{2n+2},f_{2n+2})\pi_{{\rm red}}(E_{0,0}\otimes {\bf 1})\Big).
\end{align*}

We now write
$$A(E_{0,0}\otimes \xi_k,f_k)=n^{\frac12}\cdot M_{f_k}(-\Delta_{\HH})^{-\frac12}\cdot \pi_{{\rm red}}(E_{0,0}\otimes \xi_k).$$
Thus,
\begin{align*}
    \varphi(A(x_1,&f_1)^{\ast}A(x_2,f_2)A(x_3,f_3)^{\ast}A(x_4,f_4)\cdots A(x_{2n+1},f_{2n+1})^{\ast}A(x_{2n+2},f_{2n+2}))\\
                             &=n^{n+1}\delta_{\alpha_1,\alpha_2}\delta_{\beta_2,\beta_3}\delta_{\alpha_3,\alpha_4}\cdots \delta_{\alpha_{2n+1},\alpha_{2n+2}}\delta_{\beta_{2n+2},\beta_1}\varphi\Big(\pi_{{\rm red}}(E_{0,0}\otimes {\bf 1})\\
                        &\quad \times \pi_{{\rm red}}(E_{0,0}\otimes \overline{\xi_1})\cdot (-\Delta_{\HH})^{-\frac12}M_{\bar{f}_1f_2}(-\Delta_{\HH})^{-\frac12}\cdot \pi_{{\rm red}}(E_{0,0}\otimes \xi_2)\\
                        &\quad \times \pi_{{\rm red}}(E_{0,0}\otimes \overline{\xi_3})\cdot (-\Delta_{\HH})^{-\frac12}M_{\bar{f}_3f_4}(-\Delta_{\HH})^{-\frac12}\cdot \pi_{{\rm red}}(E_{0,0}\otimes \xi_4)\\
                        &\quad \cdots\\
                        &\quad \times \pi_{{\rm red}}(E_{0,0}\otimes \overline{\xi_{2n+1}})\cdot (-\Delta_{\HH})^{-\frac12}M_{\bar{f}_{2n+1}f_{2n+2}}(-\Delta_{\HH})^{-\frac12}\cdot \pi_{{\rm red}}(E_{0,0}\otimes \xi_{2n+2})\\
                        &\quad \times \pi_{{\rm red}}(E_{0,0}\otimes {\bf 1})\Big).
\end{align*}
By Lemma \ref{commutator separable part lemma} taken with $\alpha=0,$ $\pi_{{\rm red}}(E_{0,0}\otimes \xi)$ commutes with $M_f(-\Delta_{\HH})^{-\frac12}$ modulo $(\mathcal{L}_{2n+2,\infty})_0.$ Therefore
\begin{align*}
\varphi(A(x_1,&f_1)^{\ast}A(x_2,f_2)A(x_3,f_3)^{\ast}A(x_4,f_4)\cdots A(x_{2n+1},f_{2n+1})^{\ast}A(x_{2n+2},f_{2n+2}))\\
                         &=n^{n+1}\delta_{\alpha_1,\alpha_2}\delta_{\beta_2,\beta_3}\delta_{\alpha_3,\alpha_4}\cdots \delta_{\alpha_{2n+1},\alpha_{2n+2}}\delta_{\beta_{2n+2},\beta_1} \times\varphi\Big(\pi_{{\rm red}}\left(E_{0,0}\otimes \left(\overline{\xi_1}\xi_2\cdots \overline{\xi_{2n+1}}\xi_{2n+2}\right)\right)\cdot \prod_{k=1}^{n+1}(-\Delta_{\HH})^{-\frac12}M_{\bar{f}_{2k-1}f_{2k}}(-\Delta_{\HH})^{-\frac12}\Big).
\end{align*}
By Lemma \ref{killing e00 lemma}, we have
\begin{align*}
    \varphi(A(x_1,f_1)^{\ast}&A(x_2,f_2)A(x_3,f_3)^{\ast}A(x_4,f_4)\cdots A(x_{2n+1},f_{2n+1})^{\ast}A(x_{2n+2},f_{2n+2}))\\
                             &=n^{n+1}\delta_{\alpha_1,\alpha_2}\delta_{\beta_2,\beta_3}\delta_{\alpha_3,\alpha_4}\cdots \delta_{\alpha_{2n+1},\alpha_{2n+2}}\delta_{\beta_{2n+2},\beta_1}
\varphi\Big(\pi_{{\rm red}}\left(1\otimes \left(\overline{\xi_1}\xi_2\cdots \overline{\xi_{2n+1}}\xi_{2n+2}\right)\right)\cdot \prod_{k=1}^{n+1}(-\Delta_{\HH})^{-\frac12}M_{\bar{f}_{2k-1}f_{2k}}(-\Delta_{\HH})^{-\frac12}\Big).
\end{align*}
Finally, by Lemma \ref{connes product lemma}, we have
\begin{align*}
    \varphi(A(x_1,f_1)^{\ast}&A(x_2,f_2)A(x_3,f_3)^{\ast}A(x_4,f_4)\cdots A(x_{2n+1},f_{2n+1})^{\ast}A(x_{2n+2},f_{2n+2}))\\
        &=c_n\delta_{\alpha_1,\alpha_2}\delta_{\beta_2,\beta_3}\delta_{\alpha_3,\alpha_4}\cdots \delta_{\alpha_{2n+1},\alpha_{2n+2}}\delta_{\beta_{2n+2},\beta_1}\Sigma(\overline{\xi_1}\xi_2\cdots \overline{\xi_{2n+2}}\xi_{2n+2})\int_{\mathbb{H}^n}\bar{f}_1f_2\bar{f}_3f_4\cdots\bar{f}_{2n+1}f_{2n+2}.
\end{align*}
Note that
\begin{align*}
    ({\rm Tr}\otimes \Sigma)(x_1^*x_2\cdots x_{2n+1}^*x_{2n+2}) &= ({\rm Tr}\otimes \Sigma)(E_{\alpha_1,\beta_1}^*E_{\alpha_2,\beta_2}\cdots E_{\alpha_{2n+1},\beta_{2n+1}}^*E_{\alpha_{2n+2},\beta_{2n+2}}\otimes \overline{\xi_1}\xi_2\cdots \overline{\xi_{2n+1}}\xi_{2n+2})\\
                                                                &= {\rm Tr}(E_{\alpha_1,\beta_1}^*E_{\alpha_2,\beta_2}\cdots E_{\alpha_{2n+1},\beta_{2n+1}}^*E_{\alpha_{2n+2},\beta_{2n+2}})\Sigma(\overline{\xi_1}\xi_2\cdots \overline{\xi_{2n+1}}\xi_{2n+2})\\
                                                                &= {\rm Tr}(E_{\beta_1,\alpha_1}E_{\alpha_2,\beta_2}\cdots E_{\beta_{2n+1},\alpha_{2n+1}}E_{\alpha_{2n+2},\beta_{2n+2}})\Sigma(\overline{\xi_1}\xi_2\cdots \overline{\xi_{2n+1}}\xi_{2n+2})\\
                                                                &= \delta_{\alpha_1,\alpha_2}\delta_{\beta_2,\beta_3}\delta_{\alpha_3,\alpha_4}\cdots \delta_{\alpha_{2n+1},\alpha_{2n+2}}\delta_{\beta_{2n+2},\beta_1}\Sigma(\overline{\xi_1}\xi_2\cdots \overline{\xi_{2n+1}}\xi_{2n+2}).
\end{align*}
This exactly coincides with the right hand side and the assertion follows.
\end{proof}

\begin{proof}[Proof of Corollary \ref{connes product corollary}]
Since $2n+2$ is even, the absolute value
\[
    \left|\sum_{k=0}^m A(x_k,f_k)\right|^{2n+2}
\]
may be written as a polynomial in the variables $\{A(x_k,f_k),A(x_k,f_k)^*\}_{k=0}^m.$ More specifically, we write
\begin{align*}
\bigg|\sum_{k=0}^mA(x_k,f_k)\bigg|^{2n+2}
=\sum_{F:\{1,\cdots,2n+2\}\to{\{0,\cdots,m\}}}\prod_{l=1}^{n+1}A(x_{F(2l-1)},f_{F(2l-1)})^\ast A(x_{F(2l)},f_{F(2l)}).
\end{align*}
Here, the summation is taken over all functions $F$ from the set $\{1,2,\ldots,2n+2\}$ to $\{0,1,2,\ldots,m\}.$
Since $\varphi$ is linear,
\begin{align*}
\varphi\bigg(\bigg|\sum_{k=0}^mA(x_k,f_k)\bigg|^{2n+2}\bigg)
=\sum_{F:\{1,\cdots,2n+2\}\to{\{0,\cdots,m\}}}\varphi\bigg(\prod_{l=1}^{n+1}A(x_{F(2l-1)},f_{F(2l-1)})^\ast A(x_{F(2l)},f_{F(2l)})\bigg).
\end{align*}
By Theorem \ref{connes product theorem}, there is a constant $c_n>0$ such that
\begin{align*}
\varphi\bigg(\bigg|\sum_{k=0}^mA(x_k,f_k)\bigg|^{2n+2}\bigg)
&=c_n\sum_{F:\{1,\cdots,2n+2\}\to{\{0,\cdots,m\}}}\int_{\mathbb{H}^n}\prod_{l=1}^{n+1}\bar{f}_{F(2l-1)}f_{F(2l)}\cdot({\rm Tr}\otimes\Sigma)\bigg(\prod_{l=1}^{n+1}x_{F(2l-1)}^{\ast}x_{F(2l)}\bigg)\\
&=c_n\bigg(\int_{\mathbb{H}^n}\otimes{\rm Tr}\otimes\Sigma\bigg)\bigg(\sum_{F:\{1,\cdots,2n+2\}\to{\{0,\cdots,m\}}}\prod_{l=1}^{n+1}\bar{f}_{F(2l-1)}f_{F(2l)}\otimes \prod_{l=1}^{n+1}x_{F(2l-1)}^{\ast}x_{F(2l)}\bigg).
\end{align*}
For brevity, we denote
$B(x,f):=f\otimes x.$
Then it is immediate that
\begin{align*}
\sum_{F:\{1,\cdots,2n+2\}\to{\{0,\cdots,m\}}}\prod_{l=1}^{n+1}\bar{f}_{F(2l-1)}f_{F(2l)}\otimes \prod_{l=1}^{n+1}x_{F(2l-1)}^{\ast}x_{F(2l)}&=\sum_{F:\{1,\cdots,2n+2\}\to{\{0,\cdots,m\}}}\prod_{l=1}^{n+1}B(x_{F(2l-1)},f_{F(2l-1)})^{\ast}B(x_{F(2l)},f_{F(2l)})\\
&=\bigg|\sum_{k=0}^mB(x_k,f_k)\bigg|^{2n+2}.
\end{align*}
Thus,
$$\varphi\bigg(\bigg|\sum_{k=0}^mA(x_k,f_k)\bigg|^{2n+2}\bigg)=c_n\bigg(\int_{\mathbb{H}^n}\otimes{\rm Tr}\otimes\Sigma\bigg)\bigg(\bigg|\sum_{k=0}^mB(x_k,f_k)\bigg|^{2n+2}\bigg)$$
as required.
\end{proof}

\section{Approximation and trace formula for Riesz transform commutator}\label{Approximation and trace formula for Riesz transform commutator}
\setcounter{equation}{0}
\subsection{An approximation for $[R_\ell,M_f]$}
In this section, we first establish an approximation for $[R_\ell,M_f]$ up to a negligible error term such that the main term can be written as
$$M_{X_\ell f}(-\Delta_{\HH})^{-\frac12}+\sum_{k=1}^{2n}R_\ell a_kM_{X_kf}(-\Delta_{\HH})^{-\frac12}.$$
To be more explicit, we will prove the following proposition:
\begin{prop}\label{approximation theorem}
For every $f\in C_c^\infty(\HH^n)$ and $\ell\in\{1,2,\cdots,2n\}$, there exists $E\in (\mathcal{L}_{2n+2,\infty})_0$ such that
$$[R_\ell,M_f]= M_{X_\ell f}(-\Delta_{\HH})^{-\frac12}+\sum_{k=1}^{2n}R_\ell a_kM_{X_kf}(-\Delta_{\HH})^{-\frac12}+E,$$
where
\begin{align}\label{defofak}
a_k=T^{-\Delta_{\HH},-\Delta_{\HH}}_{\psi}\Big((-\Delta_{\HH})^{-1/4}X_k(-\Delta_{\HH})^{-1/4}\Big).
\end{align}
Here, the symbol is $\psi(\alpha_0,\alpha_1) := \frac{2\alpha_0^{\frac14}\alpha_1^{\frac14}}{\alpha_0^{\frac12}+\alpha_1^{\frac12}},\; \alpha_0,\alpha_1>0$ and this double operator
integral is bounded on $\mathcal{B}(L_2(\HH^n)),$ and also restricts to a bounded operator on $\mathrm{VN}(\HH^n)$ \cite[Lemma 9]{PS-crelle}.
\end{prop}
\begin{remark}
The above form is significant for us in the next section to find an appropriate sequence $\{f_k^{\ell}\}_{k=0}^{2n}\subset L_p(\mathbb{H}^n)$ and $\{y_k^{\ell}\}_{k=0}^{2n}\subset L_p(\mathcal{B}(L_2(\mathbb{R}^n))\overline{\otimes}\mathbb{C}^2,{\rm Tr}\otimes\Sigma)$ such that the main term can be further written as (up to a negligible error term)
$$\sum_{k=0}^{2n}A(y_k^{\ell},f_k^{\ell})^{\ast},\ 1\leq \ell\leq 2n,$$
 where $A(x,f)$ is the mapping defined in Section \ref{tracelemmasection}.
 \end{remark}

To establish the approximation in Proposition \ref{approximation theorem}, we first note that by equality \eqref{Rieszdecompose}, we have
\begin{align*}
    [R_\ell,M_f]= M_{X_\ell f}(-\Delta_\HH)^{-\frac12}-R_\ell[(-\Delta_\HH)^{\frac12},M_f](-\Delta_\HH)^{-\frac12}.
\end{align*}
Then, we will use a perturbation argument (in Lemma \ref{pertubationlemma0} below) to show that
$$R_{\ell}[(1-\Delta_{\HH})^{\frac12},M_f](-\Delta_{\HH})^{-\frac12}-R_{\ell}[(-\Delta_{\HH})^{\frac12},M_f](-\Delta_{\HH})^{-\frac12}\in(\mathcal{L}_{2n+2,\infty})_0.$$
Our next step is to prove that
\begin{align}\label{expression via Ak}
R_{\ell}[(1-\Delta_\HH)^{\frac12},M_f](-\Delta_\HH)^{-\frac12}=-\sum_{k=1}^{2n}R_\ell A_kM_{X_kf}(-\Delta_{\HH})^{-\frac12}+\mathcal{E}_0,
\end{align}
where $\mathcal{E}_0\in (\mathcal{L}_{2n+2,\infty})_0$ and
\begin{align}\label{defofAk}
A_k:=\frac{2}{\pi}\int_{\mathbb{R}}\frac1{1+\lambda^2-\Delta_{\HH}}X_k\frac{\lambda^2 d\lambda}{1+\lambda^2-\Delta_{\HH}}.
\end{align}
The above steps will be presented in Lemmas \ref{preli1}-\ref{appro1} and be summarized in Corollary \ref{summuarize} with more details.
Then, it remains to show that \eqref{expression via Ak}  holds with the operator $A_k$ replaced by $a_k$ as in \eqref{defofak}  such that one can find $\{x_k\}_{k=0}^{2n}\subset \mathcal{B}(L_2(\mathbb{R}^n))\overline{\otimes} \mathbb{C}^2$ such that $a_k=\pi_{{\rm red}}(x_k)$. This will be given in the rest of this subsection.

Now we present our argument in more details by beginning with a perturbation argument.
\begin{lem}\label{pertubationlemma0}
For any $f\in C_c^{\infty}(\HH^n)$ and $\ell\in\{1,2,\cdots,2n\}$,  we have
$$R_\ell[(1-\Delta_{\HH})^{\frac12},M_f](-\Delta_{\HH})^{-\frac12}-R_\ell[(-\Delta_{\HH})^{\frac12},M_f](-\Delta_{\HH})^{-\frac12}\in(\mathcal{L}_{2n+2,\infty})_0.$$
\end{lem}
\begin{proof}
Since $R_\ell$ is bounded on $L_2(\HH^{n})$, it suffices to show that
$$[(1-\Delta_{\HH})^{\frac12},M_f](-\Delta_{\HH})^{-\frac12}-[(-\Delta_{\HH})^{\frac12},M_f](-\Delta_{\HH})^{-\frac12}\in(\mathcal{L}_{2n+2,\infty})_0.$$
To see this we write
\begin{align}\label{dfdfss}
&[(1-\Delta_{\HH})^{\frac12},M_f](-\Delta_{\HH})^{-\frac12}-[(-\Delta_{\HH})^{\frac12},M_f](-\Delta_{\HH})^{-\frac12}\nonumber\\
&=((1-\Delta_{\HH})^{\frac12}-(-\Delta_{\HH})^{\frac12})M_f(-\Delta_{\HH})^{-\frac12}-M_f((1-\Delta_{\HH})^{\frac12}-(-\Delta_{\HH})^{\frac12})(-\Delta_{\HH})^{-\frac12}\nonumber\\
&=((1-\Delta_{\HH})^{\frac12}+(-\Delta_{\HH})^{\frac12})^{-1}M_f(-\Delta_{\HH})^{-\frac12}-M_f((1-\Delta_{\HH})^{\frac12}+(-\Delta_{\HH})^{\frac12})^{-1}(-\Delta_{\HH})^{-\frac12}.
\end{align}
For the first term in the right-hand side of \eqref{dfdfss}, we conclude that
\begin{align*}
\Big\|((1-\Delta_{\HH})^{\frac12}+(-\Delta_{\HH})^{\frac12})^{-1}M_f(-\Delta_{\HH})^{-\frac12}\Big\|_{\mathcal{L}_{n+1,\infty}}&=\Big\|\Big(((1-\Delta_{\HH})^{\frac12}+(-\Delta_{\HH})^{\frac12})^{-1}(-\Delta_{\HH})^{\frac12} \Big)\,  (-\Delta_{\HH})^{-\frac12}M_f(-\Delta_{\HH})^{-\frac12}\Big\|_{\mathcal{L}_{n+1,\infty}}\\
&\leq\Big\|(-\Delta_{\HH})^{-\frac12}M_f(-\Delta_{\HH})^{-\frac12}\Big\|_{\mathcal{L}_{n+1,\infty}}\\
&\lesssim\|M_{|f|^{\frac12}}(-\Delta_{\HH})^{-\frac12}\|_{\mathcal{L}_{2n+2,\infty}}^2\lesssim\||f|^{\frac12}\|_{L_{2n+2}(\HH^n)}^2=c_n\|f\|_{L_{n+1}(\HH^n)},
\end{align*}
where the second inequality follows from H\"older's inequality and the last follows from Corollary \ref{specific_cwikel}.
For the second term in the right-hand side of \eqref{dfdfss}, we have
 that
\begin{align*}
\Big\|M_f((1-\Delta_{\HH})^{\frac12}+(-\Delta_{\HH})^{\frac12})^{-1}(-\Delta_{\HH})^{-\frac12}\Big\|_{\mathcal{L}_{2n+1,\infty}}
&=\Big\|M_f \, (-\Delta_{\HH})^{-\frac{n+1}{2n+1}}\ \Big( (-\Delta_{\HH})^{\frac{n+1}{2n+1}} ((1-\Delta_{\HH})^{\frac12}+(-\Delta_{\HH})^{\frac12})^{-1}(-\Delta_{\HH})^{-\frac12} \Big)\Big\|_{\mathcal{L}_{2n+1,\infty}}\\
&\leq\Big\|M_f(-\Delta_{\HH})^{-\frac{n+1}{2n+1}}\Big\|_{\mathcal{L}_{2n+1,\infty}}\leq c_n\|f\|_{L_{2n+1}(\HH^n)},
\end{align*}
where first inequality follows from H\"older's inequality and the last from Corollary \ref{specific_cwikel} with $\beta$ being chosen to be $\frac{2n+2}{2n+1}$.

Combining the above two inequalities with equality \eqref{dfdfss} and taking into account that $\mathcal{L}_{2n+1,\infty}\subset (\mathcal{L}_{2n+2,\infty})_0$ and $\mathcal{L}_{n+1,\infty}\subset (\mathcal{L}_{2n+2,\infty})_0$, we finishes the proof of Lemma \ref{pertubationlemma0}.
\end{proof}

\begin{lem}\label{preli1}
For any $f\in C_c^{\infty}(\HH^n)$ and for all $u\in \mathcal{S}(\HH^n)$ we have
\begin{align}\label{preequ}
[(1-\Delta_{\HH})^{\frac12},M_f](-\Delta_{\HH})^{-\frac12}u=-\frac1{\pi}\int_{\mathbb{R}}\Big[\frac1{1+\lambda^2-\Delta_{\HH}},M_f\Big](-\Delta_{\HH})^{-\frac12}u\cdot \lambda^2d\lambda.
\end{align}
The integral here is understood as a Bochner integral in the space $L_2(\HH^n).$
\end{lem}
\begin{proof}
We begin with the integral formula
$$(1-\Delta_{\HH})^{-\frac12}=\frac{1}{\pi}\int_{\mathbb{R}}\frac{d\lambda}{1+\lambda^2-\Delta_{\HH}}.$$
Here, the integral is interpreted in the sense that for every $u \in \mathcal{S}(\mathbb{H}^n)$ we have
\[
    (1-\Delta_{\HH})^{-\frac12}u = \frac{1}{\pi}\int_{\mathbb{R}}(1+\lambda^2-\Delta_{\HH})^{-1}u\,d\lambda
\]
and the integral converges as an $L_2(\HH^n)$-valued Bochner integral. Replacing $u$ with $(1-\Delta_{\HH})u,$ it follows that
$$(1-\Delta_{\HH})^{\frac12}u=\frac1{\pi}\int_{\mathbb{R}}\frac{1-\Delta_{\HH}}{1+\lambda^2-\Delta_{\HH}}ud\lambda,$$
where again the integral converges in the $L_2(\HH^n)$-valued Bochner sense.
A direct calculation yields
$$\frac{1-\Delta_{\HH}}{1+\lambda^2-\Delta_{\HH}}=1-\frac{\lambda^2}{1+\lambda^2-\Delta_{\HH}}.$$
Hence,
\begin{align}\label{rhs111}
[(1-\Delta_{\HH})^{\frac12},M_f]u=-\frac1{\pi}\int_{\mathbb{R}}\Big[\frac1{1+\lambda^2-\Delta_{\HH}},M_f\Big]u\cdot \lambda^2d\lambda,
\end{align}
where as before, the integral is convergent in the $L_2(\HH^n)$-valued Bochner sense for any $u \in \mathcal{S}(\HH^n).$
\end{proof}

To continue, we calculate the commutator in the integrand in the right hand side of \eqref{rhs111}.
\begin{lem}\label{commutator1}
For any $f\in C_c^{\infty}(\HH^n)$ and $\lambda\in\mathbb{R}$, we have
\begin{align}\label{changeorder}
\Big[\frac1{1+\lambda^2-\Delta_{\HH}},M_f\Big]=2\sum_{k=1}^{2n}\frac1{1+\lambda^2-\Delta_{\HH}}X_kM_{X_kf}\frac1{1+\lambda^2-\Delta_{\HH}}-\frac1{1+\lambda^2-\Delta_{\HH}}M_{\Delta_{\HH} f}\frac1{1+\lambda^2-\Delta_{\HH}}.
\end{align}
\end{lem}
\begin{proof}
To begin with, we apply the commutator formula \eqref{commutator} to write
\begin{align}\label{aaaaa1}
\Big[\frac1{1+\lambda^2-\Delta_{\HH}},M_f\Big]=\frac1{1+\lambda^2-\Delta_{\HH}}[\Delta_{\HH},M_f]\frac1{1+\lambda^2-\Delta_{\HH}}.
\end{align}
Next we observe that
\begin{align*}
[\Delta_{\HH},M_f]&=\sum_{k=1}^{2n}[X_k^2,M_f]=\sum_{k=1}^{2n}X_k[X_k,M_f]+[X_k,M_f]X_k\\
&=\sum_{k=1}^{2n}X_kM_{X_kf}+M_{X_kf}X_k=\sum_{k=1}^{2n}2X_kM_{X_kf}-[X_k,M_{X_kf}]\\
&=\sum_{k=1}^n2X_kM_{X_kf}-\sum_{k=1}^{2n}M_{X_k^2f}=\sum_{k=1}^{2n}X_kM_{X_kf}-M_{\Delta_{\HH} f}.
\end{align*}
Combining this with \eqref{aaaaa1} yields the proof of Lemma \ref{commutator1}.
\end{proof}
By changing the order of $M_{X_k f}$ and $(1+\lambda^2-\Delta_\HH)^{-1}$ in the first term of equality \eqref{changeorder}, we deduce the following lemma.
\begin{lem}\label{commutator2}
For any $f\in C_c^{\infty}(\HH^n)$ and $\lambda\in\mathbb{R}$, we have
\begin{align*}
\Big[\frac1{1+\lambda^2-\Delta_{\HH}},M_f\Big]&=2\sum_{k=1}^{2n}\frac1{1+\lambda^2-\Delta_{\HH}}X_k\frac1{1+\lambda^2-\Delta_{\HH}}M_{X_kf}\\
&\qquad-4\sum_{k,l=1}^{2n}\frac1{1+\lambda^2-\Delta_{\HH}}X_k\frac1{1+\lambda^2-\Delta_{\HH}}X_lM_{X_lX_kf}\frac1{1+\lambda^2-\Delta_{\HH}}\\
&\qquad+2\sum_{k=1}^{2n}\frac1{1+\lambda^2-\Delta_{\HH}}X_k\frac1{1+\lambda^2-\Delta_{\HH}}M_{\Delta_{\HH} X_kf}\frac1{1+\lambda^2-\Delta_{\HH}}\\
&\qquad-\frac1{1+\lambda^2-\Delta_{\HH}}M_{\Delta_{\HH} f}\frac1{1+\lambda^2-\Delta_{\HH}}.
\end{align*}
\end{lem}
\begin{proof}
We write each summand in the first term of the right-hand side of \eqref{changeorder} as
\begin{align}\label{commutator 33}
&\frac1{1+\lambda^2-\Delta_{\HH}}X_kM_{X_kf}\frac1{1+\lambda^2-\Delta_{\HH}}\\
&=\frac1{1+\lambda^2-\Delta_{\HH}}X_k\frac1{1+\lambda^2-\Delta_{\HH}}M_{X_kf}-\frac1{1+\lambda^2-\Delta_{\HH}}X_k\Big[\frac1{1+\lambda^2-\Delta_{\HH}},M_{X_kf}\Big].\nonumber
\end{align}
Applying Lemma \ref{commutator1} to the function $X_kf,$ we see that
\begin{align}\label{commutator 44}
\frac1{1+\lambda^2-\Delta_{\HH}}X_k\Big[\frac1{1+\lambda^2-\Delta_{\HH}},M_{X_kf}\Big]
&=2\sum_{l=1}^{2n}\bigg(\frac1{1+\lambda^2-\Delta_{\HH}}X_k\frac1{1+\lambda^2-\Delta_{\HH}}X_lM_{X_lX_kf}\frac1{1+\lambda^2-\Delta_{\HH}}\\
&\qquad\qquad-\frac1{1+\lambda^2-\Delta_{\HH}}X_k\frac1{1+\lambda^2-\Delta_{\HH}}M_{\Delta_{\HH} X_kf}\frac1{1+\lambda^2-\Delta_{\HH}}\bigg).\nonumber
\end{align}
Substituting \eqref{commutator 33} and \eqref{commutator 44} into the right-hand side of \eqref{changeorder}, we complete the proof.
\end{proof}

Combining Lemma \ref{preli1} and Lemma \ref{commutator2}, for any $u \in \mathcal{S}(\HH^n)$ we write
\begin{align*}
[(1-\Delta_{\HH})^{\frac12},M_f](-\Delta_{\HH})^{-\frac12}u
&=-\frac2{\pi}\sum_{k=1}^{2n}\int_{\mathbb{R}}\frac1{1+\lambda^2-\Delta_{\HH}}X_k\frac{1}{1+\lambda^2-\Delta_{\HH}} M_{X_kf}(-\Delta_{\HH})^{-\frac12}u\lambda^2d\lambda\nonumber\\
&\quad+\frac{4}{\pi}\sum_{k,l=1}^{2n}\int_{\mathbb{R}}\frac1{1+\lambda^2-\Delta_{\HH}}X_k\frac1{1+\lambda^2-\Delta_{\HH}}X_lM_{X_lX_kf}\frac1{1+\lambda^2-\Delta_{\HH}}(-\Delta_{\HH})^{-\frac12}u\lambda^2d\lambda\nonumber\\
&\quad-\frac{2}{\pi}\sum_{k=1}^{2n}\int_{\mathbb{R}}\frac1{1+\lambda^2-\Delta_{\HH}}X_k\frac1{1+\lambda^2-\Delta_{\HH}}M_{\Delta_{\HH} X_kf}\frac1{1+\lambda^2-\Delta_{\HH}}(-\Delta_{\HH})^{-\frac12}u\lambda^2d\lambda\nonumber\\
&\quad+\frac1{\pi}\int_{\mathbb{R}}\frac1{1+\lambda^2-\Delta_{\HH}}M_{\Delta_{\HH} f}\frac1{1+\lambda^2-\Delta_{\HH}}(-\Delta_{\HH})^{-\frac12}u\lambda^2d\lambda\nonumber\\
&=:\sum_{k=1}^{2n}A_{k}M_{X_kf}(-\Delta_\HH)^{-\frac{1}{2}}u+\mathcal{E}_1u+\mathcal{E}_2u+\mathcal{E}_3u,
\end{align*}
where $A_k$ is as in \eqref{defofAk}. We show in Lemma \ref{welldefine} that $A_k$ defines a bounded operator. In Lemmas \ref{appro3}--\ref{appro1} we will show that the integrals
$\mathcal{E}_1$, $\mathcal{E}_2$ and $\mathcal{E}_3$ converge in the Bochner sense in $(\mathcal{L}_{2n+2,\infty})_0$.
Granted this, it will follow that we have the equality of operators
\begin{equation}\label{keydecom}
    [(1-\Delta_{\HH})^{\frac12},M_f](-\Delta_{\HH})^{-\frac12}=\sum_{k=1}^{2n}A_{k}M_{X_kf}(-\Delta_\HH)^{-\frac{1}{2}}+\mathcal{E}_1+\mathcal{E}_2+\mathcal{E}_3.
\end{equation}

Since $\mathcal{L}_{2n+1,\infty}\subset (\mathcal{L}_{2n+2,\infty})_0$, the following Lemma in fact shows a stronger statement than $\mathcal{E}_1\in (\mathcal{L}_{2n+2,\infty})_0.$

\begin{lem}\label{appro3}
 For any  $f\in C_c^{\infty}(\HH^n)$ and $k,\ell\in \{1,2,\cdots,n\}$, we have
$$\int_{\mathbb{R}}\frac1{1+\lambda^2-\Delta_{\HH}}X_k\frac1{1+\lambda^2-\Delta_{\HH}}X_lM_{X_lX_kf}\frac1{1+\lambda^2-\Delta_{\HH}}(-\Delta_{\HH})^{-\frac12}\lambda^2d\lambda\in\mathcal{L}_{2n+1,\infty}.$$
\end{lem}
\begin{proof}
By H\"older's inequality, we have
\begin{align*}
&\Big\|\frac1{1+\lambda^2-\Delta_{\HH}}X_k\frac1{1+\lambda^2-\Delta_{\HH}}X_lM_{X_lX_kf}\frac1{1+\lambda^2-\Delta_{\HH}}(-\Delta_{\HH})^{-\frac12}\Big\|_{\mathcal{L}_{2n+1,\infty}}\\
&\leq\Big\|\frac1{1+\lambda^2-\Delta_{\HH}}\Big\|_{\mathcal{L}_{\infty}}\Big\|X_k\frac1{1+\lambda^2-\Delta_{\HH}}X_l\Big\|_{\mathcal{L}_{\infty}}\Big\|M_{X_lX_kf}\frac{(-\Delta_{\HH})^{-\frac12}}{1+\lambda^2-\Delta_{\HH}}\Big\|_{\mathcal{L}_{2n+1,\infty}}\\
&\lesssim\frac1{1+\lambda^2}\Big\|M_{X_lX_kf}\frac{(-\Delta_{\HH})^{-\frac12}}{1+\lambda^2-\Delta_{\HH}}\Big\|_{\mathcal{L}_{2n+1,\infty}}.
\end{align*}
Again using H\"older's inequality, we have
\begin{align*}
\Big\|M_{X_lX_kf}\frac{(-\Delta_{\HH})^{-\frac12}}{1+\lambda^2-\Delta_{\HH}}\Big\|_{\mathcal{L}_{2n+1,\infty}}
&\leq \Big\|M_{X_lX_kf}(-\Delta_{\HH})^{-\frac{n+1}{2n+1}}\Big\|_{\mathcal{L}_{2n+1,\infty}}\Big\|\frac{(-\Delta_{\HH})^{\frac1{4n+2}}}{1+\lambda^2-\Delta_{\HH}}\Big\|_{\mathcal{L}_{\infty}}\\
&\lesssim\frac1{(1+\lambda^2)^{\frac{4n+1}{4n+2}}}\Big\|M_{X_lX_kf}(-\Delta_{\HH})^{-\frac{n+1}{2n+1}}\Big\|_{\mathcal{L}_{2n+1,\infty}}.
\end{align*}
where in the last inequality we applied the spectral theorem and the elementary fact that
$$\sup\limits_{x\geq 0}\frac{x^{\frac{1}{2n+1}}}{1+\lambda^2+x^2}\lesssim \frac{1}{(1+\lambda^{2})^{\frac{4n+1}{4n+2}}},\ \ \lambda\geq 0.$$
Then applying Corollary \ref{specific_cwikel} with $f$ and $\beta$ being chosen to be $X_\ell X_k f$ and $\frac{2n+2}{2n+1}$, respectively, we conclude that
$$\Big\|\frac1{1+\lambda^2-\Delta_{\HH}}X_k\frac1{1+\lambda^2-\Delta_{\HH}}X_lM_{X_lX_kf}\frac1{1+\lambda^2-\Delta_{\HH}}(-\Delta_{\HH})^{-\frac12}\Big\|_{\mathcal{L}_{2n+1,\infty}}\leq\frac{c_n}{(1+\lambda^2)^{2-\frac1{4n+2}}}\|X_lX_kf\|_{L_{2n+1}(\HH^n)}.$$
Thus, the integrand is absolutely integrable in $\mathcal{L}_{2n+1,\infty}$. By \cite[Proposition 1.2.3]{HNVWbook},  this ends the proof of Lemma \ref{appro3}.
\end{proof}
The estimate of term $\mathcal{E}_2$ is the following.
\begin{lem}\label{appro2}
 For any  $f\in C_c^{\infty}(\HH^n)$ and $k\in\{1,2,\cdots,n\}$,
we have
$$\int_{\mathbb{R}}\frac1{1+\lambda^2-\Delta_{\HH}}X_k\frac1{1+\lambda^2-\Delta_{\HH}}M_{\Delta_{\HH} X_kf}\frac1{1+\lambda^2-\Delta_{\HH}}(-\Delta_{\HH})^{-\frac12}\lambda^2d\lambda\in(\mathcal{L}_{2n+2,\infty})_0.$$
\end{lem}
\begin{proof}
By H\"older's inequality, we have
\begin{align*}
&\Big\|\frac1{1+\lambda^2-\Delta_{\HH}}X_k\frac1{1+\lambda^2-\Delta_{\HH}}M_{\Delta_{\HH} X_kf}\frac1{1+\lambda^2-\Delta_{\HH}}(-\Delta_{\HH})^{-\frac12}\Big\|_{\mathcal{L}_{n+1,\infty}}\\
&\leq\Big\|\frac1{1+\lambda^2-\Delta_{\HH}}X_k\frac{(-\Delta_{\HH})^{\frac12}}{1+\lambda^2-\Delta_{\HH}}\Big\|_{\mathcal{L}_{\infty}}\Big\|(-\Delta_{\HH})^{-\frac12}M_{\Delta_{\HH} X_k f}(-\Delta_{\HH})^{-\frac12}\Big\|_{\mathcal{L}_{n+1,\infty}}\Big\|\frac1{1+\lambda^2-\Delta_{\HH}}\Big\|_{\mathcal{L}_{\infty}}\\
&\lesssim\Big\|(-\Delta_{\HH})^{-\frac12}M_{\Delta_{\HH} X_kf}(-\Delta_{\HH})^{-\frac12}\Big\|_{\mathcal{L}_{n+1,\infty}}\\
&\lesssim\Big\|M_{|\Delta_{\HH} X_kf|^{\frac12}}(-\Delta_{\HH})^{-\frac12}\|_{\mathcal{L}_{2n+2,\infty}}^2.
\end{align*}
Applying Corollary \ref{specific_cwikel} with $f$ and $\beta$ being chosen to be $|\Delta_\HH X_k f|^{\frac{1}{2}}$ and $1$, respectively, we conclude that
\begin{align*}
\Big\|\frac1{1+\lambda^2-\Delta_{\HH}}X_k\frac1{1+\lambda^2-\Delta_{\HH}}M_{\Delta_{\HH} X_kf}\frac1{1+\lambda^2-\Delta_{\HH}}(-\Delta_{\HH})^{-\frac12}\Big\|_{\mathcal{L}_{n+1,\infty}}
\lesssim\||\Delta_{\HH} X_kf|^{\frac12}\|_{L_{2n+2}(\HH^n)}^2=c_n\|\Delta_{\HH} X_kf\|_{L_{n+1}(\HH^n)}.
\end{align*}
Hence, the integrand belongs to $\mathcal{L}_{n+1,\infty}$, and therefore belongs to $(\mathcal{L}_{2n+2,\infty})_0.$

It remains to show absolute integrability in $\mathcal{L}_{2n+2,\infty}.$ By H\"older's inequality, we have
\begin{align}\label{ergvv}
&\Big\|\frac1{1+\lambda^2-\Delta_{\HH}}X_k\frac1{1+\lambda^2-\Delta_{\HH}}M_{\Delta_{\HH} X_kf}\frac1{1+\lambda^2-\Delta_{\HH}}(-\Delta_{\HH})^{-\frac12}\Big\|_{\mathcal{L}_{2n+2,\infty}}\nonumber\\
&\leq\Big\|\frac1{1+\lambda^2-\Delta_{\HH}}X_k\frac1{1+\lambda^2-\Delta_{\HH}}\Big\|_{\mathcal{L}_{\infty}}\Big\|M_{\Delta_{\HH} X_k f}(-\Delta_{\HH})^{-\frac12}\Big\|_{\mathcal{L}_{2n+2,\infty}}\Big\|\frac1{1+\lambda^2-\Delta_{\HH}}\Big\|_{\mathcal{L}_{\infty}}\nonumber\\
&\leq \Big\|\frac{1}{1+\lambda^{2}-\Delta_\HH}\Big\|_{\mathcal{L}_\infty}^2\|X_k(-\Delta_\HH)^{-\frac{1}{2}}\|_{\mathcal{L}_{\infty}}\Big\|\frac{(-\Delta_\HH)^{\frac{1}{2}}}{1+\lambda^{2}-\Delta_\HH}\Big\|_{\mathcal{L}_{\infty}}\|M_{\Delta_\HH X_kf}(-\Delta_\HH)^{-\frac{1}{2}}\|_{\mathcal{L}_{2n+2,\infty}}\nonumber\\
&\lesssim \frac{1}{(1+\lambda^{2})^{\frac{5}{2}}}\|M_{\Delta_\HH X_kf}(-\Delta_\HH)^{-\frac{1}{2}}\|_{\mathcal{L}_{2n+2,\infty}},
\end{align}
where in the last inequality we applied the spectral theorem and the elementary facts that
$$\sup\limits_{x\geq 0}\frac{1}{1+\lambda^{2}+x^2}\leq \frac{1}{1+\lambda^2},\ \ \sup\limits_{x\geq 0}\frac{x}{1+\lambda^{2}+x^{2}}\lesssim \frac{1}{(1+\lambda^2)^{\frac{1}{2}}}.$$
Applying Corollary \ref{specific_cwikel} with $f$ and $\beta$ being chosen to be $\Delta_\HH X_k f$ and $1$, respectively, we see that
\begin{align}\label{werwerew}
\|M_{\Delta_\HH X_kf}(-\Delta_\HH)^{-\frac{1}{2}}\|_{\mathcal{L}_{2n+2,\infty}}\lesssim \|\Delta_\HH X_k f\|_{L_{2n+2}(\HH^n)}.
\end{align}
Combining estimates \eqref{ergvv} and \eqref{werwerew} together, we see that
\begin{align*}
\Big\|\frac1{1+\lambda^2-\Delta_{\HH}}X_k\frac1{1+\lambda^2-\Delta_{\HH}}M_{\Delta_{\HH} X_kf}\frac1{1+\lambda^2-\Delta_{\HH}}(-\Delta_{\HH})^{-\frac12}\Big\|_{\mathcal{L}_{2n+2,\infty}}
\lesssim\frac{1}{(1+\lambda^2)^{\frac52}}\|\Delta_{\HH} X_kf\|_{L_{2n+2}(\HH^n)}.
\end{align*}
Thus, the integrand is absolutely integrable in $\mathcal{L}_{2n+2,\infty}$. By \cite[Proposition 1.2.3]{HNVWbook}, this ends the proof of Lemma \ref{appro2}.
\end{proof}
The estimate of term $\mathcal{E}_3$ is the following.
\begin{lem}\label{appro1}
 For any  $f\in C_c^{\infty}(\HH^n)$, we have
$$\int_{\mathbb{R}}\frac1{1+\lambda^2-\Delta_{\HH}}M_{\Delta_{\HH} f}\frac1{1+\lambda^2-\Delta_{\HH}}(-\Delta_{\HH})^{-\frac12}\cdot \lambda^2d\lambda\in(\mathcal{L}_{2n+2,\infty})_0.$$
\end{lem}
\begin{proof}
By H\"older's inequality, we have
\begin{align*}
\Big\|\frac1{1+\lambda^2-\Delta_{\HH}}M_{\Delta_{\HH} f}\frac1{1+\lambda^2-\Delta_{\HH}}(-\Delta_{\HH})^{-\frac12}\Big\|_{\mathcal{L}_{n+1,\infty}}
&\leq\Big\|\frac{(-\Delta_{\HH})^{\frac12}}{1+\lambda^2-\Delta_{\HH}}\Big\|_{\mathcal{L}_{\infty}}\Big\|(-\Delta_{\HH})^{-\frac12}M_{\Delta_{\HH} f}(-\Delta_{\HH})^{-\frac12}\Big\|_{\mathcal{L}_{n+1,\infty}}\Big\|\frac1{1+\lambda^2-\Delta_{\HH}}\Big\|_{\mathcal{L}_{\infty}}\\
&\lesssim\Big\|(-\Delta_{\HH})^{-\frac12}M_{\Delta_{\HH} f}(-\Delta_{\HH})^{-\frac12}\Big\|_{\mathcal{L}_{n+1,\infty}}\\&\lesssim\Big\|M_{|\Delta_{\HH} f|^{\frac12}}(-\Delta_{\HH})^{-\frac12}\|_{\mathcal{L}_{2n+2,\infty}}^2.
\end{align*}
Applying Corollary \ref{specific_cwikel} with $f$ and $\beta$ being chosen to be $|\Delta_\HH f|^{\frac{1}{2}}$ and $1$, we have
$$\Big\|\frac1{1+\lambda^2-\Delta_{\HH}}M_{\Delta_{\HH} f}\frac1{1+\lambda^2-\Delta_{\HH}}(-\Delta_{\HH})^{-\frac12}\Big\|_{\mathcal{L}_{n+1,\infty}}\lesssim\||\Delta_{\HH} f|^{\frac12}\|_{L_{2n+2}(\HH^n)}^2=c_n\|\Delta_{\HH} f\|_{L_{n+1}(\HH^n)}.$$
Hence, the integrand belongs to $\mathcal{L}_{n+1,\infty}$, and therefore belongs to $(\mathcal{L}_{2n+2,\infty})_0.$

It remains to show absolute integrability in $\mathcal{L}_{2n+2,\infty}.$ By H\"older's inequality, we have
$$\Big\|\frac1{1+\lambda^2-\Delta_{\HH}}M_{\Delta_{\HH} f}\frac1{1+\lambda^2-\Delta_{\HH}}(-\Delta_{\HH})^{-\frac12}\Big\|_{\mathcal{L}_{2n+2,\infty}}\leq\Big\|\frac1{1+\lambda^2-\Delta_{\HH}}\Big\|_{\mathcal{L}_{\infty}}^2\Big\|M_{\Delta_{\HH} f}(-\Delta_{\HH})^{-\frac12}\Big\|_{\mathcal{L}_{2n+2,\infty}}.$$
Applying Corollary \ref{specific_cwikel} again, we obtain that
\begin{align*}
\Big\|\frac1{1+\lambda^2-\Delta_{\HH}}M_{\Delta_{\HH} f}\frac1{1+\lambda^2-\Delta_{\HH}}(-\Delta_{\HH})^{-\frac12}\Big\|_{\mathcal{L}_{2n+2,\infty}}\leq \frac{c_n}{(1+\lambda^2)^2}\|\Delta_{\HH} f\|_{L_{2n+2}(\HH^n)}.
\end{align*}
Thus, the integrand is absolutely integrable in $\mathcal{L}_{2n+2,\infty}.$ By \cite[Proposition 1.2.3]{HNVWbook}, this ends the proof of Lemma \ref{appro1}.
\end{proof}
Summarising the results in Lemmas \ref{appro3}--\ref{appro1} together with equality \eqref{keydecom} and the $L_2(\HH^n)$-boundedness of Riesz transform, we have
\begin{cor}\label{summuarize}
For any $f\in C_c^{\infty}(\HH^n)$ and $\ell\in\{1,2,\cdots,2n\}$,  we have
\begin{align}
R_{\ell}[(1-\Delta_{\HH})^{\frac12},M_f](-\Delta_{\HH})^{-\frac12}
&=-\sum_{k=0}^{2n}R_{\ell}A_{k}M_{X_kf}(-\Delta_\HH)^{-\frac{1}{2}}+\mathcal{E}_0,
\end{align}
where $A_k$ is defined in \eqref{defofAk} and $\mathcal{E}_0=R_{\ell}\mathcal{E}_1+R_{\ell}\mathcal{E}_2+R_{\ell}\mathcal{E}_3\in (\mathcal{L}_{2n+2,\infty})_{0}$.
\end{cor}

Now we claim that
$A_k$ can be written as $T^{1-\Delta_{\HH},1-\Delta_{\HH}}_{\psi}((1-\Delta_{\HH})^{-\frac14}X_k(1-\Delta_{\HH})^{-\frac14})$, which has a similar expression as $a_k$. Here we denote $\psi(\alpha_0,\alpha_1):=\frac{\alpha_0^{1/4}\alpha_1^{1/4}}{\alpha_0^{1/2}+\alpha_1^{1/2}}$, for $\alpha_0$, $\alpha_1\geq0$. After that, we will use a perturbation argument again to show Corollary \ref{summuarize} still holds with the expression $A_k$ being replaced by $a_k$. To show the claim, we first recall that for a Lipschitz function $f$ defined on $\mathbb{R}^{+}$ and two distinct points $x_1$ and $x_2$, the first order divided difference of $f$ is defined by
$$ f^{[1]}(x_1,x_2)=\frac{f(x_1)-f(x_2)}{x_1-x_2}.$$

We need the following two auxiliary lemmas.

\begin{lem} Let $A\geq 0$ be a (possibly unbounded) linear operator on $L_2(\HH^n)$ and $F(\alpha)=\alpha\tan^{-1}(\alpha).$ 
Then
$$\|T^{A,A}_{F^{[1]}}\|_{\mathcal{L}_{\infty}\to\mathcal{L}_{\infty}}\leq c_{{\rm abs}}.$$
\end{lem}
\begin{proof}
    We write $F(\alpha) = \frac{\pi}{2}\alpha+F_0(\alpha),$ where
    \[
        F_0(\alpha) := \alpha\Bigg(\tan^{-1}(\alpha)-\frac{\pi}{2}\Bigg). 
    \]
    A direct calculation verified that $F_0', F_0'' \in L_2(\mathbb{R}_+),$ and hence \cite[Lemma 2.2]{ACDS}, \cite[Lemma 7]{PS-crelle} imply that
    \[
        \|T^{A,A}_{F_0^{[1]}}\|_{\mathcal{L}_{\infty}\to \mathcal{L}_{\infty}} < \infty.
    \]
    Since $(F-F_0)^{[1]} = {\pi\over2}$ and $T^{A,A}_1$ is the identity operator, this completes the proof.
%
\end{proof}

\begin{lem}\label{weakconverge}
Let $W$ be a bounded linear operator on $L_2(\HH^n)$, and $A\geq 0$ be a (possibly unbounded) linear operator on $L_2(\HH^n)$, then
$$\lim_{n\rightarrow \infty}T^{\frac1nA,\frac1nA}_{F^{[1]}}(W)=0$$
in weak operator topology.
\end{lem}
\begin{proof} Let $P_k=\chi_{[0,k]}(A).$ We claim that
$$\|P_k\cdot T^{\frac1nA,\frac1nA}_{F^{[1]}}(W)\cdot P_k\|_{\mathcal{L}_\infty}\leq c_{{\rm abs}}\frac{k}{n}\|W\|_{\mathcal{L}_\infty},\quad n\geq k.$$
Note that
$$P_k\cdot T^{\frac1nA,\frac1nA}_{F^{[1]}}(W)\cdot P_k=T^{\frac1nAP_k,\frac1nAP_k}_{F^{[1]}}(W).$$

Let $\Psi\in C^{\infty}_c(\mathbb{R})$ be such that $\Psi(\alpha^2)=F(\alpha)$ on $[0,1]$, then we have
$$F^{[1]}(\alpha_0,\alpha_1)=(\alpha_0+\alpha_1)\cdot\Psi^{[1]}(\alpha_0^2,\alpha_1^2),\quad \alpha_0,\alpha_1\in[0,\frac{k}{n}],\quad k\leq n.$$
Thus,
\begin{align*}
T^{\frac1nAP_k,\frac1nAP_k}_{F^{[1]}}(W)&=\frac1nAP_k\cdot T^{\frac1nAP_k,\frac1nAP_k}_{\Psi^{[1]}(\alpha_0^2,\alpha_1^2)}(W)+T^{\frac1nAP_k,\frac1nAP_k}_{\Psi^{[1]}(\alpha_0^2,\alpha_1^2)}(W)\cdot\frac1nAP_k\\
&=\frac1nAP_k\cdot T^{n^{-2}A^2P_k,n^{-2}A^2P_k}_{\Psi^{[1]}}(W)+T^{n^{-2}A^2P_k,n^{-2}A^2P_k}_{\Psi^{[1]}}(W)\cdot\frac1nAP_k.
\end{align*}
By H\"older's inequality, we have
$$\|P_k\cdot T^{\frac1nA,\frac1nA}_{F^{[1]}}(W)\cdot P_k\|_{\mathcal{L}_\infty}\leq2\|\frac1nAP_k\|_{\mathcal{L}_\infty}\| T^{n^{-2}A^2P_k,n^{-2}A^2P_k}_{\Psi^{[1]}}\|_{\mathcal{L}_{\infty}\to\mathcal{L}_{\infty}}\|W\|_{\mathcal{L}_\infty}\leq\frac{k}{n}c_{{\rm abs}}\|\Psi'\|_{W^{1,2}}\|W\|_{\mathcal{L}_\infty}.$$
This proves the claim.

Let $\xi,\eta\in L_2(\HH^n)$ be normalised vectors. Fix $\epsilon\in(0,1)$ and choose $k\in\mathbb{N}$ such that $\|(1-P_k)\xi\|_{L_2(\HH^{n})}<\epsilon$ and $\|(1-P_k)\eta\|_{L_{2}(\HH^{n})}<\epsilon.$ For $n>k\epsilon^{-1},$ we write
$$\langle T^{\frac1nA,\frac1nA}_{F^{[1]}}(W) \xi,\eta\rangle=\langle T^{\frac1nA,\frac1nA}_{F^{[1]}}(W) P_k\xi,P_k\eta\rangle+\langle T^{\frac1nA,\frac1nA}_{F^{[1]}}(W) (1-P_k)\xi,P_k\eta\rangle+\langle T^{\frac1nA,\frac1nA}_{F^{[1]}}(W) \xi,(1-P_k)\eta\rangle.$$
Now we estimate these terms separately:
\begin{align*}
|\langle T^{\frac1nA,\frac1nA}_{F^{[1]}}(W) (1-P_k)\xi,P_k\eta\rangle|&\leq \|T^{\frac1nA,\frac1nA}_{F^{[1]}}(W)\|_{\mathcal{L}_\infty}\| (1-P_k)\xi\|_{L_2(\HH^n)}\|P_k\eta\|_{L_2(\HH^n)}\\
&\leq \|T^{\frac1nA,\frac1nA}_{F^{[1]}}\|_{\mathcal{L}_{\infty}\to\mathcal{L}_{\infty}}\|W\|_{\mathcal{L}_\infty}\cdot\epsilon\leq c_{{\rm abs}}\|W\|_{\mathcal{L}_\infty}\epsilon,
\end{align*}
\begin{align*}
|\langle T^{\frac1nA,\frac1nA}_{F^{[1]}}(W) \xi,(1-P_k)\eta\rangle|&\leq\|T^{\frac1nA,\frac1nA}_{F^{[1]}}(W)\|_{\mathcal{L}_\infty}\|\xi\|_{L_2(\HH^n)}\|(1-P_k)\eta\|_{L_2(\HH^n)}\\
&\leq \|T^{\frac1nA,\frac1nA}_{F^{[1]}}\|_{\mathcal{L}_{\infty}\to\mathcal{L}_{\infty}}\|W\|_{\mathcal{L}_\infty}\cdot\epsilon\leq c_{{\rm abs}}\|W\|_{\mathcal{L}_\infty}\epsilon,
\end{align*}
\begin{align*}
|\langle T^{\frac1nA,\frac1nA}_{F^{[1]}}(W) P_k\xi,P_k\eta\rangle|&=|\langle P_k(T^{\frac1nA,\frac1nA}_{F^{[1]}}(W) P_k\xi),P_k\eta\rangle|\leq\|P_k\cdot T^{\frac1nA,\frac1nA}_{F^{[1]}}(W) \cdot P_k\|_{\mathcal{L}_\infty}\\&\leq c_{{\rm abs}}\frac{k}{n}\|W\|_{\mathcal{L}_\infty}\leq c_{{\rm abs}}\|W\|_{\mathcal{L}_\infty}\epsilon.
\end{align*}
So, for $n>k\epsilon^{-1},$ we have
$$|\langle T^{\frac1nA,\frac1nA}_{F^{[1]}}(W) \xi,\eta\rangle|\leq c_{{\rm abs}}\|W\|_{\mathcal{L}_\infty}\epsilon.$$
Since $\epsilon$ is arbitrarily small, the assertion follows.
\end{proof}

\begin{lem}\label{welldefine}
Let $V$ be a bounded linear operator on $L_2(\HH^n)$, and $A\geq 1$ be a (possibly unbounded) linear operator on $L_2(\HH^n)$, then we have
$$\lim_{n\to\infty}\int_{-n}^n\frac{\lambda A^{1/4}}{\lambda^2+{A}}V\frac{\lambda A^{1/4}}{\lambda^2+{A}}d\lambda=\frac{\pi}{2} T^{A,A}_{\psi}(V)$$
in weak operator topology, where we denote $\psi(\alpha_0,\alpha_1):=\frac{2\alpha_0^{1/4}\alpha_1^{1/4}}{\alpha_0^{1/2}+\alpha_1^{1/2}}$, for $\alpha_0$, $\alpha_1\geq0$.

Consequently, for any $k\in\{1,2,\cdots,2n\}$, the operator  $A_k$ defined in \eqref{defofAk} is well defined and can be further written as
\begin{align}\label{Ak double integral}
A_k=T^{1-\Delta_{\HH},1-\Delta_{\HH}}_{\psi}((1-\Delta_{\HH})^{-\frac14}X_k(1-\Delta_{\HH})^{-\frac14}).
\end{align}
\end{lem}
\begin{proof}
Be the definition of double operator integral, we see that
$$\int_{-n}^n\frac{\lambda A^{1/4}}{\lambda^2+A}V\frac{\lambda A^{1/4}}{\lambda^2+A}d\lambda=T^{A,A}_{\phi_n}(V),$$
where
\begin{align*}
\phi_n(\alpha_0,\alpha_1)
&=\int_{-n}^n\frac{\lambda \alpha_0^{1/4}}{\lambda^2+\alpha_0}\frac{\lambda \alpha_1^{1/4}}{\lambda^2+\alpha_1}d\lambda\\
&=\frac{2\alpha_0^{1/4}\alpha_1^{1/4}}{\alpha_0-\alpha_1}\bigg(\alpha_0^{1/2}\tan^{-1}\Big(\frac{n}{\alpha_0^{1/2}}\Big)-\alpha_1^{1/2}\tan^{-1}\Big(\frac{n}{\alpha_1^{1/2}}\Big)\bigg)\\
&=\psi(\alpha_0,\alpha_1)\bigg(\frac{\pi}{2}-F^{[1]}\Big(\frac1n\alpha_0^{1/2},\frac1n\alpha_1^{1/2}\Big)\bigg),
\end{align*}
where $F=\alpha\tan^{-1}(\alpha)$. We write
$W=T^{A,A}_{\psi}(V)$,
then
$$T^{A,A}_{\phi_n}(V)=\frac{\pi}{2} W-T^{\frac1nA^{1/2},\frac1nA^{1/2}}_{F^{[1]}}(W).$$
This, in combination with Lemma \ref{weakconverge}, ends the proof of the first statement.  To show the second one, we simply apply the first statement for $V = (1-\Delta_{\HH})^{-\frac14}X_k(1-\Delta_{\HH})^{-\frac14}$ and $A= 1-\Delta_{\HH}.$
\end{proof}

\begin{lem}\label{incursionlemma}
The following inclusion holds:
\begin{align}\label{firstincursion}
(1-(-\Delta_{\HH})^{1/4}(1-\Delta_{\HH})^{-1/4})\in L_{n+1,\infty}({\rm VN}(\mathbb{H}^n),\tau).
\end{align}
\end{lem}
\begin{proof}
For simplicity, we denote $g(\lambda)=1-\lambda^{1/4}(1+\lambda)^{-1/4}$, $\lambda\geq 0$, and note that
for any $0<s<1$, we have
$$\chi_{(s,\infty)}(g(-\Delta_{\HH}))=\chi_{\big(0,\frac{(1-s)^{4}}{1-(1-s)^{4}}\big)}(-\Delta_\HH).$$
Combining this with Corollary \ref{tau of gDelta}, we see that
\begin{align*}
\tau\big(\chi_{(s,\infty)}(g(-\Delta_{\HH}))\big)=\tau\big(\chi_{\big(0,\frac{(1-s)^{4}}{1-(1-s)^{4}}\big)}(-\Delta_\HH)\big)=\int_{0}^{\frac{(1-s)^{4}}{1-(1-s)^{4}}}t^{n}dt=\frac{1}{n+1}\left(\frac{(1-s)^{4}}{1-(1-s)^{4}}\right)^{n+1}.
\end{align*}
From this it is easily to verify that
\begin{align*}
\mu(t;|g(-\Delta_\HH)|)=\inf\left\{s\geq 0:\tau(\chi_{(s,\infty)}(g(-\Delta_\HH)))\leq t\right\}\leq Ct^{-\frac{1}{n+1}}.
\end{align*}%
Thus, we end the proof of \eqref{firstincursion}, and then Lemma \ref{incursionlemma}.
\end{proof}

\begin{lem}\label{doilemma}
Let $A_1,A_2\geq 0$ be (possibly unbounded) linear operators on $L_2(\HH^n)$ affiliated with ${\rm VN}(\HH^n).$ For any $1< p\leq\infty$, the double operator integral
$$T^{A_1,A_2}_{\frac{\alpha_0^{1/4}\alpha_1^{1/4}}{\alpha_0^{1/2}+\alpha_1^{1/2}}}$$
is bounded on $L_{p,\infty}({\rm VN}(\HH^n),\tau)$.
\end{lem}
\begin{proof}
Consider $$\phi(\lambda):=\frac{\lambda^{1/4}}{1+\lambda^{1/2}}\ {\rm and}\ f(\alpha_0,\alpha_1):=\frac{\alpha_0^{1/4}\alpha_1^{1/4}}{\alpha_0^{1/2}+\alpha_1^{1/2}},$$
then
$$f(\alpha_0,\alpha_1)=\phi\Big(\frac{\alpha_1}{\alpha_0}\Big).$$
By \cite[Lemma 7 and Lemma 9]{PS-crelle},
\begin{align*}
\bigg\|T^{A_1,A_2}_{\frac{\alpha_0^{1/4}\alpha_1^{1/4}}{\alpha_0^{1/2}+\alpha_1^{1/2}}}\bigg\|_{L_{\infty}({\rm VN}(\HH^n),\tau)\rightarrow L_{\infty}({\rm VN}(\HH^n),\tau)}\lesssim \|\hat{g}\|_{L_{1}(\mathbb{R})}\lesssim 1,
\end{align*}
where $g:=\phi\circ \exp\in \mathcal{S}(\mathbb{R}).$ With the $L_{\infty}({\rm VN}(\HH^n),\tau)$-boundedness of $T^{A_1,A_2}_{\frac{\alpha_0^{1/4}\alpha_1^{1/4}}{\alpha_0^{1/2}+\alpha_1^{1/2}}}$, the $L_{p,\infty}({\rm VN}(\HH^n),\tau)$-boundedness follows from duality and interpolation immediately.
\end{proof}

\begin{lem}\label{VNpertubation}
Let $V \in {\rm VN}(\HH^n)$. We have
\begin{align}\label{doubleperturbation}
T^{1-\Delta_{\HH},1-\Delta_{\HH}}_{\frac{\alpha_0^{1/4}\alpha_1^{1/4}}{\alpha_0^{1/2}+\alpha_1^{1/2}}}(V)-T^{-\Delta_{\HH},-\Delta_{\HH}}_{\frac{\alpha_0^{1/4}\alpha_1^{1/4}}{\alpha_0^{1/2}+\alpha_1^{1/2}}}(V)\in L_{n+1,\infty}({\rm VN}(\mathbb{H}^n),\tau).
\end{align}
\end{lem}
\begin{proof} By Definition \ref{double integral} and the linearity property of double operator integrals, we write
\begin{align*}
T^{1-\Delta_{\HH},1-\Delta_{\HH}}_{\frac{\alpha_0^{1/4}\alpha_1^{1/4}}{\alpha_0^{1/2}+\alpha_1^{1/2}}}(V)-T^{-\Delta_{\HH},-\Delta_{\HH}}_{\frac{\alpha_0^{1/4}\alpha_1^{1/4}}{\alpha_0^{1/2}+\alpha_1^{1/2}}}(V)
&=T^{-\Delta_{\HH},-\Delta_{\HH}}_{\frac{(\alpha_0+1)^{1/4}(\alpha_1+1)^{1/4}}{(\alpha_0+1)^{1/2}+(\alpha_1+1)^{1/2}}-\frac{\alpha_0^{1/4}\alpha_1^{1/4}}{\alpha_0^{1/2}+\alpha_1^{1/2}}}(V)\\
&=T^{-\Delta_{\HH},-\Delta_{\HH}}_{\frac{(\alpha_0+1)^{1/4}(\alpha_1+1)^{1/4}}{(\alpha_0+1)^{1/2}+(\alpha_1+1)^{1/2}}-\frac{\alpha_0^{1/4}(\alpha_1+1)^{1/4}}{(\alpha_0+1)^{1/2}+(\alpha_1+1)^{1/2}}}(V)+T^{-\Delta_{\HH},-\Delta_{\HH}}_{\frac{\alpha_0^{1/4}(\alpha_1+1)^{1/4}}{(\alpha_0+1)^{1/2}+(\alpha_1+1)^{1/2}}-\frac{\alpha_0^{1/4}\alpha_1^{1/4}}{(\alpha_0+1)^{1/2}+(\alpha_1+1)^{1/2}}}(V)\\
&\quad+T^{-\Delta_{\HH},-\Delta_{\HH}}_{\frac{\alpha_0^{1/4}\alpha_1^{1/4}}{(\alpha_0+1)^{1/2}+(\alpha_1+1)^{1/2}}-\frac{\alpha_0^{1/4}\alpha_1^{1/4}}{\alpha_0^{1/2}+(\alpha_1+1)^{1/2}}}(V)
+T^{-\Delta_{\HH},-\Delta_{\HH}}_{\frac{\alpha_0^{1/4}\alpha_1^{1/4}}{\alpha_0^{1/2}+(\alpha_1+1)^{1/2}}-\frac{\alpha_0^{1/4}\alpha_1^{1/4}}{\alpha_0^{1/2}+\alpha_1^{1/2}}}(V)\\
&=:{\rm (I)}+{\rm (II)}+{\rm (III)}+{\rm (IV)}.
\end{align*}

Now we handle these double operator integrals separately.
$${\rm (I)}=T^{-\Delta_{\HH},-\Delta_{\HH}}_{\frac{(\alpha_0+1)^{1/4}(\alpha_1+1)^{1/4}}{(\alpha_0+1)^{1/2}+(\alpha_1+1)^{1/2}}}\Big(T^{-\Delta_{\HH},-\Delta_{\HH}}_{1-\frac{\alpha_0^{1/4}}{(\alpha_0+1)^{1/4}}}(V)\Big)=T^{1-\Delta_{\HH},1-\Delta_{\HH}}_{\frac{\alpha_0^{1/4}\alpha_1^{1/4}}{\alpha_0^{1/2}+\alpha_1^{1/2}}}\Big(\frac{(1-\Delta_{\HH})^{1/4}-(-\Delta_{\HH})^{1/4}}{(1-\Delta_{\HH})^{1/4}}V\Big).$$
Its follows from the inclusion \eqref{firstincursion} that
\begin{align*}
\frac{(1-\Delta_{\HH})^{1/4}-(-\Delta_{\HH})^{1/4}}{(1-\Delta_{\HH})^{1/4}}V\in L_{n+1,\infty}({\rm VN}(\mathbb{H}^n),\tau).
\end{align*}
This, in combination with Lemma \ref{doilemma}, concludes that
the first integral ${\rm (I)}$ falls into $L_{n+1,\infty}({\rm VN}(\mathbb{H}^n),\tau).$

For the term ${\rm (II)}$, we write
$${\rm (II)}=T^{-\Delta_{\HH},-\Delta_{\HH}}_{\frac{(\alpha_0+1)^{1/4}(\alpha_1+1)^{1/4}}{(\alpha_0+1)^{1/2}+(\alpha_1+1)^{1/2}}}\Big(T^{-\Delta_{\HH},-\Delta_{\HH}}_{\frac{\alpha_0^{1/4}}{(\alpha_0+1)^{1/4}}(1-\frac{\alpha_1^{1/4}}{(\alpha_1+1)^{1/4}})}(V)\Big)=T^{1-\Delta_{\HH},1-\Delta_{\HH}}_{\frac{\alpha_0^{1/4}\alpha_1^{1/4}}{\alpha_0^{1/2}+\alpha_1^{1/2}}}\Big(\frac{(-\Delta_{\HH})^{\frac14}}{(1-\Delta_{\HH})^{\frac14}}V\frac{(1-\Delta_{\HH})^{1/4}-(-\Delta_{\HH})^{1/4}}{(1-\Delta_{\HH})^{1/4}}\Big).$$
Then similarly to first integral ${\rm (I)}$, the term ${\rm (II)}$ also falls into  $L_{n+1,\infty}({\rm VN}(\mathbb{H}^n),\tau).$

For the term ${\rm (III)}$, we write
$${\rm (III)}=T^{-\Delta_{\HH},-\Delta_{\HH}}_{\frac{\alpha_0^{1/4}(\alpha_1+1)^{1/4}}{\alpha_0^{1/2}+(\alpha_1+1)^{1/2}}}\bigg(T^{-\Delta_{\HH},-\Delta_{\HH}}_{\frac{(\alpha_0+1)^{1/4}(\alpha_1+1)^{1/4}}{(\alpha_0+1)^{1/2}+(\alpha_1+1)^{1/2}}}\Big(T^{-\Delta_{\HH},-\Delta_{\HH}}_{\frac{\alpha_1^{1/4}(\alpha_0^{1/2}-(\alpha_0+1)^{1/2})}{(\alpha_0+1)^{1/4}(\alpha_1+1)^{1/2}}}(V)\Big)\bigg)=$$
$$=T^{-\Delta_{\HH},1-\Delta_{\HH}}_{\frac{\alpha_0^{1/4}\alpha_1^{1/4}}{\alpha_0^{1/2}+\alpha_1^{1/2}}}\bigg(T^{1-\Delta_{\HH},1-\Delta_{\HH}}_{\frac{\alpha_0^{1/4}\alpha_1^{1/4}}{\alpha_0^{1/2}+\alpha_1^{1/2}}}\Big(\frac{(-\Delta_{\HH})^{1/2}-(1-\Delta_{\HH})^{1/2}}{(1-\Delta_{\HH})^{1/4}}V\frac{(-\Delta_{\HH})^{1/4}}{(1-\Delta_{\HH})^{1/2}}\Big)\bigg).
$$
Then similarly to first integral ${\rm (I)}$, the term ${\rm (III)}$ also falls into  $L_{n+1,\infty}({\rm VN}(\mathbb{H}^n),\tau).$

For the term ${\rm (IV)}$, we write
$${\rm (IV)}=T^{-\Delta_{\HH},-\Delta_{\HH}}_{\frac{\alpha_0^{1/4}\alpha_1^{1/4}}{\alpha_0^{1/2}+\alpha_1^{1/2}}}\bigg(T^{-\Delta_{\HH},-\Delta_{\HH}}_{\frac{(\alpha_1+1)^{1/2}}{(\alpha_0+1)^{1/2}+(\alpha_1+1)^{1/2}}}\Big(T^{-\Delta_{\HH},-\Delta_{\HH}}_{\frac{\alpha_1^{1/2}-(\alpha_1+1)^{1/2}}{(\alpha_1+1)^{1/2}}}(V)\Big)\bigg)=$$
$$=T^{-\Delta_{\HH},-\Delta_{\HH}}_{\frac{\alpha_0^{1/4}\alpha_1^{1/4}}{\alpha_0^{1/2}+\alpha_1^{1/2}}}\bigg(T^{(1-\Delta_{\HH})^{1/2},(1-\Delta_{\HH})^{1/2}}_{\frac{\alpha_0}{\alpha_0+\alpha_1}}\Big(V\frac{(-\Delta_{\HH})^{1/2}-(1-\Delta_{\HH})^{1/2}}{(1-\Delta_{\HH})^{1/2}}\Big)\bigg).$$
Then similarly to first integral ${\rm (I)}$ and Lemma \ref{fraction schur lemma}, the term ${\rm (IV)}$ also falls into  $L_{n+1,\infty}({\rm VN}(\mathbb{H}^n),\tau).$
\end{proof}

\begin{lem}\label{ak def lemma} Let $a_k$ and $A_k$ be the operators defined as in \eqref{defofak} and \eqref{Ak double integral}, respectively.
Then $$A_k-a_k\in L_{n+1,\infty}({\rm VN}(\mathbb{H}^n),\tau).$$
\end{lem}
\begin{proof} Note that
\begin{align}\label{pertubation1212}
&(1-\Delta_{\HH})^{-1/4}X_k(1-\Delta_{\HH})^{-1/4}-(-\Delta_{\HH})^{-1/4}X_k(-\Delta_{\HH})^{-1/4}\nonumber\\
&=((1-\Delta_{\HH})^{-1/4}(-\Delta_{\HH})^{1/4}-1)\cdot (-\Delta_{\HH})^{-1/4}X_k(-\Delta_{\HH})^{-1/4}\cdot (-\Delta_{\HH})^{1/4}(1-\Delta_{\HH})^{-1/4}\nonumber\\
&\quad+(-\Delta_{\HH})^{-1/4}X_k(-\Delta_{\HH})^{-1/4}\cdot ((1-\Delta_{\HH})^{-1/4}(-\Delta_{\HH})^{1/4}-1).
\end{align}

 Recall that
 $(-\Delta_{\HH})^{-1/4}X_k(-\Delta_{\HH})^{-1/4}=R_{k}$ is bounded on $L_2(\HH^n)$. This, in combination with the boundedness of the operator $(-\Delta_{\HH})^{1/4}(1-\Delta_{\HH})^{-1/4}$ and inclusion \eqref{firstincursion} as well as equality \eqref{pertubation1212}, yields that
$$(1-\Delta_{\HH})^{-\frac14}X_k(1-\Delta_{\HH})^{-\frac14}-(-\Delta_{\HH})^{-\frac14}X_k(-\Delta_{\HH})^{-\frac14}\in L_{n+1,\infty}({\rm VN}(\mathbb{H}^n),\tau).$$

Set
$$B_k:=T^{1-\Delta_{\HH},1-\Delta_{\HH}}_{\frac{2\alpha_0^{1/4}\alpha_1^{1/4}}{\alpha_0^{1/2}+\alpha_1^{1/2}}}((-\Delta_{\HH})^{-1/4}X_k(-\Delta_{\HH})^{-1/4}).$$
By Lemma \ref{doilemma},  $T^{1-\Delta_{\HH},1-\Delta_{\HH}}_{\frac{\alpha_0^{1/4}\alpha_1^{1/4}}{\alpha_0^{1/2}+\alpha_1^{1/2}}}$ is bounded on $L_{n+1,\infty}({\rm VN}(\mathbb{H}^n),\tau)$,
which implies that
$$A_k-B_k\in L_{n+1,\infty}({\rm VN}(\mathbb{H}^n),\tau),$$
where $A_k$ is as in \eqref{Ak double integral}.
On the other hand, by Lemma \ref{VNpertubation}, we have
$$B_k-a_k\in L_{n+1,\infty}({\rm VN}(\mathbb{H}^n),\tau).$$
A combination of the last two inclusions ends the proof of Lemma \ref{ak def lemma}.
\end{proof}

%
%

\begin{cor}\label{perturbation111}
For any $k,\ell\in\{1,2,\cdots,2n\}$ and  $f\in C_c^{\infty}(\HH^n)$,  we have
\begin{align}
R_{\ell}(A_{k}-a_{k})M_{X_kf}(-\Delta_\HH)^{-\frac{1}{2}}\in \mathcal{L}_{n+1,\infty},
\end{align}
where $a_k$ and $A_k$ are defined in \eqref{defofak} and \eqref{defofAk} {\rm (see also \eqref{Ak double integral})}, respectively.
\end{cor}
\begin{proof}
Since $R_{\ell}$ is a bounded operator on $L_2(\HH^n)$, it suffices to show that
\begin{align}\label{arrive}
(A_{k}-a_{k})M_{X_kf}(-\Delta_\HH)^{-\frac{1}{2}}\in \mathcal{L}_{n+1,\infty}.
\end{align}
Indeed, by H\"older's inequality, we have
\begin{align}\label{abc1}
\Big\|(A_k-a_k)M_{X_kf}(-\Delta_{\HH})^{-\frac12}\Big\|_{\mathcal{L}_{n+1,\infty}}&\leq \Big\|(A_k-a_k)M_{|X_kf|^{\frac12}}\Big\|_{\mathcal{L}_{2n+2,\infty}}\Big\|M_{|X_kf|^{\frac12}}(-\Delta_{\HH})^{-\frac12}\Big\|_{\mathcal{L}_{2n+2,\infty}}.
\end{align}
For the first term, we apply  Theorem \ref{L_2_L_infty_cwikel_estimate} and then Lemma \ref{ak def lemma} to see that
\begin{align}\label{abc2}
\Big\|(A_k-a_k)M_{|X_kf|^{\frac12}}\Big\|_{\mathcal{L}_{2n+2,\infty}}\lesssim \|A_k-a_k\|_{L_{2n+2,\infty}({\rm VN}(\HH^{n}),\tau)}\||X_kf|^{\frac12}\|_{L_{2n+2}(\HH^n)}\lesssim\||X_kf|^{\frac12}\|_{L_{2n+2}(\HH^n)}=\|X_kf\|_{L_{n+1}(\HH^n)}^{1/2}.
\end{align}
For the second term, we apply Corollary \ref{specific_cwikel} with $f$ and $\beta$ being chosen to be $|X_k f|^{\frac12}$ and $1$, respectively, to see that
\begin{align}\label{abc3}
\Big\|M_{|X_kf|^{\frac12}}(-\Delta_{\HH})^{-\frac12}\Big\|_{\mathcal{L}_{2n+2,\infty}}\lesssim\||X_kf|^{\frac12}\|_{L_{2n+2}(\HH^n)}=\|X_kf\|_{L_{n+1}(\HH^n)}^{1/2}.
\end{align}
Combining estimates \eqref{abc1}, \eqref{abc2} and \eqref{abc3}, we arrive at \eqref{arrive} and then complete the proof of Corollary \ref{perturbation111}.
\end{proof}

Combining all the results in this section, we can now show Proposition \ref{approximation theorem}.

\noindent{\it Proof of Proposition \ref{approximation theorem}.}
Recall from \eqref{Rieszdecompose} that for any $\ell\in\{1,2,\cdots,2n\}$,
\begin{align*}
    [R_\ell,M_f]&= M_{X_\ell f}(-\Delta_\HH)^{-\frac12}-R_\ell[(-\Delta_\HH)^{\frac12},M_f](-\Delta_\HH)^{-\frac12}\\
    &=M_{X_\ell f}(-\Delta_\HH)^{-\frac12}-R_{\ell}[(1-\Delta_{\HH})^{\frac12},M_f](-\Delta_{\HH})^{-\frac12}+\left(R_{\ell}[(1-\Delta_{\HH})^{\frac12},M_f](-\Delta_{\HH})^{-\frac12}-R_{\ell}[(-\Delta_{\HH})^{\frac12},M_f](-\Delta_{\HH})^{-\frac12}\right).
\end{align*}
By Lemma \ref{pertubationlemma0}, the third term in the bracket belongs to $(\mathcal{L}_{2n+2,\infty})_{0}$. By Corollary \ref{summuarize} the second term can be written as
\begin{align}\label{opsss}
R_{\ell}[(1-\Delta_{\HH})^{\frac12},M_f](-\Delta_{\HH})^{-\frac12}
&=-\sum_{k=0}^{2n}R_{\ell}A_{k}M_{X_kf}(-\Delta_\HH)^{-\frac{1}{2}}+\mathcal{E}_0\nonumber\\
&=-\sum_{k=0}^{2n}R_{\ell}a_{k}M_{X_kf}(-\Delta_\HH)^{-\frac{1}{2}}-\sum_{k=0}^{2n}R_{\ell}(A_{k}-a_{k})M_{X_kf}(-\Delta_\HH)^{-\frac{1}{2}}+\mathcal{E}_0,
\end{align}
where $A_k$ is defined in \eqref{defofAk} and $\mathcal{E}_0=R_{\ell}\mathcal{E}_1+R_{\ell}\mathcal{E}_2+R_{\ell}\mathcal{E}_3\in (\mathcal{L}_{2n+2,\infty})_{0}$. Besides, applying Corollary \ref{perturbation111} and taking into account that $\mathcal{L}_{n+1,\infty}\subset (\mathcal{L}_{2n+2,\infty})_0$, we see that the second term in \eqref{opsss} also belongs to $(\mathcal{L}_{2n+2,\infty})_0$. This completes the proof of Proposition \ref{approximation theorem}.
\hfill $\square$


\subsection{Trace formula for Riesz transform commutator}
At the end of this section, we combine the general trace formula proven in Section \ref{tracelemmasection} with the approximation of Riesz transform commutator in Proposition \ref{approximation theorem} to give a more explicit trace formula for Riesz transform commutator.
\begin{prop}\label{keytraceformula}
For any $f\in L_{\infty}(\Heis^n)\cap \dot{W}^{1,2n+2}(\HH^n)$, $\ell\in\{1,2,\cdots,2n\}$ and normalised continuous trace $\varphi$ on $\mathcal{L}_{1,\infty}$,  we have
\begin{align}\label{traceformula00000}
\varphi(|[R_\ell,M_f]|^{2n+2})=c_n\bigg\|\sum_{k=0}^{2n}f_k^{\ell}\otimes y_k^{\ell}\bigg\|_{L_{2n+2}(L_{\infty}(\mathbb{H}^n)\bar{\otimes}\mathcal{B}(L_2(\mathbb{R}^n))\overline{\otimes}\mathbb{C}^2)}^{2n+2},
\end{align}
where $\{f_k^{\ell}\}_{k=0}^{2n}\in L_{2n+2}(\mathbb{H}^n)$ and $\{y_k^{\ell}\}_{k=0}^{2n}\in L_{2n+2}(\mathcal{B}(L_2(\mathbb{R}^n))\overline{\otimes}\mathbb{C}^2,{\rm Tr}\otimes\Sigma)$  satisfying
$$ f_k^{\ell}=\left\{\begin{array}{ll}\overline{X_\ell f}, &k=0\\ \overline{X_kf}, &k=1,2,\cdots,2n\end{array}\right.\ {\rm and}\ \ \pi_{\rm red}(y_k^{\ell})=\left\{\begin{array}{ll}(-\Delta_{\HH})^{-\frac{1}{2}}|T|^{\frac{1}{2}}, &k=0 \\(-\Delta_{\HH})^{-\frac{1}{2}}|T|^{\frac{1}{2}}(R_\ell a_k)^*, &k=1,2,\cdots,2n.\end{array}\right.$$
\end{prop}
\begin{proof}
We first show the conclusion holds for $f \in C^\infty_c(\Heis^n)$. To begin with, by Proposition \ref{approximation theorem}, there exists $E\in (\mathcal{L}_{2n+2,\infty})_0$ such that
$$[R_\ell,M_f]= M_{X_\ell f}(-\Delta_{\HH})^{-\frac12}+\sum_{k=1}^{2n}R_\ell a_kM_{X_kf}(-\Delta_{\HH})^{-\frac12}+E,$$
where
\begin{align}
a_k=T^{-\Delta_{\HH},-\Delta_{\HH}}_{\frac{2\alpha_0^{\frac14}\alpha_1^{\frac14}}{\alpha_0^{\frac12}+\alpha_1^{\frac12}}}((-\Delta_{\HH})^{-1/4}X_k(-\Delta_{\HH})^{-1/4}).
\end{align}
For any $f\in L_{2n+2}(\mathbb{H}^n)$ and $x\in L_{2n+2}(\mathcal{B}(L_2(\mathbb{R}^n))\overline{\otimes}\mathbb{C}^2,{\rm Tr}\otimes\Sigma)$, we let $$A(x,f)=M_f\pi_{{\rm red}}(x)|T|^{-\frac12}$$ be the map defined in Section \ref{tracelemmasection}.
Then we write
\begin{align}\label{adjointRiesz}
[R_\ell,M_f]^*&= (-\Delta_{\HH})^{-\frac12}M_{\overline{X_\ell f}}+\sum_{k=1}^{2n}(-\Delta_{\HH})^{-\frac12}M_{\overline{X_kf}}(R_\ell a_k)^*+E^*\nonumber\\
&= \big(M_{\overline{X_\ell f}}(-\Delta_{\HH})^{-\frac12}|T|^{\frac12}\big)|T|^{-\frac12}+\sum_{k=1}^{2n}\big(M_{\overline{X_kf}}(-\Delta_{\HH})^{-\frac12}|T|^{\frac12}\big)|T|^{-\frac12}(R_\ell a_k)^*\nonumber\\
&\quad+[(-\Delta_{\HH})^{-\frac12},M_{\overline{X_\ell f}}]+\sum_{k=1}^{2n}[(-\Delta_{\HH})^{-\frac12},M_{\overline{X_kf}}](R_\ell a_k)^*+E^*\nonumber\\
&=\sum_{k=0}^{2n}A(y_k^{\ell},f_k^{\ell})+[(-\Delta_{\HH})^{-\frac12},M_{\overline{X_\ell f}}]+\sum_{k=1}^{2n}[(-\Delta_{\HH})^{-\frac12},M_{\overline{X_kf}}](R_\ell a_k)^*+E^*,
\end{align}
where in the last equality we used the fact that $|T|^{-\frac12}$ commutes with $(R_{\ell}a_k)^*$. From the proof of Lemma \ref{doilemma}, we see that
$T^{-\Delta_{\HH},-\Delta_{\HH}}_{\frac{2\alpha_0^{1/4}\alpha_1^{1/4}}{\alpha_0^{1/2}+\alpha_1^{1/2}}}:\mathcal{L}_{\infty}\rightarrow \mathcal{L}_{\infty}$
is bounded.  This, in combination with the $L_2(\HH^n)$-boundedness of $R_\ell$, implies the $L_2(\HH^n)$-boundedness of $(R_\ell a_k)^*$. Besides, it follows from Lemma \ref{cwikel-like commutator lemma} that for any $k\in\{1,2,\cdots,2n\}$, the commutator $[(-\Delta_{\HH})^{-\frac12},M_{\overline{X_kf}}]$ belongs to $(\mathcal{L}_{2n+2,\infty})_0$. Therefore, we have
 $$[(-\Delta_{\HH})^{-\frac12},M_{\overline{X_\ell f}}]+\sum_{k=1}^{2n}[(-\Delta_{\HH})^{-\frac12},M_{\overline{X_kf}}](R_\ell a_k)^*+E^*\in (\mathcal{L}_{2n+2,\infty})_0.$$
This, together with Corollary \ref{specific_cwikel}, equality \eqref{adjointRiesz} and H\"{o}lder's inequality, implies that
\begin{align*}
|[R_\ell,M_f]^*|^{2n+2}=\left|\sum_{k=0}^{2n}A(y_k^{\ell},f_k^{\ell})\right|^{2n+2}+\mathcal{E},
\end{align*}
for some $\mathcal{E}\in (\mathcal{L}_{1,\infty})_0$.

Since $\varphi$ is a normalised continuous trace on $\mathcal{L}_{1,\infty}$, it vanishes on  $(\mathcal{L}_{1,\infty})_0$ (see for example \cite[Corollary 5.7.7]{LSZ1}). Therefore,
$$\varphi(|[R_\ell,M_f]^{\ast}|^{2n+2})=\varphi\bigg(\bigg|\sum_{k=0}^{2n}A(y_k^{\ell},f_k^{\ell})\bigg|^{2n+2}\bigg).$$
Combining this with Corollary \ref{connes product corollary}, we conclude that
$$\varphi(|[R_\ell,M_f]|^{2n+2})=\varphi(|[R_\ell,M_f]^{\ast}|^{2n+2})=c_n\bigg\|\sum_{k=0}^{2n}f_k^{\ell}\otimes y_k^{\ell}\bigg\|_{L_{2n+2}(L_{\infty}(\mathbb{H}^n)\bar{\otimes}\mathcal{B}(L_2(\mathbb{R}^n))\overline{\otimes}\mathbb{C}^2)}^{2n+2}.$$

Now we apply an approximation argument to remove the smooth assumption   $f \in C^\infty_c(\Heis^n)$. To this end, we suppose $f\in  L_{\infty}(\Heis^n)\cap \dot{W}^{1,2n+2}(\HH^n)$ and let $\{f_m\}_{m\geq 1}$ be the sequence chosen in Lemma \ref{density2}, then $f_m\in C_c^\infty(\HH^n)$ for $m\geq 1$ and $f_m\rightarrow f$ in $\dot{W}^{1,2n+2}(\HH^n)$. From this we conclude that $f_{m,k}^{\ell}\rightarrow f_k^\ell$ in $L_{2n+2}(\HH^n)$ and that
\begin{align*}
\bigg\|\sum_{k=0}^{2n}f_{m,k}^{\ell}\otimes y_k^{\ell}-\sum_{k=0}^{2n}f_k^{\ell}\otimes y_k^{\ell}\bigg\|_{L_{2n+2}(L_{\infty}(\mathbb{H}^n)\overline{\otimes}\mathcal{B}(L_2(\mathbb{R}^n))\overline{\otimes}\mathbb{C}^2)}&\leq \sum_{k=0}^{2n}\|(f_{m,k}^{\ell}-f_k^{\ell})\otimes y_k^{\ell}\|_{L_{2n+2}(L_{\infty}(\mathbb{H}^n)\overline{\otimes}\mathcal{B}(L_2(\mathbb{R}^n))\overline{\otimes}\mathbb{C}^2)}\\
&= \sum_{k=0}^{2n}\|f_{m,k}^{\ell}-f_k^{\ell}\|_{L_{2n+2}(\HH^n)}\|y_k^{\ell}\|_{L_{2n+2}(\mathcal{B}(L_2(\mathbb{R}^n))\overline{\otimes}\mathbb{C}^2)}\\
&\leq \|f_{m}-f\|_{\dot{W}^{1,2n+2}(\HH^n)}\sum_{k=0}^{2n}\|y_k^{\ell}\|_{L_{2n+2}(\mathcal{B}(L_2(\mathbb{R}^n))\overline{\otimes}\mathbb{C}^2)}\rightarrow 0,
\end{align*}
as $m\rightarrow\infty$, where for simplicity we denote that
$$ f_{m,k}^{\ell}=\left\{\begin{array}{ll}\overline{X_\ell f_m}, &k=0\\ \overline{X_k{f_m}}, &k=1,2,\cdots,2n.\end{array}\right.$$

We recall from the proof of Proposition \ref{sufficiency theorem} that $[R_\ell, M_{f_m}]\rightarrow [R_\ell, M_f]$  in $\mathcal{L}_{2n+2,\infty}$. In particular, $\sup\limits_{m\geq 1}\|[R_\ell,M_{f_m}]\|_{\mathcal{L}_{2n+2,\infty}}<+\infty$. By taking the $(2n+2)$-power and applying formula (21) in \cite{LMSZ}, we show that $|[R_\ell, M_{f_m}]|^{2n+2}\rightarrow |[R_\ell, M_f]|^{2n+2}$ in $\mathcal{L}_{1,\infty}$. Therefore, $\varphi(|[R_\ell,M_{f_m}]|^{2n+2})\rightarrow \varphi(|[R_\ell,M_{f}]|^{2n+2})$ as $m\rightarrow \infty$. Since we have already shown that \eqref{traceformula00000} holds for $f\in C_c^\infty(\HH^n)$, one deduce that
\begin{align*}
\varphi(|[R_\ell,M_{f}]|^{2n+2})&=\lim_{m\rightarrow \infty}\varphi(|[R_\ell,M_{f_m}]|^{2n+2})\\&=\lim_{m\rightarrow \infty}c_n\bigg\|\sum_{k=0}^{2n}f_{m,k}^{\ell}\otimes y_k^{\ell}\bigg\|_{L_{2n+2}(L_{\infty}(\mathbb{H}^n)\bar{\otimes}\mathcal{B}(L_2(\mathbb{R}^n))\overline{\otimes}\mathbb{C}^2)}^{2n+2}\\&=c_n\bigg\|\sum_{k=0}^{2n}f_k^{\ell}\otimes y_k^{\ell}\bigg\|_{L_{2n+2}(L_{\infty}(\mathbb{H}^n)\bar{\otimes}\mathcal{B}(L_2(\mathbb{R}^n))\overline{\otimes}\mathbb{C}^2)}^{2n+2}.
\end{align*}

This complete the proof of Proposition \ref{keytraceformula}.
\end{proof}

\begin{remark}
The constant $c_n$ in Proposition \ref{keytraceformula} is independent of the continuous normalised trace $\varphi$. This implies that the operator $|[R_\ell,M_f]|^{2n+2}\in\mathcal{L}_{1,\infty}$ is a measurable operator in the sense of Conne's quantised calculus (see \cite[Chapter 10]{LSZ1}). It also implies convergence properties of the ordered eigenvalues of $|[R_\ell,M_f]|^{2n+2}$ (see \cite{SSUZ1}).
\end{remark}

\begin{remark}
       Theorem \ref{keytraceformula1} follows immediately
       from Proposition \ref{keytraceformula}, on taking
       \[
          x_k^{\ell} = \begin{cases}
                          y_k^{\ell},\quad k\neq \ell,\\
                          y_\ell^\ell+y_0^{\ell},\quad k=\ell.
                        \end{cases}
       \]
       It is immediate from the linear independence of $\{y_k^{\ell}\}_{k=0}^{2n},$ established in Lemma \ref{independent} below,
       that $\{x_{k}^{\ell}\}_{k=1}^{2n}$ are also linearly independent.
\end{remark}

%
%

\section{Proof of lower bound}\label{Proof of necessity direction}
\setcounter{equation}{0}
This section is devoted to giving a proof of the lower bound in Theorem \ref{main}. That is to show the following proposition.
\begin{prop}\label{necessary theorem}
Let $f\in L_\infty(\Heis^n)$. If we have $[R_{\ell},M_f]\in\mathcal{L}_{2n+2,\infty}$  for some $\ell\in\{1,2,\cdots,2n\}$, then $f\in \dot{W}^{1,2n+2}(\HH^n)$ and there exists a positive constant $c_n>0$ such that
\begin{align*}
\|[R_{\ell},M_f]\|_{\mathcal{L}_{2n+2,\infty}}\geq c_n\|f\|_{\dot{W}^{1,2n+2}(\HH^n)}.
\end{align*}
\end{prop}

The following Lemma is crucial for us to obtain the weak-Schatten lower bound of Riesz transform commutator.
\begin{lem}\label{lowerbd}
Let $\{f_k\}_{k=0}^m\subset L_p(\mathbb{H}^n)$ and $\{x_k\}_{k=0}^m\subset L_p(\mathcal{B}(L_2(\mathbb{R}^n))\overline{\otimes}\mathbb{C}^2,{\rm Tr}\otimes\Sigma)$ for some $1<p<\infty$.  If the sequence $\{x_k\}_{k=0}^m$ is linearly independent in $L_p(\mathcal{B}(L_2(\mathbb{R}^n)\overline{\otimes}\mathbb{C}^2),\mathrm{Tr}\otimes \Sigma)$, then there exists a constant $c(p,\{x_k\}_{k=0}^m)$ such that
$$\bigg\|\sum_{k=0}^mf_k\otimes x_k\bigg\|_{L_p(L_\infty(\mathbb{H}^n)\overline{\otimes} \mathcal{B}(L_2(\mathbb{R}^n))\overline{\otimes} \mathbb{C}^2)} \geq c(p,\{x_k\}_{k=0}^m)\sum_{k=0}^m\|f_k\|_{L_p(\Heis^n)}.$$
\end{lem}
\begin{proof}
    Write the tensor product $L_p$ norm as a Bochner norm, i.e.
    \[
        \bigg\|\sum_{k=0}^mf_k\otimes x_k\bigg\|_{L_p(L_\infty(\mathbb{H}^n)\overline{\otimes} \mathcal{B}(L_2(\mathbb{R}^n))\overline{\otimes} \mathbb{C}^2)}^p = \int_{\mathbb{H}^n} \bigg\|\sum_{k=1}^m f_k(g)x_k\bigg\|_{L_p(\mathcal{B}(L_2(\mathbb{R}^n))\overline{\otimes}\mathbb{C}^2)}^p \,dg.
    \]
    This is possible due to the separability of $L_p(\mathcal{B}(L_2(\mathbb{R}^n))\overline{\otimes} \mathbb{C}^2)$, see \cite[Lemma 6.2]{BerksonGillespieMuhly1986}.
    Since the elements $\{x_k\}_{k=1}^m$ are linearly independent in $L_p(\mathcal{B}(L_2(\mathbb{R}^n)\overline{\otimes}\mathbb{C}^2),\mathrm{Tr}\otimes \Sigma),$ there exists a constant $C(p,\{x_k\}_{k=1}^m)$ such that
    \[
        \bigg\|\sum_{k=1}^m \lambda_k x_k\bigg\|_{L_p(\mathcal{B}(L_2(\mathbb{R}^n))\overline{\otimes}\mathbb{C}^2)}^p\geq C(p,\{x_{k}\}_{k=1}^m)\bigg(\sum_{k=1}^m |\lambda_k|^p\bigg),\quad \lambda_k\in \mathbb{C}.
    \]
    Hence,
    \[
        \int_{\mathbb{H}^n} \bigg\|\sum_{k=1}^m f_k(g)x_k\bigg\|_{L_p(\mathcal{B}(L_2(\mathbb{R}^n))\overline{\otimes}\mathbb{C}^2)}^p \,dg \geq C(p,\{x_k\}_{k=1}^m)\int_{\Heis^n} \sum_{k=1}^m |f_k(g)|^p\,dg.
    \]
    That is,
    \[
        \bigg\|\sum_{k=0}^mf_k\otimes x_k\bigg\|_{L_p(L_\infty(\mathbb{H}^n)\overline{\otimes} \mathcal{B}(L_2(\mathbb{R}^n))\overline{\otimes} \mathbb{C}^2)} \geq C(p,\{x_k\}_{k=1}^m)^{\frac1p}\left(\sum_{k=1}^m \|f_k\|_{L_p(\Heis^n)}^p\right)^{\frac1p}.
    \]
    Taking $c(p,\{x_k\}_{k=1}^m) = C(p,\{x_k\}_{k=1}^m)^{\frac1p}m^{\frac1p-1}$ completes the proof.
\end{proof}

To continue, for any $\theta\in [0,2\pi]$ and $k\in\{1,2,\cdots,n\}$, we define a rotation map by the usual formula
$$(U_{k,\theta}\xi)(x,y,t)=\xi(x_1,y_1,\cdots,x_{k-1},y_{k-1},x_k\cos\theta-y_k\sin\theta,x_k\sin\theta+y_k\cos\theta,x_{k+1},y_{k+1},\cdots,x_n,y_n,t),$$
where $\xi$ is a function on $\mathbb{R}^{2n+1}$ and $(x,y,t)=(x_1,y_1,\cdots,x_{2n},y_{2n},t)\in \mathbb{R}^{2n+1}.$

The following lemma can be shown by a direct calculation, we omit the details and leave it to the readers.
\begin{lem}\label{rotation}
For any $\theta\in [0,2\pi]$ and $k\in\{1,2,\cdots,n\}$, we have
$$U_{k,\theta}^{-1}X_kU_{k,\theta}=\cos\theta X_k+\sin\theta Y_k,\quad U_{k,\theta}^{-1}Y_kU_{k,\theta}=-\sin\theta X_k+\cos\theta Y_k,$$
$$U_{k,\theta}^{-1}X_lU_{k,\theta}=X_l,\quad U_{k,\theta}^{-1}Y_lU_{k,\theta}=Y_l,\quad l\neq k.$$
\end{lem}
%

In the following lemma we use that fact that for all $1\leq \ell\leq 2n,$ the Riesz transform $R_{\ell}$ does not admit a right inverse. In the Euclidean case,
the same fact about Riesz transforms is an immediate consequence of their representation as Fourier multipliers. However, in the Heisenberg case this can instead be seen via $\pi_{\red}.$
Observe from the explicit matrix representation of $p_j$ and $q_j$, $1\leq j\leq n$ in Section \ref{auxiliarysection} that
the operators
\[
    p_jH^{-\frac12},\quad q_jH^{-\frac12},\quad 1\leq j\leq n.
\]
have trivial kernel (that is, ker $p_jH^{-\frac12}$=ker $q_jH^{-\frac12}$=$0$), and do not have bounded right inverse. It follows that
\[
    p_jH^{-\frac12}\otimes {\bf 1},\quad q_jH^{-\frac12}\otimes z,\quad 1\leq j\leq n
\]
also have trivial kernel and do not admit a bounded inverse.

It follows from this that the Riesz transforms $R_1,R_2,\ldots,R_{2n}$ have trivial kernel.
Taking $R_1$ for example, we have
\[
    R_1 = \pi_{\red}(ip_1H^{-\frac12}\otimes {\bf 1} ).
\]
From this it follows that the support projection of $R_1$ is equal to $1.$ Indeed, the support projection of $R_1$
is the image under $\pi_{\red}$ of the support projection of $ip_1H^{-\frac12}\otimes {\bf 1},$ however this operator has trivial kernel
and hence has support projection equal to $1.$ Thus $R_1$ has trivial kernel, and in particular any operator $S$ such that $R_1S = 0$ is zero.

We may also deduce from this that every Riesz transform $R_{\ell},$ $1\leq \ell \leq 2n$ does not admit a bounded right-inverse.
To see this in the case of $R_1,$ observe that $R_1$ is invariant under conjugation by the unitary semigroup $r\mapsto \sigma_r$ \eqref{unitary_dilation_definition}. If $R_1$
has it has a bounded right-inverse $A$ then
\[
    R_1\sigma_r A\sigma_{r^{-1}} = 1
\]
and therefore $\sigma_r A\sigma_{r^{-1}}$ is also a right-inverse for $R_1.$ Considering that the difference $S = A-\sigma_r A\sigma_{r^{-1}}$
satisfies $R_1S = 0,$ it follows that
\[
    \sigma_rA\sigma_{r^{-1}}=A,\quad r \in \mathbb{R}.
\]
Hence $A$ is dilation-invariant, and thus $A$ belongs to the image of $\pi_{\red}.$ It follows that if $R_1$ has a right-inverse in $\Bc(L_2(\HH^n)),$
then $p_1H^{-\frac12}\otimes {\bf 1}$ has a right-inverse in $\Bc(L_2(\mathbb{R}^n))\otimes \mathbb{C}^2,$ which is impossible.

Similarly, we may prove that $R_{\ell}$ does not have a bounded right-inverse for any $1\leq \ell\leq 2n.$

\begin{lem}\label{independent}
 Let $\{a_k\}_{k=1}^{2n}$ be the operators defined in \eqref{defofak}. Then for any $1\leq \ell\leq 2n$, the sequence
$$\{1,R_\ell a_1,R_\ell a_2,\cdots,R_\ell a_{2n}\}$$
is linearly independent.
\end{lem}
\begin{proof} Fix $1\leq \ell\leq 2n$ and let
$$c_0+\sum_{j=1}^{2n}c_jR_\ell a_j=0.$$
If $c_0\neq0,$ then $R_\ell$ admits bounded right inverse, which is not the case. Hence, $c_0=0.$ Thus,
$$\sum_{j=1}^{2n}c_jR_\ell a_j=0.$$
Since the right support of $R_\ell$ equals $1,$ it follows that
\begin{align}\label{equalzero}
\sum_{j=1}^{2n}c_ja_j=0.
\end{align}

It can be verified directly from the Lemma \ref{rotation} that for any $\theta\in [0,2\pi]$ and $k\in\{1,2,\cdots,n\}$, $U_{k,\theta}$ commutes with $-\Delta_\HH$ and $X_j$, where $j\notin \{k,k+n\}$, and therefore commutes with $(-\Delta_\HH)^{-\frac14}X_j(-\Delta_\HH)^{-\frac14}$. Thus, $U_{k,\theta}$ commutes with $a_j,$ where $j\notin \{k,k+n\}.$ From Lemma \ref{rotation} and the definition of double operator integral we see that
\begin{align*}
U_{k,\theta}^{-1}a_kU_{k,\theta}=T^{-\Delta_{\HH},-\Delta_{\HH}}_{\frac{2\alpha_0^{1/4}\alpha_1^{1/4}}{\alpha_0^{1/2}+\alpha_1^{1/2}}}\Big((-\Delta_{\HH})^{-1/4}U_{k,\theta}^{-1}X_kU_{k,\theta}(-\Delta_{\HH})^{-1/4}\Big)=\cos\theta a_k+\sin\theta a_{k+n},
\end{align*}
\begin{align*}
U_{k,\theta}^{-1}a_{k+n}U_{k,\theta}=T^{-\Delta_{\HH},-\Delta_{\HH}}_{\frac{2\alpha_0^{1/4}\alpha_1^{1/4}}{\alpha_0^{1/2}+\alpha_1^{1/2}}}\Big((-\Delta_{\HH})^{-1/4}U_{k,\theta}^{-1}Y_kU_{k,\theta}(-\Delta_{\HH})^{-1/4}\Big)=-\sin\theta a_k+\cos\theta a_{k+n}.
\end{align*}
Therefore,
\begin{align}\label{firstequation}
U_{k,\theta}^{-1}\cdot\bigg(\sum_{j=1}^{2n}c_ja_j\bigg)\cdot U_{k,\theta}=\sum_{\substack{1\leq j\leq 2n\\ j\neq k,k+n}}c_ja_j+c_k(\cos\theta a_k+\sin\theta a_{k+n})+c_{n+k}(-\sin\theta a_k+\cos\theta a_{k+n}).
\end{align}
On the other hand, by equation \eqref{equalzero},
\begin{align}\label{secondequation}
U_{k,\theta}^{-1}\cdot\bigg(\sum_{j=1}^{2n}c_ja_j\bigg)\cdot U_{k,\theta}=0.
\end{align}
Combining equations \eqref{firstequation} and \eqref{secondequation} with the fact that the functions $1$, $\cos\theta$, $\sin\theta$ are linearly independent,  we see that the coefficients in front of $1,\cos\theta,\sin\theta$ are zeroes. Thus, we have
\begin{equation}\label{linear_relations_for_a_k}
    \sum_{\substack{1\leq j\leq 2n\\ l\neq k,k+n}}c_ja_j=0,\quad c_ka_k+c_{n+k}a_{k+n}=0,\quad c_ka_{k+n}-c_{n+k}a_k=0.
\end{equation}

We will show the crucial fact that each $\{a_k\}_{k=1}^{2n}$ belongs to the image of $\pi_{\red}.$ To see this, it suffices to verify that $\sigma_r a_k\sigma_{r^{-1}} = a_k$ for all $r>0.$
By definition, we have
\[
    a_k = T^{-\Delta_{\HH},-\Delta_{\HH}}_{\psi}((-\Delta_{\HH})^{-1/4}X_k(-\Delta_{\HH})^{-1/4}).
\]
Since
\[
    \sigma_r (-\Delta_{\HH}) \sigma_{r^{-1}} = r^{-2}(-\Delta_{\HH}), \quad r > 0
\]
and $\sigma_r X_k\sigma_{r^{-1}} = r^{-1}X_k,$
it follows that for all $r>0$ we have
\begin{align*}
    \sigma_r a_k\sigma_{r^{-1}} &= T^{-r^{-2}\Delta_{\HH},-r^{-2}\Delta_{\HH}}_{\psi}(\sigma_r(-\Delta_{\HH})^{-1/4}X_k(-\Delta_{\HH})^{-1/4}\sigma_{r^{-1}})\\
                                &= T^{-r^{-2}\Delta_{\HH},-r^{-2}\Delta_{\HH}}_{\psi}((-\Delta_{\HH})^{-1/4}X_k(-\Delta_{\HH})^{-1/4}).
\end{align*}
Since the function $\psi$ is dilation invariant, we have
\[
    T^{-r^{-2}\Delta_{\HH},-r^{-2}\Delta_{\HH}}_{\psi} = T^{-\Delta_{\HH},-\Delta_{\HH}}_{\psi},\quad r > 0
\]
and hence for every $r>0$
\[
    \sigma_ra_k\sigma_{r^{-1}} = a_k,\quad 1\leq k\leq 2n.
\]
It follows that $a_k\in \pi_{\red}(\mathcal{B}(L_2(\mathbb{R}^n))\otimes \mathbb{C}^2).$

In fact, we have
\begin{equation}\label{pi_red_description_of_a_k}
    a_k \in \begin{cases} \pi_{\red}(\mathcal{B}(L_2(\mathbb{R}^n))\otimes {\bf 1}),\quad 1\leq k\leq n,\\
                          \pi_{\red}(\mathcal{B}(L_2(\mathbb{R}^n))\otimes z),\quad n+1\leq k\leq 2n.
            \end{cases}
\end{equation}
That is, $a_k=\pi_{{\rm red}}(u_k\otimes {\bf 1})$ and $a_{n+k}=\pi_{{\rm red}}(v_k\otimes z)$ for some $u_k,v_k\in \mathcal{B}(L_2(\mathbb{R}^n)).$
To see this, let $W$ denote the unitary linear map on $L_2(\HH^n)$ given by
\[
    W\xi([x+iy,t]) = \xi([x-iy,-t]),\quad [x+iy,t] \in \HH^n,\; \xi \in L_2(\HH^n).
\]
Then $WX_kW^* = X_k$ for $1\leq k\leq n$ and $WX_kW^* = -X_k$ for $n+1\leq k\leq 2n.$ Since $W\Delta_{\HH}W^* = \Delta_{\HH},$ we have
\[
    Wa_kW^* = \begin{cases}
                a_k,\quad 1\leq k\leq n,\\
                -a_k,\quad n+1\leq k\leq 2n.
              \end{cases}
\]
Given that $W\pi_{\red}(x\otimes \mathbf{1})W^* = \pi_{\red}(x\otimes \mathbf{1})$ and $W\pi_{\red}(x\otimes z)W^* = -\pi_{\red}(x\otimes z),$
we deduce \eqref{pi_red_description_of_a_k}.

Hence, for $1\leq k\leq n,$ $a_k = \pi_{{\rm red}}(u_k\otimes \mathbf{1})$ for some $u_k \in \mathcal{B}(L_2(\mathbb{R}^n))$ and for $n+1\leq k\leq 2n,$ $a_k =\pi_{{\rm red}} (v_k\otimes z)$ for some
$v_k \in \mathcal{B}(L_2(\mathbb{R}^n)).$ Since $\pi_{\red}$ is injective, it follows from \eqref{linear_relations_for_a_k} that
$$c_ku_k\otimes {\bf 1}+c_{n+k}v_k\otimes z=0,\quad c_kv_k\otimes z-c_{n+k}u_k\otimes {\bf 1}=0.$$
Thus, $c_ku_k\otimes {\bf 1}=0$ and $c_{n+k}v_k\otimes z=0.$ Thus, $c_ka_k=0$ and $c_{n+k}a_{n+k}=0.$ Since none of $a_j$'s is $0,$ it follows that $c_j=0$ for $1\leq j\leq 2n.$ This ends the proof of Lemma \ref{independent}.
\end{proof}

\begin{cor}\label{xkindependent}
Let $\{a_k\}_{k=1}^{2n}$ be the operators defined in \eqref{defofak} and for any $1\leq \ell\leq 2n$, $\{y_k^{\ell}\}_{k=0}^{2n}\subset L_{2n+2}(\mathcal{B}(L_2(\mathbb{R}^n))\overline{\otimes}\mathbb{C}^2,{\rm Tr}\otimes\Sigma)$ be the sequence satisfying
$$\pi_{\rm red}(y_k^{\ell})=\left\{\begin{array}{ll}(-\Delta_{\HH})^{-\frac{1}{2}}|T|^{\frac{1}{2}}, &k=0 \\(-\Delta_{\HH})^{-\frac{1}{2}}|T|^{\frac{1}{2}}(R_\ell a_k)^*, &k=1,2,\cdots,2n.\end{array}\right.$$
Then the sequence $\{y_k^{\ell}\}_{k=1}^{2n}$ is linearly independent.
\end{cor}
\begin{proof}
 To begin with, we suppose that
$\sum_{k=0}^{2n}c_k y_k^{\ell}=0,$
which implies that
$\sum_{k=0}^{2n}c_k \pi_{{\rm red}}(y_k^{\ell})=0.$ That is
$$c_0(-\Delta_\HH)^{-1/2}|T|^{-1/2}+\sum_{k=1}^{2n}c_k(-\Delta_\HH)^{-1/2}|T|^{-1/2}(R_\ell a_k)^*=0.$$
Therefore,
$$c_0+\sum_{k=1}^{2n}c_k(R_\ell a_k)^*=0.$$
By Lemma \ref{independent}, for any $1\leq \ell\leq 2n$, the sequence
$$\{1,R_\ell a_1,R_\ell a_2,\cdots,R_\ell a_{2n}\}$$
is linearly independent, which implies that $c_k=0$, $k=0,1,\cdots,2n$, and thus $\{y_k^{\ell}\}_{k=0}^{2n}$ is linearly independent.
\end{proof}

\begin{proof}[Proof of Proposition \ref{necessary theorem}]
%

We first show the conclusion holds for $f\in L_\infty(\Heis^n)\cap \dot{W}^{1,2n+2}(\HH^n).$ By Proposition \ref{keytraceformula}, for every normalised continuous trace $\varphi$ on $\mathcal{L}_{1,\infty}$, we have $$\varphi(|[R_\ell,M_f]|^{2n+2})=c_n\bigg\|\sum_{k=0}^{2n}f_k^{\ell}\otimes y_k^{\ell}\bigg\|_{L_{2n+2}(L_\infty(\mathbb{H}^n)\overline{\otimes} \mathcal{B}(L_2(\mathbb{R}^n))\overline{\otimes} \mathbb{C}^2)}^{2n+2},$$
where $\{f_k^{\ell}\}_{k=0}^{2n}\in L_{2n+2}(\mathbb{H}^n)$ and $\{y_k^{\ell}\}_{k=0}^{2n}\subset L_{2n+2}(\mathcal{B}(L_2(\mathbb{R}^n))\overline{\otimes}\mathbb{C}^2,{\rm Tr}\otimes\Sigma)$ are such that
$$ f_k^{\ell}=\left\{\begin{array}{ll}\overline{X_\ell f}, &k=0\\ \overline{X_kf}, &k=1,2,\cdots,2n\end{array}\right.\ {\rm and}\ \ \pi_{\rm red}(y_k^{\ell})=\left\{\begin{array}{ll}(-\Delta_{\HH})^{-\frac{1}{2}}|T|^{\frac{1}{2}}, &k=0 \\(-\Delta_{\HH})^{-\frac{1}{2}}|T|^{\frac{1}{2}}(R_\ell a_k)^*, &k=1,2,\cdots,2n.\end{array}\right.$$
Therefore,
\begin{align}\label{almostfinal}
\big\|[R_\ell,M_f]\big\|_{\mathcal{L}_{2n+2,\infty}}^{2n+2}=\big\||[R_\ell,M_f]|^{2n+2}\big\|_{\mathcal{L}_{1,\infty}}\gtrsim \bigg\|\sum_{k=0}^{2n}f_k^{\ell}\otimes y_k^{\ell}\bigg\|_{L_{2n+2}(L_\infty(\mathbb{H}^n)\overline{\otimes} \mathcal{B}(L_2(\mathbb{R}^n))\overline{\otimes} \mathbb{C}^2)}^{2n+2}.
\end{align}
Recall from Corollary \ref{xkindependent} that $\{y_k^{\ell}\}_{k=1}^{2n}$ is linearly independent. Combining this with Lemma \ref{lowerbd} conclude that
\begin{align*}
\big\|[R_\ell,M_f]\big\|_{\mathcal{L}_{2n+2,\infty}}\gtrsim \bigg\|\sum_{k=0}^{2n}f_k^{\ell}\otimes y_k^{\ell}\bigg\|_{L_{2n+2}(L_\infty(\mathbb{H}^n)\overline{\otimes} \mathcal{B}(L_2(\mathbb{R}^n))\overline{\otimes} \mathbb{C}^2)}\gtrsim \sum_{j=1}^{2n}\|X_{j}f\|_{L_{2n+2}(\Heis^{n})}.
\end{align*}
Next, we applied an approximation argument to show that under the assumption of Proposition \ref{necessary theorem}, one has  $f \in \dot{W}^{1,2n+2}(\HH^n)$, and therefore the priori assumption $f \in \dot{W}^{1,2n+2}(\HH^n)$ can be removed. To this end, we suppose that $f\in L^\infty(\HH^n)$ and consider a sequence of $C_c^\infty(\HH^n)$ functions defined by $f_m=(f\ast \phi_m)\eta_m$, where $\phi_m$ and $\eta_m$ are defined in the proof of Lemma \ref{predensity} and \ref{density2}, respectively. By the previous discussion, one has
\begin{align*}
\|f_m\|_{\dot{W}^{1,2n+2}(\HH^n)}\lesssim \|[R_{\ell},M_{f_m}]\|_{\mathcal{L}_{2n+2,\infty}}.
\end{align*}
We claim that
\begin{align}\label{ccclaim}
\sup\limits_{m\geq 1}\|[R_{\ell},M_{f_m}]\|_{\mathcal{L}_{2n+2,\infty}}\lesssim \|f\|_{L_\infty(\HH^n)}+\|[R_\ell,M_f]\|_{\mathcal{L}_{2n+2,\infty}}.
\end{align}
To show this, we first apply Leibniz's rule to see that
$$[R_{\ell},M_{f_m}]=M_{f\ast \phi_m}[R_{\ell},M_{\eta_m}]+[R_{\ell},M_{f\ast \phi_m}]M_{\eta_m}.$$
To deal with the first term, we apply Proposition \ref{sufficiency theorem} to see that
$$\|M_{f\ast \phi_m}[R_{\ell},M_{\eta_m}]\|_{\mathcal{L}_{2n+2,\infty}}\leq \|M_{f\ast \phi_m}\|_{\mathcal{L}_{\infty}}\|[R_\ell,M_{\eta_m}]\|_{\mathcal{L}_{2n+2,\infty}}\lesssim \|f\ast\phi_m\|_{L_\infty(\HH^n)}\lesssim \|f\|_{L_\infty(\HH^n)}.$$
For the second term, applying \cite[Lemma 18]{LMSZ}, one has
$$\|[R_{\ell},M_{f\ast \phi_m}]M_{\eta_m}\|_{\mathcal{L}_{2n+2,\infty}}\lesssim \|[R_{\ell},M_{f\ast \phi_m}]\|_{\mathcal{L}_{2n+2,\infty}}\|M_{\eta_m}\|_{\mathcal{L}_\infty}\lesssim \|[R_{\ell},M_f]\|_{\mathcal{L}_{2n+2,\infty}}.$$
Thus, our claim \eqref{ccclaim} holds. In particular, one has
\begin{align}\label{uniformlybd}
\sup\limits_{m\geq 1}\|f_m\|_{\dot{W}^{1,2n+2}(\HH^n)}<+\infty.
\end{align}
Next we show that the condition $f\in L^\infty(\HH^n)$ and inequality \eqref{uniformlybd} imply that $f\in \dot{W}^{1,2n+2}(\HH^n)$. Indeed, by the reflexivity of $L_{2n+2}(\HH^n,\mathbb{C}^{2n})$, there is a subsequence of $\{\nabla f_m\}_{m=1}^\infty$ converging weakly to some $g\in L_{2n+2}(\HH^n,\mathbb{C}^{2n})$. Note that
\begin{align*}
\langle f_m,\psi \rangle=\langle f\ast\phi_m,\eta_m \psi\rangle=\langle f,(\eta_m \psi)\ast\phi_m\rangle,\quad {\rm for}\ {\rm any}\ \psi\in \mathcal{S}(\HH^n).
\end{align*}
It is easy to verify that $(\eta_m \psi)\ast\phi_m\rightarrow \psi$ in $L_1(\HH^n)$. Thus,
\begin{align*}
|\langle f_m,\psi \rangle-\langle f,\psi \rangle|=|\langle f,(\eta_m \psi)\ast\phi_m-\psi\rangle|\leq \|f\|_{L_\infty(\HH^n)}\|(\eta_m \psi)\ast\phi_m-\psi\|_{L_1(\HH^n)}\rightarrow 0,
\end{align*}
as $m\rightarrow +\infty$. In other words, the sequence $\{f_m\}_{m=1}^\infty$ converges to $f$ in $(\mathcal S(\HH^n))'$. Consequently, the sequence $\{\nabla f_m\}_{m=1}^\infty$ converges to $\nabla f$ in $(\mathcal S(\HH^n))'$. Since the weak convergence in $L_{2n+2}(\HH^n,\mathbb{C}^{2n})$ implies the convergence in $(\mathcal S(\HH^n))'$ (see for example \cite[Chapter 1]{FoSt}), it follows that $g=\nabla f$ and so $\nabla f\in L_{2n+2}(\HH^n,\mathbb{C}^{2n})$. Equivalently, $f\in  \dot{W}^{1,2n+2}(\HH^n)$.

This completes the proof of Proposition \ref{necessary theorem}.
\end{proof}


\end{document}